\documentclass[11pt]{article}

\usepackage{amsfonts}
\usepackage{amssymb}
\usepackage{amsmath}
\usepackage{amsthm}
\usepackage{latexsym}
\usepackage{verbatim}
\usepackage{epsfig}
\usepackage{amsbsy}
\usepackage{amscd}
\usepackage{bm}

\usepackage{graphicx}

\usepackage{ifthen} 
\usepackage{xcolor}
\usepackage[colorlinks=true, linkcolor=teal, anchorcolor=teal,
            citecolor=blue,  filecolor=blue, menucolor=blue, pagecolor=teal,											urlcolor=blue, plainpages=false, pdfpagelabels]{hyperref}

\newcommand{\mymail}[1]{\href{mailto:#1}{\texttt{#1}}}

\usepackage{fancyhdr}
\usepackage[realmainfile]{currfile}
\fancypagestyle{firststyle}{

\cfoot{{\footnotesize $^\ast$: Supported by NSF grant CCF-1849876 $\qquad \qquad$ $^\dagger$: Supported by NSF grant DMS-2015376}} 
}

\newcommand{\setauthA}[1]{\def\authA{#1}}
\newcommand{\setauthB}[1]{\def\authB{#1}}

\def\printA{\begin{tabular}{l} \authA \end{tabular}}
\def\printB{\begin{tabular}{l} \authB \end{tabular}}

\newcommand{\makemytitle}[1]{\begin{center}{\textsf{\LARGE #1}}
  \end{center}
}

\usepackage[sf,bf]{titlesec}

\usepackage{algorithmic}
\usepackage{algorithm}

\usepackage[nonamebreak,square,numbers,sort]{natbib}

\usepackage[letterpaper, textheight = 676pt]{geometry}
\providecommand{\M}[1]{\mathbf#1}
\providecommand{\mc}[1]{\mathcal#1}
\providecommand{\mc}[1]{\mathcal#1}

\newcommand{\R}{{\mathbb R}}

\DeclareMathOperator{\E}{\mathbf{E}}
\DeclareMathOperator{\p}{\mathbf{P}}

\DeclareMathOperator{\cov}{Cov}

\DeclareMathOperator{\tr}{tr}

\providecommand{\T}{\top} 

\providecommand{\wt}[1]{\widetilde{#1}}
\providecommand{\wh}[1]{\widehat{#1}}

\providecommand{\nnorm}[1]{ \lVert#1 \rVert}
\newcommand{\scp}[2]{\left\langle#1, #2\right\rangle}
\newcommand{\nscp}[2]{\langle#1, #2\rangle}

\newcommand{\blanco}[1]{  }

\newcommand{\deriv}[3]{%
\ifthenelse{#1 = 1}{\frac{d\,#2}{d\,#3}}{\frac{d^{{#1}} #2}{d{#3}^{{#1}}}}
}

\newcommand{\partials}[3]{%
\ifthenelse{#1 = 1}{\frac{\partial\,#2}{\partial\,#3}}{\frac{\partial^{#1}
    #2}{\partial#3^{#1}}}
} 

\def\su{\sum_{i=1}^n}

\def \coloneq{\mathrel{\mathop:}=}
\def \invcoloneq{=\mathrel{\mathop:}}
\def \eps{\varepsilon}

\def \lec{\preceq}

\newtheorem{theo}{Theorem}
\newtheorem{propo}{Theorem}
\newtheorem{definitio}{Theorem}

\newtheorem{rema}{Theorem}

\newtheorem{lemmachen}{Theorem}

\newtheorem{theorem}[theo]{Theorem}
\newtheorem{defn}[definitio]{Definition}

\newtheorem{prop}[propo]{Proposition}

 \newtheorem{lemma}[lemmachen]{Lemma}

\newtheorem{rem}[rema]{Remark}

\newenvironment{bew}{\begin{proof}[Proof]}{\end{proof}}

\def\R{\mathbb{R}}
\def\tr{\mathrm{tr}}

\def\eps{\epsilon}

\def\dH{\textsf{H}} 
\def\dtv{\textsf{TV}} 
\def\dW{\textsf{W}} 
\def\mphi{\varphi} 

\usepackage{subfigure}

\newboolean{isjournal}
\setboolean{isjournal}{true}

\newboolean{nocolors}
\setboolean{nocolors}{false}

\newboolean{usemtpro}
\setboolean{usemtpro}{false}

\ifthenelse{\boolean{nocolors}}{\hypersetup{colorlinks=false}}{}

\makeatletter
\newcommand\footnoteref[1]{\protected@xdef\@thefnmark{\ref{#1}}\@footnotemark}
\makeatother

\setauthB{{\bfseries Bodhisattva Sen$^{2\dagger}$}}

\setauthA{{\bfseries Martin Slawski}$^{1*}$}

\begin{document}
\thispagestyle{firststyle}

\makemytitle{{\Large {\bfseries {Permuted and Unlinked Monotone Regression in $\R^d$: an approach based on mixture modeling and optimal transport\ }}}}
\vskip 3.5ex
%
{\large\begin{center}
\printA
\printB
\vskip1.5ex
{\scriptsize $^{1}$Department of Statistics, George Mason University, Fairfax, VA 22030, USA $\; \; \;$}\\[.5ex] 
\scriptsize{$^{2}$Department of Statistics, Columbia University, New York, NY 10027, USA}  \\[.5ex]
\mymail{mslawsk3@gmu.edu} \hspace{0.75in} \mymail{bodhi@stat.columbia.edu}
\end{center}}
\vskip 3.5ex

\begin{abstract} 
Suppose that we have a regression problem with response variable $Y \in \R^d$ and predictor $X \in \R^d$, for $d \ge 1$. In {\it permuted or unlinked regression} we have access to {\it separate unordered} data on $X$ and $Y$, as opposed to data on $(X,Y)$-pairs in usual regression. So far in the literature the case $d=1$ has received attention, see e.g.,~the recent papers by Rigollet and Weed [\emph{Information \& Inference}, {\bfseries 8}, 619--717] and Balabdaoui \emph{et al}. [\emph{J.~Mach.~Learn.~Res.}, {\bfseries 22}(172), 1--60]. In this paper, we consider the general multivariate setting with $d \geq 1$. We show that the notion of {\it cyclical monotonicity} of the regression function is sufficient for identification and estimation in the permuted/unlinked regression model. We study permutation recovery in the permuted regression setting and develop a computationally efficient and easy-to-use algorithm for denoising based on the Kiefer-Wolfowitz [\emph{Ann.~Math.~Statist.}, {\bfseries 27}, 887--906] nonparametric maximum likelihood estimator and techniques from the theory of optimal transport. We provide explicit upper bounds on the associated mean squared denoising error for Gaussian noise. As in previous work on the case $d = 1$, the permuted/unlinked setting involves slow (logarithmic) rates of convergence rooting in the underlying deconvolution problem. Numerical studies 
corroborate our theoretical analysis and show that the proposed approach performs at least on par with the
methods in the aforementioned prior work in the case $d = 1$ while achieving substantial reductions in terms of computational complexity.
\end{abstract}

\section{Introduction}\label{sec:intro}
In their 1971 paper \cite{DeGroot1971} DeGroot \emph{et al}.~considered the following problem: given photographs of $n$ film stars and another set of photographs of the same film stars taken at
a younger age, can we identify corresponding pairs of photographs (i.e., belonging to the same film star) 
based on, e.g., $d$ biometric measurements extracted from each photograph? A specific variant of this problem 
(illustrated in Figure \ref{fig:movie_stars}) is studied in the present paper. Let $\mc{X}_n = \{ X_i \}_{i = 1}^n$ and $\mc{Y}_n = \{ Y_i \}_{i = 1}^n$ be given $\R^d$-valued ($d \geq 1$) samples of data (e.g., $\mc{X}_n$ denoting past photographs and $\mc{Y}_n$ recent photographs) pertaining to a common set of $n$ entities, and suppose that there is a function $f^*: \R^d \rightarrow \R^d$ transforming data in
$\mc{X}_n$ to their matching counterparts in $\mc{Y}_n$, modulo additive noise, i.e., for some \emph{unknown} permutation $\pi^*$ of $\{1,\ldots,n\}$, we have that 
\begin{equation}\label{eq:permuted_regression}
Y_i = f^*(X_{\pi^*(i)}) + \eps_i, \quad 1 \leq i \leq n, 
\end{equation}
where the $\{ \eps_i \}_{i = 1}^n$ represent i.i.d.~zero-mean additive noise. Note that if $\pi^*$ was known, 
the problem boils down to a standard regression / (non-parametric) function estimation setup. On the other hand, if $f^*$ was known, the problem boils down to a standard matching problem \cite{Burkard2009, Collier2016}. In this paper, both $f^*$ and $\pi^*$ are assumed to be unknown, and the following tasks
are considered:\\[1ex] 
({\bfseries T1}): (Exact) \emph{Permutation recovery}, i.e., inferring the permutation $\pi^*$ without error,\\[1ex]
({\bfseries T2}): \emph{Denoising}, i.e., the construction of estimators $\{\wh{f}(X_i) \}_{i = 1}^n$ for 
$\{f^*(X_i) \}_{i = 1}^n$. 
\vskip1.5ex
\noindent Task ({\bfseries T2}) will also be studied in a slightly more general setup in which samples of different
size, say, $\mc{X}_n$ and $\mc{Y}_{m}$ are observed such that samples in the latter are i.i.d.~copies of $Y \overset{\mc{D}}{=} f^*(X) + \eps$ and samples in $\mc{X}_n$ are i.i.d.~copies of $X \sim \mu$ for some suitable probability measure $\mu$ on $\R^d$, with $\overset{\mc{D}}{=}$ denoting equality in distribution. Adopting the terminology in \cite{Balabdoui2020}, this generalized setup will be referred to as \emph{unlinked regression}, whereas the basic setup \eqref{eq:permuted_regression} will be referred to as \emph{permuted regression}. In the latter case, $\{ X_i \}_{i = 1}^n$ will be considered as fixed, unless stated otherwise.
\vskip1.5ex
\noindent {\bfseries Applications}. The problem outlined above arises in a series of applications in various domains. In computer vision, a common task is to identify corresponding pairs of images, with one image arising as a distorted image of the other \cite{Hartley2004}; in this context, the function $f^*$ may represent a specific combination of distortions (e.g., scaling, rotations, blur, etc.). Specific instances of \eqref{eq:permuted_regression} that have received considerable attention lately are
\emph{unlabeled sensing} or \emph{linear regression with unknown permutation}, e.g., \cite{Unnikrishnan2015, Pananjady2016, Pananjady2017, Abid2017, Hsu2017, SlawskiBenDavid2017, Tsakiris2018, TsakirisICML19, SlawskiBenDavidLi2019, ZhangSlawskiLi2021, SlawskiDiao2019} in which case $f^*$ is an affine transformation (albeit not necessarily from 
$\R^d$ to $\R^d$). Among these works, the papers \cite{SlawskiBenDavid2017, SlawskiDiao2019} discuss applications in record
linkage \cite{Herzog2007, Christen2012, Winkler14}, specifically post-linkage data analysis \cite{Scheuren93, Scheuren97, Lahiri05}. The papers \cite{Grave2019, Shi2018} consider the case in which $\mc{X}_n$ and $\mc{Y}_n$ are points in the unit sphere
in $\R^d$ and $f^*$ is a unitary map with applications in automated translation between different word embeddings. As elaborated 
in more detail in $\S$\ref{subsec:permutation_recovery} below, the setup \eqref{eq:permuted_regression} also arises in matrix estimation problems,
in which a noisy row-permuted version of a matrix, whose columns exhibit the same ordering pattern (decreasing or increasing), is observed. The papers \cite{Flammarion16, Ma2020, Ma2021} discuss applications in statistical seriation \cite{Liiv2010} and microbiome data analysis. Finally, model \eqref{eq:permuted_regression} bears a relation to linkage attacks in the literature on data privacy \cite{Sweeney2001, Narayanan2008}: here, $\mc{Y}_n$ may represent (anonymized) sensitive data while an adversary holds auxiliary data $\mc{X}_n$ along with identifiers (e.g., individuals' names) and tries to leverage the functional relationship between the two data sets to guess the values
of the sensitive attributes contained in $\mc{Y}_n$ for each or a subset of the identifiers.
\begin{figure}
\begin{equation*}
\begin{array}{ccccccccc}
X_1    &   X_2    &  \ldots   & X_n & \hspace*{8ex} & Y_1    &  Y_2        &  \ldots   & Y_n  \\  
\includegraphics[height = 0.05\textheight]{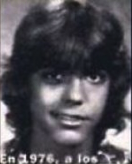}       & \includegraphics[height = 0.05\textheight]{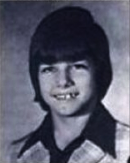}         &   & \includegraphics[height = 0.05\textheight]{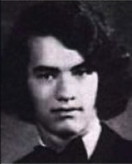}  &   & \includegraphics[height = 0.05\textheight]{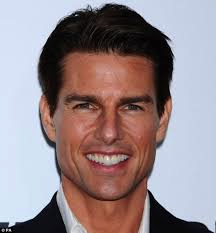}       &  \includegraphics[height = 0.05\textheight]{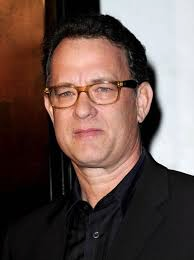}        &  \ldots   & \includegraphics[height = 0.05\textheight]{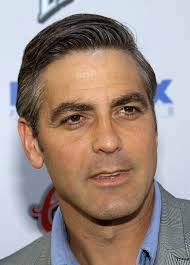}  
\end{array}
\end{equation*}
\vspace*{-2ex}
\caption{Illustration of the film stars correspondence problem described in DeGroot \emph{et al}.~\cite{DeGroot1971}. In terms of model \eqref{eq:permuted_regression}, one can potentially think of $\{ Y_1, \ldots, Y_n \}$ as the image of 
$\{ X_1, \ldots, X_n \}$ under some ``morphing" function $f^*$, modulo unstructured noise.}\label{fig:movie_stars}
\end{figure}

\vskip1.5ex
\noindent {\bfseries{Summary of contributions and related work}.} In a nutshell, the current paper can be seen as an extension
of the setup in the papers \cite{Carpentier2016, Weed2018, Balabdoui2020} which consider (variants of) \eqref{eq:permuted_regression}
with $d = 1$ and $f^*$ monotone with known direction of monotonicity (say, non-decreasing). A fundamental question associated with 
\eqref{eq:permuted_regression} asks for what class of functions $f^*$ it is possible to perform tasks ({\bfseries T1}) and 
({\bfseries T2}) in a statistically consistent manner. In fact, even in the absence of noise and the additional requirement that $f^*$ be smooth, ({\bfseries T1}) is generally hopeless already for $d = 1$ as can be seen from a simple example (cf.~$\S$\ref{sec:problem_approach}). 

In this paper, we establish that ({\bfseries T1}) and ({\bfseries T2}) can be accomplished  
if $$f^* = \nabla \psi_{f^*}$$ where $\psi_{f^*}: \R^d \rightarrow \R$ is a strictly convex function. Such functions $f^*$ provide a natural generalization of increasing functions for $d = 1$ in view of the property that 
\begin{equation*}
\nscp{\nabla \psi_{f^*}(y) - \nabla \psi_{f^*}(x)}{y - x} > 0 \quad  \text{for all} \; x,y \in \R^d. 
\end{equation*}
Note that in particular, functions of the form $f^* = (f_1^*, \ldots, f_d^*)$ with 
$f_j^*$ increasing on $\R$, $1 \leq j \leq d$, as studied in \cite{Flammarion16, Ma2020, Ma2021} are included, corresponding to 
component-wise separable additive (strictly) convex functions of the form 
\begin{equation*}
\psi_{f^*}(x_1, \ldots, x_d) = \textstyle\sum_{j = 1}^d \psi_{f_j^*}(x_j). 
\end{equation*}
Permutation recovery in the presence of noise based on the solution of a linear assignment problem associated with $\mc{X}_n$ and $\mc{Y}_n$ is shown to succeed if a certain minimum signal condition similar to conditions in related papers \cite{Flammarion16, Ma2020, Ma2021, ZhangSlawskiLi2021} is met. As a byproduct, the result on permutation recovery herein yields the novel insight that the unlabeled sensing problem in \cite{ZhangSlawskiLi2021} can be solved efficiently whenever the unknown linear transformation is positive (semi)-definite.   

Regarding the task ({\bfseries T2}) of denoising, we leverage a connection to the Brenier theorem in optimal transportation, e.g., \cite{COT2019, Villani2009, Villani2003, Santambrogio2015}. According to this connection, the sample $\mc{Y}_n$ is thought of as the image
of $\mc{X}_n$ under an optimal transport map $f^* = \nabla \psi_{f^*}$, contaminated by additive noise. Denoising is achieved 
via deconvolution of the measure $\frac{1}{n} \su \delta_{Y_i}$ and subsequent computation of an optimal coupling $\wh{\gamma}$ between the deconvolution estimate and the measure $\frac{1}{n} \su \delta_{X_i}$; finally, we take $\{ \wh{f}(X_i) \}_{i = 1}^n$ as the so-called barycentric projection of $\wh{\gamma}$. Deconvolution is based on the Kiefer-Wolfowitz
NPMLE for location mixtures \cite{Kiefer1956consistency, Koenker2014} and requires knowledge of the noise distribution. The approach developed herein is free of tuning parameters, and directly generalizes to the unlinked regression setting  with samples $\mc{X}_n$ and $\mc{Y}_m$ of different size described above at the end of the first paragraph. We provide upper bounds on 
the mean-square denoising error $\frac{1}{n} \su \nnorm{f^*(X_i) - \wh{f}(X_i)}_2^2$ in terms of the Hellinger distance of 
the Kiefer-Wolfowitz NPMLE to the underlying location mixture generating $\mc{Y}_n$ and the rate of decay of the noise distribution, the latter being a common ingredient in deconvolution problems. For Gaussian errors, all quantities can be made explicit, yielding rather slow rates of convergence in alignment with prior work \cite{Carpentier2016, Weed2018, Balabdoui2020} on the case $d = 1$.

The main innovations of the present work over \cite{Carpentier2016, Weed2018, Balabdoui2020} is the generalization to arbitrary dimension $d$, whereas \cite{Carpentier2016, Weed2018, Balabdoui2020} only consider $d = 1$. All three works are based on
deconvolution, and a connection to optimal transportation, albeit for $d = 1$, is already made in \cite{Weed2018}. However, even for $d = 1$, we argue that the approach developed in this paper is computationally more appealing than those in \cite{Carpentier2016, Weed2018, Balabdoui2020}. The method in \cite{Carpentier2016} is based on the truncated characteristic function estimator originating in the deconvolution literature and hence entails a tuning parameter. The method in \cite{Weed2018} is tuning-free and based on convex optimization; however, their deconvolution procedure involves Wasserstein distance minimization and in turn a non-smooth optimization problem that is less straightforward to solve than the Kiefer-Wolfowitz NPMLE. The method in \cite{Balabdoui2020} is based on a non-convex optimization problem. 

The theoretical results presented in \cite{Carpentier2016, Weed2018, Balabdoui2020} are of different flavors, and hence not directly comparable. The paper \cite{Carpentier2016} does not provide explicit rates of convergence. The paper \cite{Balabdoui2020} is different from \cite{Weed2018} in the sense that the former emphasizes on the unlinked regression setting and provides rates
for function estimation in the $L_1$-distance, whereas \cite{Weed2018} studies the mean-squared denoising error in the permuted regression setting \eqref{eq:permuted_regression}. For $d = 1$, the denoising performance metric in \cite{Weed2018} (mean squared error at the $\{ X_i \}_{i = 1}^n$) coincides with what is considered in the present paper. The rate herein is slightly slower than the minimax rate shown in \cite{Weed2018}, but given that both rates decrease only logarithmically in $n$, the gap is not that pronounced. More detailed comparisons are postponed to later sections in this paper. Finally, we would like to mention the paper \cite{Meis2021} that studies the setting in \cite{Weed2018} under discrete errors. 

The approach taken in this paper and the techniques used for its analysis bear various connections to recent developments in the literature on optimal transport, e.g., on the estimation of (smooth) optimal transport maps \cite{GhosalSen2019, Hutter2021, DebSen2021, Manole2021, Chizat2020}. Key steps in our proofs are based on adaptations of parts of the analysis in \cite{Chizat2020, DebSen2021, Manole2021}. At a technical level, the main distinction of the present work compared to these earlier works 
is the convolution setting considered herein. 

\vskip1.5ex
\noindent {\bfseries{Paper outline}.} This paper is organized as follows. Section $\S$\ref{sec:problem_approach} provides
a more detailed discussion of the problem sketched in the introduction, and presents an overview of the technical approach taken. The theoretical properties of that approach are studied in $\S$\ref{sec:main_results} and corroborated with numerical results
in $\S$\ref{sec:numerical_results}. A conclusion is provided in $\S$\ref{sec:conclusion}. Proofs of our results and additional
technical details can be found in the Appendix. 
\vskip1.5ex
\noindent {\bfseries Notation}. For the convenience of the reader, notation that is used frequently in this paper is summarized in the following table. 
{\footnotesize
\begin{center}
\begin{tabular}{|ll|ll|}
\hline & & &\\[-1ex]
$\{ \theta_i^* \}_{i = 1}^n $ & unknown location parameters & $f^*$ & function of interest \\[1ex]
$\nu_n^*$ & measure $\frac{1}{n} \su \delta_{\theta_i^*}$ & $\psi_{f^*}$ & convex function associated with $f^*$\\[1ex]
$\nu_n$ & measure $\frac{1}{n} \su \delta_{Y_i}$ & $\psi_{f^*}^{\star}$ & Conjugate of $\psi_{f^*}$  \\[1ex]
$\mu_n$ & measure $\frac{1}{n} \su \delta_{X_i}$  & $\pi^*$ & ground truth permutation of $\{1,\ldots,n\}$\\[1ex] 
$\varphi$ &  PDF of $\epsilon/\sigma$, $\cov(\epsilon) = \sigma^2 I_d$ & $\Pi^*$ & permutation matrix corresponding to $\pi^*$  \\
$\varphi_{\sigma}$ & PDF of $\epsilon$  & $\mc{P}(n)$ & set of permutation matrices of order $n$ \\[1ex]
$\star$ & convolution & $\pi, \Pi$ & generic elements of $\mc{P}(n)$  \\[1ex]
$\textsf{f}_n = \varphi_{\sigma} \star \nu_n^*$ & average location mixture PDF  &  $\M{F}$ & Fourier transformation (operator) \\
$\wh{\textsf{f}}_n$ & NPMLE of $\textsf{f}_n$  & $\wh{g}$ & short for $\M{F}[g]$  \\
\hline
\end{tabular}
\end{center}}
\noindent We often refer to a permutation via the underlying map $\pi$ and the corresponding matrix $\Pi$ in an interchangeable fashion, and accordingly $\mc{P}(n)$ may refer to both maps and matrices. 


\section{Estimation strategy }\label{sec:problem_approach}
In this section we describe our estimation procedure for both tasks -- ({\bfseries T1}) and ({\bfseries T2}). We start with a simple example that illustrates the non-identifiability of $f^*$ and $\pi^*$ in~\eqref{eq:permuted_regression}  without further assumptions on the structure of $f^*$; see Remark~\ref{rem:Neg-Ex} below. It turns out that if $f^*$ is {\it cyclically monotone} (see Section~\ref{sec:Mono-Op} where we formally define this notion along with other related concepts) then model~\eqref{eq:permuted_regression} is identifiable and consistent estimation can be successfully carried out. To solve the denoising problem ({\bfseries T2}) we leverage ideas from the theory of optimal transport and the Kiefer-Wolfowitz NPMLE for location mixtures which is discussed in detail in Section~\ref{sec:Path-OT}. We also give our main algorithm (see Algorithm~\ref{alg:denoising}) and discuss the computational approach in Section~\ref{sec:Path-OT}.

\begin{rem}\label{rem:Neg-Ex}{\bf (A negative example)} To gain some insights into the feasibility of tasks ({\bfseries T1}) and ({\bfseries T2}) given the observation model \eqref{eq:permuted_regression}, let us first consider a simple example which shows that recovery of $f^*$ or $\pi^*$
is generally hopeless even in seemingly benign settings ($d = 1$, no noise, $f^*$ smooth). Specifically, suppose that
$X_i = Y_i = i/M$, $1 \leq i \leq n = M-1$ for $M \geq 2$. Then both pairs 
$f_1^*(x) = x$ with $\pi_1^*(i) = i$, $1 \leq i \leq n$, and $f_2^*(x) = 1- x$ with $\pi_2^*(i) = n - i$, $1 \leq i \leq n$, satisfy~\eqref{eq:permuted_regression}. Clearly, additionally requiring that $f^*$ be increasing rules out this ambiguity. In fact, 
estimation of monotone $f^*$ with known direction of monotonicity under the permuted regression setup \eqref{eq:permuted_regression} has been shown to be feasible even in the presence of noise \cite{Carpentier2016, Weed2018, Balabdoui2020}. At the same time, estimation
of the direction of monotonicity itself is generally not possible even if $f^*$ is linear \cite{DeGroot1980, Bai05}.
\end{rem}

\subsection{Monotone operators and linear assignment problems}\label{sec:Mono-Op} The example above for $d = 1$ (and in the absence of noise) provides some
useful clues regarding the generalization to arbitrary dimension $d \geq 1$. If $f^*$ is known to be increasing, the 
underlying permutation $\pi^*$ is immediately determined by the requirement that $Y_i$ must match the 
corresponding order statistic in $\mc{X}_n$, i.e., $X_{\pi^*(i)} = X_{\text{rank}(i)}$, where 
$\text{rank}(i)$ denotes the rank of $Y_i$ among $\mc{Y}_n$, $1 \leq i \leq n$. It can also be shown that 
$\pi^*$ minimizes the optimization problem 
\begin{equation}\label{eq:sorting}
\min_{\pi} \frac{1}{2} \su |Y_i - X_{\pi(i)}|_2^2 = -\max_{\pi} \su X_{\pi(i)} Y_i +c, \qquad \quad c \coloneq \frac{1}{2}\sum_{i = 1}^n (X_i^2 + Y_i^2)
\end{equation}
over all permutations $\pi$ of $\{1,\ldots,n\}$. The above problem is a specifically simple instance of the class of \emph{linear assignment problems} (LAPs) that are of the form
\begin{equation}\label{eq:LAP}
\min_{\Pi \in \mc{P}(n)} \sum_{i = 1}^n \sum_{j = 1}^n \Pi_{ij} C_{ij} = \min_{\Pi \in \mc{P}(n)}  \tr(C^{\T} \Pi),     
\end{equation}
where $\mc{P}(n) = \{ \Pi \in \{0,1 \}^{n \times n}: \, \sum_{i = 1}^n \Pi_{ij} = 1, \, 1 \leq j \leq n, \; \sum_{j = 1}^n \Pi_{ij} = 1, \, 1 \leq i \leq n \}$ denotes the set of permutation matrices of dimension $n$ and $C = (c_{ij})$ is a cost matrix with entry
$(i,j)$ representing the cost associated with the pairing $(i, j)$, $1 \leq i,j \leq n$. LAPs \eqref{eq:LAP} constitute a well-studied class of optimization problems that are known as bipartite matching problems in the literature on combinatorial optimization \cite{Burkard2009}. In light of the celebrated Birkhoff-von Neumann theorem \cite{Ziegler1995}, \eqref{eq:LAP} can be solved efficiently via linear programming. Tailored algorithms such as the Hungarian Algorithm \cite{Bertsekas1992} and the Auction Algorithm \cite{Kuhn1955} have runtime complexity $O(n^3)$, and approximate solutions can be obtained via Sinkhorn iterations in time $O(n^2 \log n)$ \cite{Cuturi2019}; especially simple instances such as \eqref{eq:sorting} in which $C$ has rank one reduce to sorting.

\begin{figure}[t]
    \includegraphics[width = 0.32\textwidth]{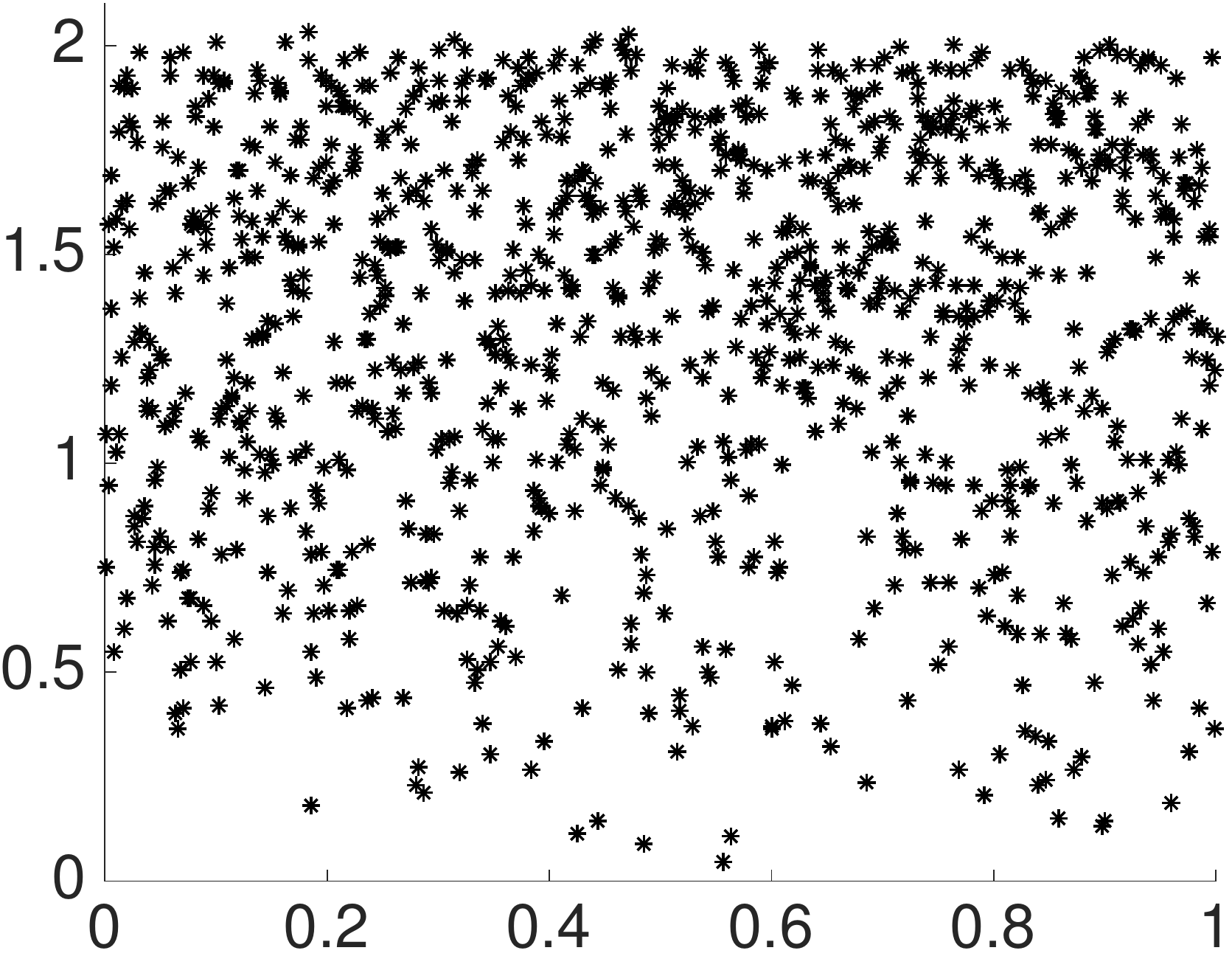} 
    \includegraphics[width = 0.32\textwidth]{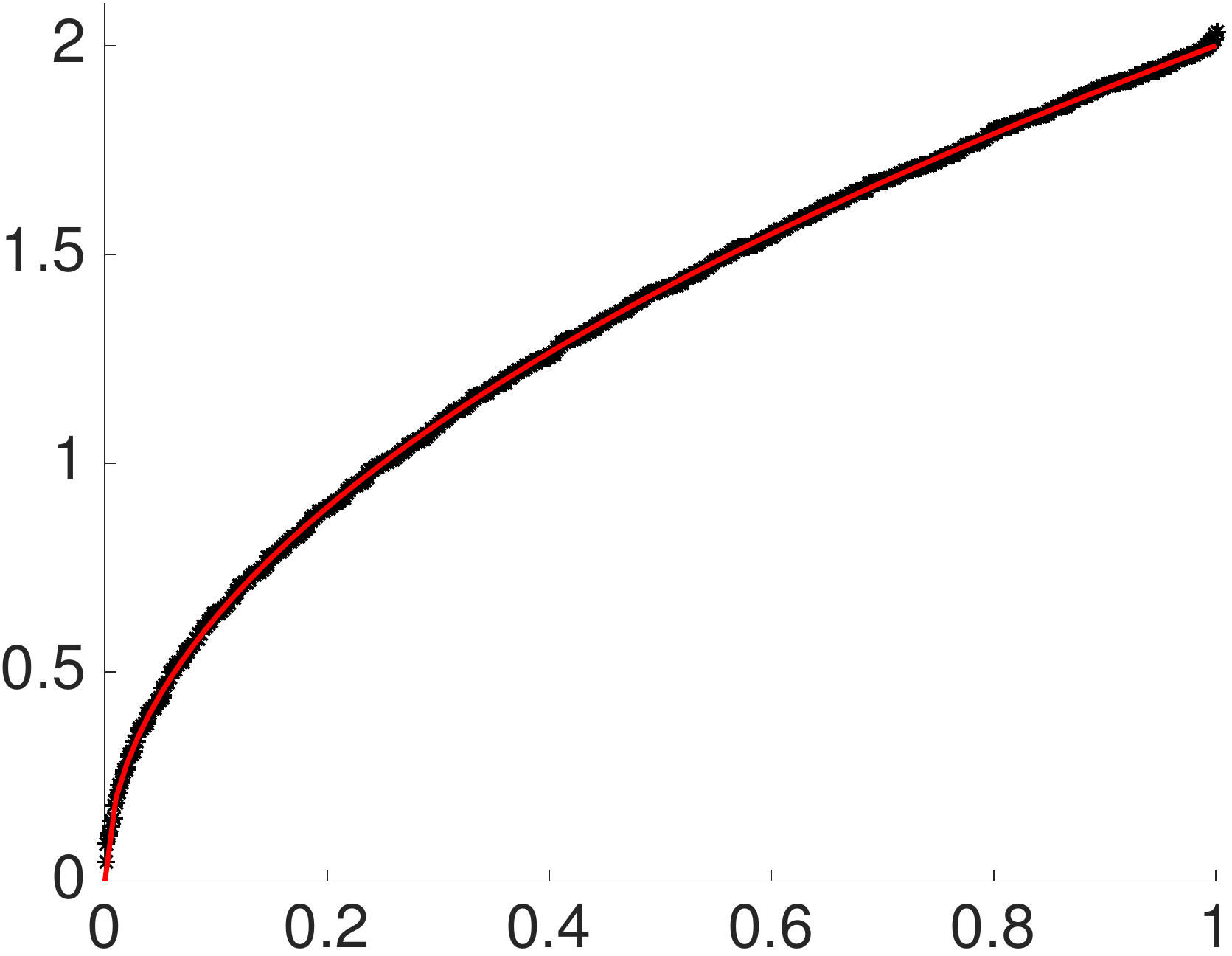} 
    \includegraphics[width = 0.32\textwidth]{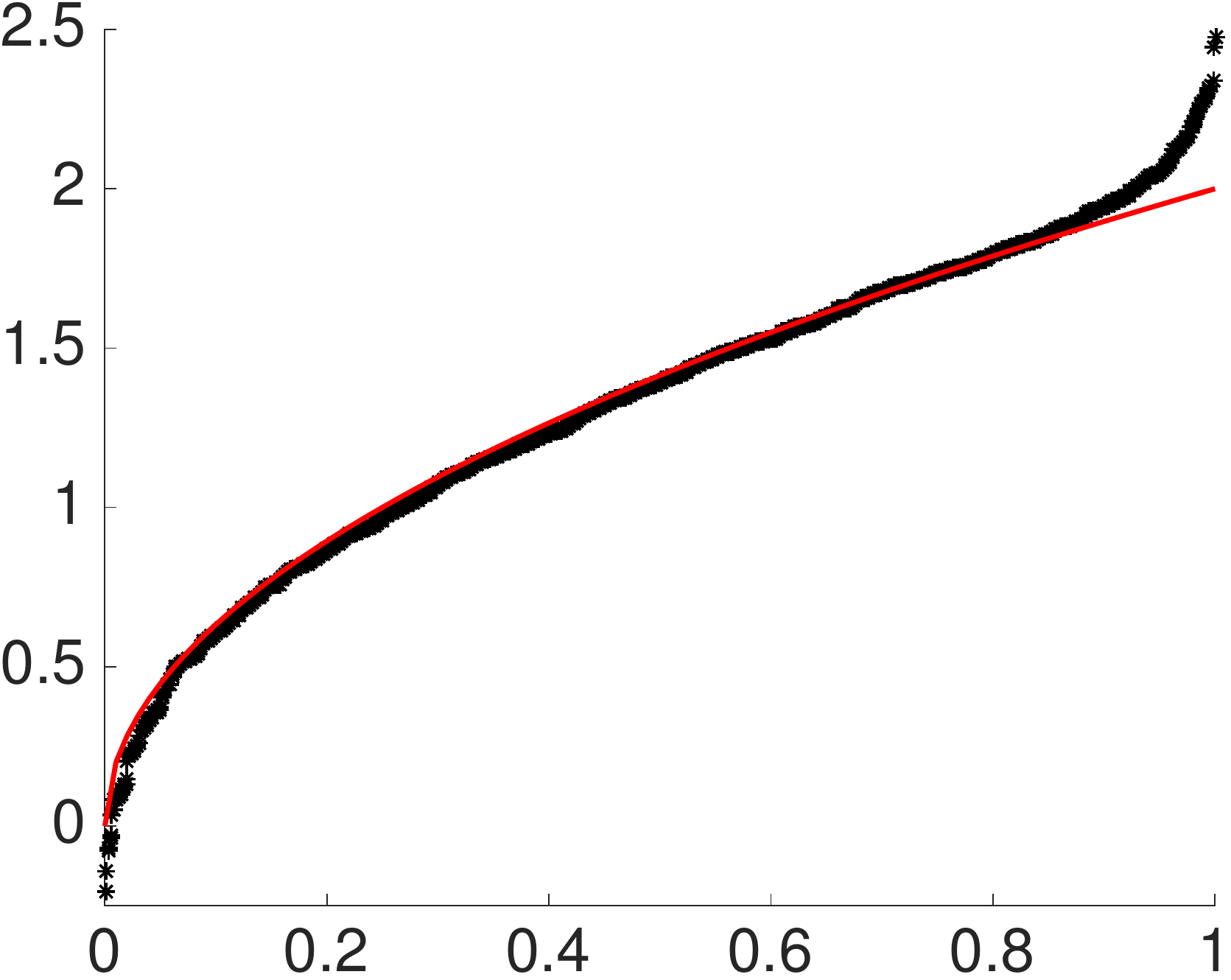}
    \vspace*{-1ex}
    \caption{Left: Shuffled Data $\{ (X_i, Y_i) \}_{i = 1}^n$. Middle: Sorted
    Data $(X_{(i)}, Y_{(i)})_{i = 1}^n$ in case of negligible noise; underlying function
    $x \mapsto f^*(x) \coloneq 2\sqrt{x}$ in red. Right: Sorted Data $(X_{(i)}, Y_{(i)})_{i = 1}^n$ in case
    of substantial noise. The results indicates a serious amount of bias, particularly near the boundaries.}
    \label{fig:simple-example}
\end{figure}

The crucial insight here is that knowing $f^*$ is monotone increasing immediately allows us to recover $\pi^*$ in the 
absence of noise via the optimization problem \eqref{eq:sorting}. This observation prompts the following generalization. 
Let $\mathbb{X}^n \subset \R^d \times \ldots \times \R^d$ be a domain containing all possible samples $\mc{X}_n$. We require that 
for all $\mc{X}_n \subset \mathbb{X}^n$ and all $n \geq 1$, $f^*$ has the property that 
\begin{equation}\label{eq:permutation_min_general}
\min_{\pi} \frac{1}{2} \su \nnorm{Y_i - X_{\pi(i)}}_2^2,
\end{equation}
is (uniquely) minimized by $\pi = \pi^*$, where $Y_i = f^*(X_{\pi^*(i)}), \; 1 \leq i \leq n$. This requirement can be 
expressed more succinctly via the notion of (strict) \emph{cyclical monotonicity}, a notion that arises in the study of monotone operators in convex analysis \cite{Bauschke2011} as well as in optimal transportation \cite[e.g.,][Definition 2.1]{McCann2011}, a connection that plays a fundamental role in the developments further below.
\begin{prop}\label{prop:cyc_monotonicity} Without loss of generality, suppose that $\pi^*$ equals the identity $\emph{\textsf{id}}$ permutation. The optimization problem \eqref{eq:permutation_min_general} is uniquely minimized by $\pi = \textsf{\emph{id}}$ iff $\Gamma_{f^*} := \{(x, f^*(x)): \, x \in \mathbb{X} \} \subset \R^{d} \times \R^d$ is a (strictly) cyclically monotone set (with respect to the Euclidean norm), i.e., if for all 
$k \geq 1$ and all $\{ (x_i, y_i) \}_{i = 1}^k \subset \Gamma_{f^*}$, it holds that
\begin{equation}\label{eq:cyclical_monotonicity}
-\sum_{i = 1}^k \nscp{x_i}{y_i} < \sum_{i = 1}^k -\nscp{x_{i+1}}{y_i}, \qquad x_{k+1} \coloneq x_1.
\end{equation}
\end{prop}
\noindent Proposition \ref{prop:cyc_monotonicity} is obtained by omitting the square terms in the objective as in \eqref{eq:sorting} and decomposing permutations into their disjoint cycles; a formal proof is omitted for the sake of brevity.
The following result, due to Rockafellar, precisely characterizes the class of functions $f: \R^d \rightarrow \R^d$ whose graphs are cyclically monotone.
\begin{theo}\label{theo:rockafellar}\emph{\cite{Rockafellar1966}}
The graph of the sub-differential $\partial \psi$ of a convex function $\psi: \R^d \rightarrow \R^d$, i.e., 
$\Gamma_{\partial \psi} :=  \{(x,y) \in  \R^d \times \R^d: \psi(z) \geq \psi(x) + \nscp{z-x}{y} \; \forall z \in \R^d \}$
is a cyclically monotone subset of $\R^d \times \R^d$. Moreover, any cyclically monotone subset of 
$\R^d \times \R^d$ is contained in such a set. 
\end{theo}
\noindent The subdifferential of a convex function $\psi$ is a \emph{monotone operator} in the sense that 
the relation $\{(x, \partial \psi(x)): \, x \in \R^d \}$ has the property that $\nscp{x - z}{g_x - g_z} \geq 0$ for all
$x,z \in \R^d$ and all $g_x \in \partial \psi(x), \, g_z \in \partial \psi(z)$, which is in analogy to the fact monotone functions on the real line arise as derivatives of convex functions. 

In combination, Proposition \ref{prop:cyc_monotonicity} and Rockafellar's theorem above prompt the requirement $$f^* = \nabla \psi_{f^*}$$ 
for a convex function $\psi_{f^*}: \R^d \rightarrow \R^d$. Working with gradients instead of subdifferentials is needed in order to ensure that $f^*$ is actually a map from $\R^d$ to  $\R^d$ even though the distinction is somewhat minor in light of the fact that convex functions are differentiable
(Lebesgue) almost everywhere.  

For the purpose of permutation recovery ({\bfseries T1}) and in turn for \emph{strict} cyclical monotonicity to hold, we need to impose the additional requirement that $\psi_{f^*}$ be \emph{strictly} convex, i.e., 
the strengthened first-order convexity condition
\begin{equation}\label{eq:strict_convexity}
\psi_{f^*}(z) > \psi_{f^*}(x) + \nscp{\nabla \psi_{f^*}(x)}{z-x} \quad \forall \; x,z \in \R^d, \; x \neq z.     
\end{equation}
Note that $\nabla \psi_{f^*}: \R^d \rightarrow \R^d$ is injective if and only if \eqref{eq:strict_convexity} holds. In the presence
of noise, strict convexity will further be strengthened to strong convexity (cf.~Proposition \ref{prop:permutation_recovery} in $\S$\ref{sec:main_results} below).

\subsection{A path towards denoising (T2) via optimal transportation}\label{sec:Path-OT}
Gradients of convex functions are also known as Brenier maps in the field of optimal (measure) transportation \cite[e.g.,][]{Villani2009, Villani2003, Santambrogio2015}. Specifically, for random variables $U \sim \rho$ and $V \sim \tau$ with
$\rho$ and $\tau$ absolutely continuous with respect to the Lebesgue measure on $\R^d$ such that $\E_{U \sim \rho}[\nnorm{U}_2^2], \E_{V \sim \tau}[\nnorm{V}_2^2]$ are both finite, Brenier's theorem (in short) states that the minimization problem 
\begin{equation*}
\inf_T \, \frac{1}{2} \E_{U \sim \rho}[\nnorm{U - T(U)}_2^2],  \end{equation*}
over all measurable functions $T$ such that $T(U) \sim \tau$ has a solution $T^{\ast} = \nabla \psi_{T^*}$ for a convex function $\psi_{T^*}$ with $T^*$ being uniquely
determined almost everywhere. Moreover, the solution of the reverse problem in which $\tau$ is optimally transported to 
$\rho$ in the above sense is the optimal transport map given by $\nabla \psi_{T^*}^{\star}$ with $\psi_{T^*}^{\star}$ denoting
the Legendre-Fenchel conjugate of $\psi_{T^*}$; we refer to Appendix \ref{app:optimaltransport} for a more detailed background and references. 

Linear assignment problems of the form \eqref{eq:sorting} and $\eqref{eq:permutation_min_general}$ can be interpreted as specific discrete
optimal transport problems between the atomic measures 
\begin{equation*}
\mu_n := \textstyle\frac{1}{n} \textstyle\su \delta_{X_i}, \qquad \text{and} \qquad \nu_n := \frac{1}{n} \su \delta_{Y_i}.
\end{equation*}
The requirement that $\mu_n(T^{-1}(Y_i)) = 1/n$, $1 \leq i \leq n$, immediately implies that the 
resulting optimal transport problem seeks for an optimal pairing $\{ (X_{\pi(i)}, Y_i) \}_{i = 1}^n$ over all permutations
$\pi$ of $\{1,\ldots, n\}$. 

The connection to optimal transportation turns out to be fruitful since it suggests a natural approach for the task
of denoising ({\bfseries T2}), i.e., the construction of estimators $\{ \wh{f}(X_i) \}_{i = 1}^n$ of $\{f^*(X_i) \}_{i = 1}^n$ under the
permuted regression model \eqref{eq:permuted_regression}. Note that solving the linear assignment problem \eqref{eq:permutation_min_general} to find an optimal collection of $(X,Y)$-pairs is not suitable
for this task in general since all noise inherent in the $\{ Y_i \}_{i = 1}^n$ is retained (cf.~Figure \ref{fig:simple-example}). In fact, we are interested
in the pairings $\{ (X_i, \theta_i^*) \}_{i = 1}^n$ with $\theta_i^* = f(X_i^*)$, $1 \leq i \leq n$, which
corresponds to the optimal transportation problem between $\mu_n$ and $\nu_n^* := \frac{1}{n} \su \delta_{\theta_i^*}$. Since the latter is not given  --- in fact, it corresponds to the target to be recovered --- suggests the need for its estimation. 
Below, we shall present an atomic estimator $\wh{\nu}$ of $\nu_n^*$ of the form 
\begin{equation*}
\wh{\nu} := \textstyle \sum_{j = 1}^p \wh{\alpha}_j  \delta_{\wh{\theta}_j}
\end{equation*}
with atoms $\{ \wh{\theta}_j \}_{j = 1}^p \subset \R^d$ and masses (i.e., positive numbers summing to one) $\{ \wh{\alpha}_j \}_{j = 1}^p$. Since $p \neq n$ in general, there does not exist a transport map 
between $\mu_n$ and $\wh{\nu}$\footnote{The measure preservation property $\mu_n(T^{-1}(\wh{\theta}_j)) = 1/n$, $1 \leq j \leq p$, cannot hold since in general $n \neq p$.}. However, the \emph{Kantorovich problem}, a relaxation
of the optimal transportation problem (cf.~Appendix \ref{app:optimaltransport}), can be used to obtain a proxy as follows. The Kantorovich problem is given by 
the optimization problem 
\begin{equation}\label{eq:Kantorovich_main}
\min_{\gamma \in \Pi(\mu_n, \wh{\nu})}  \int \int \frac{1}{2} \nnorm{x - \theta}_2^2 \; d\gamma(x, \theta),
\end{equation}
where the minimum is over all couplings $\gamma$ of $\mu_n$ and $\wh{\nu}$, i.e., all probability measures on the set $\{ X_i \}_{i = 1}^n \times \{ \wh{\theta}_j \}_{j = 1}^p$ whose marginal distributions are given by $\mu_n$ and $\wh{\nu}$, respectively. 

Let $\wh{\gamma}$ denote a minimizer of \eqref{eq:Kantorovich_main}. We then use the estimator
\begin{equation}\label{eq:barycentric_proj_estimator}
\wh{f}(X_i) \coloneq \E_{(\theta, X) \sim \wh{\gamma}}[\theta | X = X_i] = \frac{\int_{\theta} \theta \; d\wh{\gamma}(\theta, X_i)}{{\int_{\theta} \; d\wh{\gamma}(\theta, X_i)}} =  \frac{\int_{\theta} \theta \; d\wh{\gamma}(\theta, X_i)}{\mu_n(\{ X_i \})}, \quad 1 \leq i \leq n,
\end{equation}    
i.e., the conditional expectation of $\theta$ given $X = X_i$, $1 \leq i \leq n$, resulting from the optimal coupling
$\wh{\gamma}$. 
The map $x \mapsto \E_{(X,\theta) \sim \wh{\gamma}}[\theta | X= x]$, $x \in \mc{X}_n$, is usually referred to as the \emph{barycentric projection} of $\wh{\gamma}$ in
the optimal transport literature \cite[][Definition 2]{Paty2020}. 

In order to finalize the outline of our approach for task ({\bfseries T2}), which is summarized in Algorithm \ref{alg:denoising}, it remains to present a specific estimator $\wh{\nu}$ of $\nu_n^*$. Let $\varphi$ denote the density of the i.i.d.~standardized noise terms $\{ \epsilon_i/\sigma \}_{i = 1}^n$, where we assume that $\E[\eps_1] = 0$ and $\cov(\epsilon_1) = \sigma^2 I_d$, for $\sigma > 0$. Then the average density of the $\{ Y_i \}_{i = 1}^n$ is given by the location mixture density $\textsf{f}_n := \varphi_{\sigma} \star \nu_n^*$ with
$\star$ denoting convolution and $\varphi_{\sigma}(\cdot) := \sigma^{-d} \varphi(\cdot/\sigma)$, i.e., 
\begin{equation*}
\textsf{f}_n(y) = \int \varphi_{\sigma}(y - \theta) \;d\nu_n^*(\theta) = \frac{1}{n} \su \varphi_{\sigma}(y - \theta_i^*), \quad y \in \R^d. 
\end{equation*}    
We propose to estimate $\textsf{f}_n$ via the Kiefer-Wolfowitz nonparametric maximum likelihood estimator (NPMLE)\footnote{Terminology varies in the literature; \cite{Zhang2009} uses the term ``generalized MLE".} \cite{Kiefer1956consistency, Koenker2014} given by 
\begin{equation}\label{eq:Kiefer_Wolfowitz}
\inf_{\textsf{f} \in \mc{F}_{\varphi, \sigma}} -\su \log \textsf{f}(Y_i), \quad \mc{F}_{\varphi, \sigma} \coloneq \left \{ \textsf{f} = \int \varphi_{\sigma}(y - \theta) d\nu(\theta): \;\,  \nu \, \, \text{distribution on $\R^d$} \right\}.   
\end{equation}
\begin{algorithm}[t]
\hspace*{-2.5ex}\begin{tabular}{ll}
$\begin{array}{l}
\text{{\bfseries Inputs}: $\mc{X}_n$, $\mc{Y}_{m}$, $\varphi_{\sigma}, \mathbb{G}$.} \\[1ex]
\text{1. Solve problem \eqref{eq:Kiefer_Wolfowitz_approx1}:}\\[.5ex]
\text{$\leadsto$ $\wh{\nu} = \sum_{j = 1}^p \wh{\alpha}_j \delta_{\wh{\theta}_j}$}\\[1ex]
\text{2. Compute an optimal coupling} \\
\text{$\quad$ between $\mu_n$ and $\wh{\nu}$ via the}\\
\text{$\quad$ linear program \eqref{eq:Kantorovich_finite}.}\\[.5ex]
\text{$\leadsto$ $\wh{\Gamma} \in \R_+^{n \times p}$.} \\[1ex]
\text{{\bfseries Return} $\wh{f}(X_i) = n \sum_{j} \wh{\Gamma}_{ij} \wh{\theta}_j$}, \\
\text{$\qquad \quad \;\,$ $1 \leq i \leq n$.} \\[30ex]
\end{array}$
& \hspace*{-2ex}\includegraphics[height = 0.2\textheight]{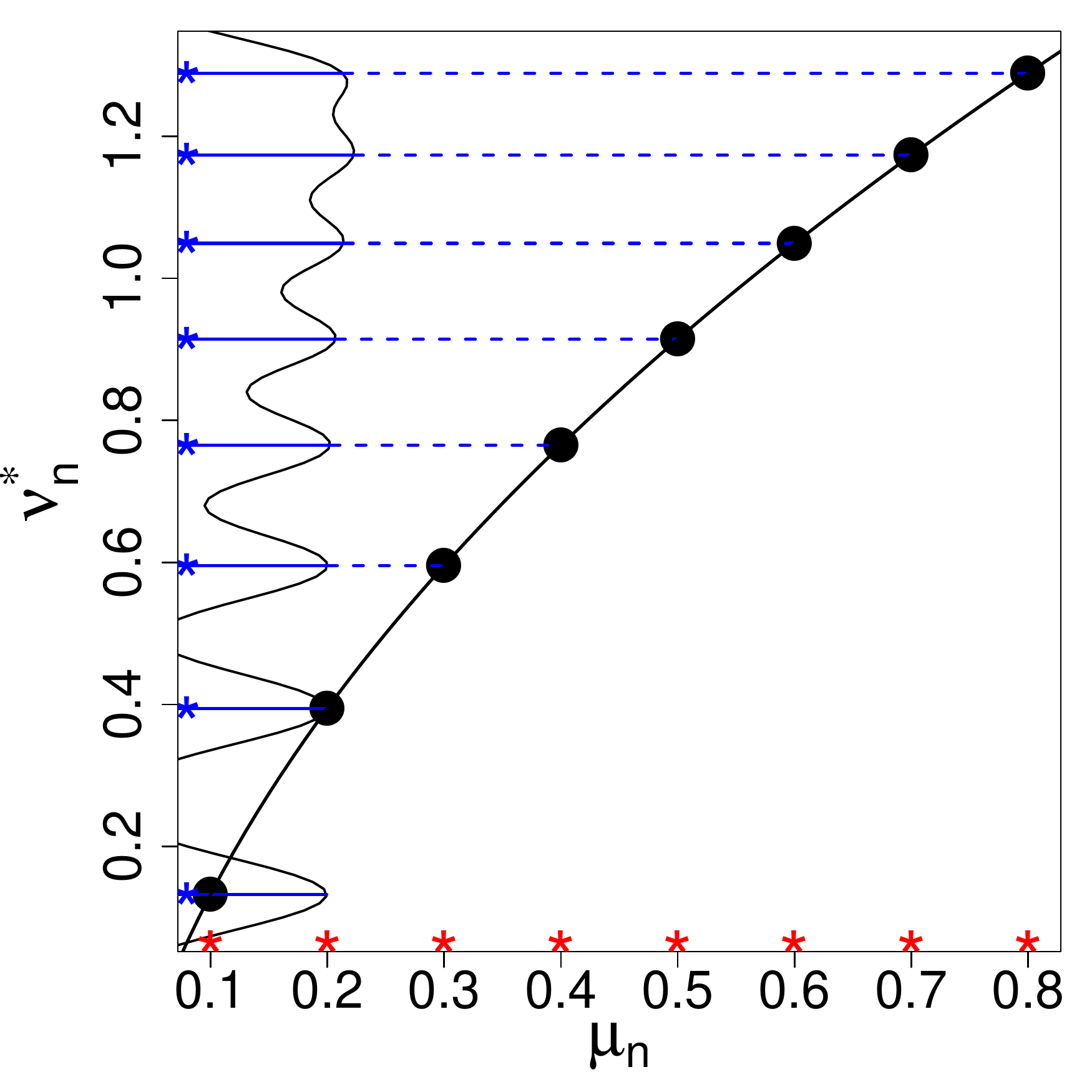} \hspace*{-1ex}\includegraphics[height = 0.2\textheight]{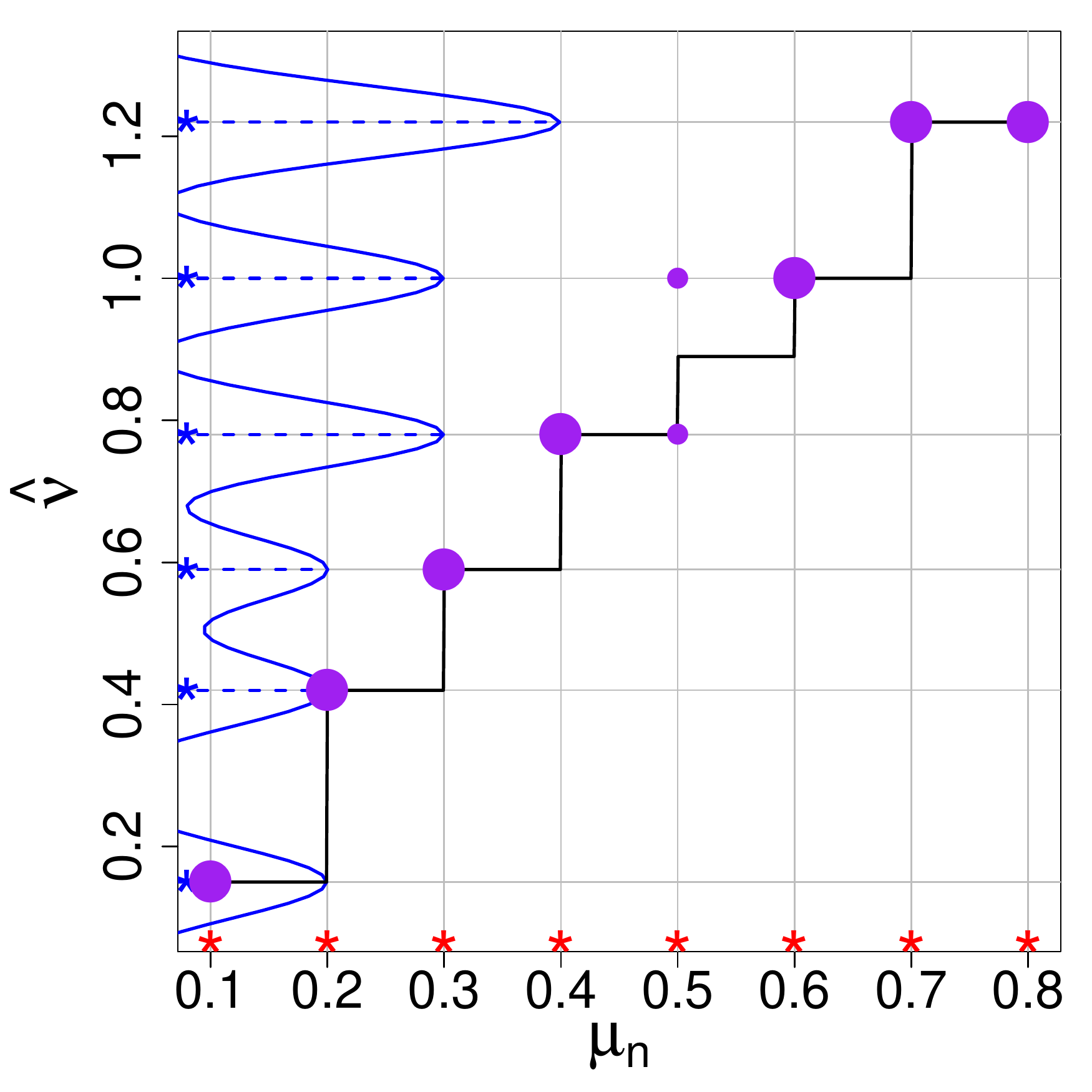}
    \end{tabular}
\vspace*{-30ex}    
    \caption{Denoising for Permuted or Unlinked Regression}
    \label{alg:denoising}\\[1ex]
  Left figure: $\nu_n^* = \frac{1}{n} \su \delta_{\theta_i^*}$ with $\theta_i^* = 2\sqrt{X}_i - 0.5$, $1 \leq i \leq n$, with $\mc{X}_n = \{0.1,0.2,\ldots,0.8 \}$. The solid black line drawn over the vertical axis represents the mixture density $\textsf{f}_n = \nu_n^* \star \varphi_{\sigma}$. Right figure: Estimated mixture density $\wh{\textsf{f}}_n$ and mixing measure $\wh{\nu}$ (blue). The resulting
optimal coupling $\wh{\Gamma}$ between $\mu_n$ and $\wh{\nu}$ is represented by purple dots (with sizes proportional to the 
corresponding entry of $\wh{\Gamma}$. Solid black line: Function estimate $\wh{f}$ obtained by constant interpolation
based on $\{ (X_i, \wh{f}(X_i)) \}_{i = 1}^n$.   
\end{algorithm}
Even though the optimization problem \eqref{eq:Kiefer_Wolfowitz} is infinite-dimensional, it can be shown that a solution $\wh{\textsf{f}}_n $ exists, and that the associated mixing measure $\wh{\nu}$ is atomic with a finite number of atoms \cite{Koenker2014, Lindsay1983}. We shall use $\wh{\nu}$ as an estimator of $\nu_n^*$ that is then plugged into the Kantorovich problem \eqref{eq:Kantorovich_main}. Note that the Kiefer-Wolfowitz problem assumes knowledge of the density $\varphi_{\sigma}$, i.e., the noise distribution. This assumption is common in deconvolution problems \cite{Meister2009} as encountered here; in fact, without any knowledge about the noise distribution, deconvolution problems are generally ill-defined. The assumption of known $\sigma$ can potentially be relaxed (cf.~$\S$\ref{sec:conclusion}).
\vskip1.5ex
\noindent {\bfseries Unlinked Regression}. The estimator \eqref{eq:barycentric_proj_estimator} remains applicable in 
the unlinked regression setting in which $\mc{X}_n$ and $\mc{Y}_{m}$ are of different sizes $n \neq m$ as described in the introduction 
with the elements of $\mc{Y}_{m}$ being i.i.d.~as $Y \overset{\mc{D}}{=} f^*(X) + \eps$ with $X \sim \mu$ for
some absolutely probability measure $\mu$ supported on a compact subset of $\R^d$ and $f^* = \nabla \psi_{f^*}$ for
$\psi_{f^*}$ convex. In fact, $\mc{Y}_{m}$ can be used to obtain an estimator $\wh{\nu}$ of $\frac{1}{n} \sum_{i = 1}^{n} \delta_{f^*(X_i)}$ as before via \eqref{eq:Kiefer_Wolfowitz}, and all subsequent steps in Algorithm \ref{alg:denoising} can be executed. The 
rates of convergence for the denoising error are almost identical to the permuted regression setting with $n = m$ as long as $n \asymp m$, cf.~$\S$\ref{subsec:denoising}. 
\vskip1.5ex
\noindent {\bfseries Computation}. Algorithm \ref{alg:denoising} requires computation of the Kiefer-Wolfowitz NPMLE, the 
Kantorovich problem \eqref{eq:Kantorovich_main}, and finally the barycentric projections \eqref{eq:barycentric_proj_estimator}. 
The Kiefer-Wolfowitz problem can be reformulated as a (non-convex) finite mixture likelihood optimization problem, and then solved
via the EM algorithm \cite{Jiang2009}. Instead, in order to preserve convexity, we approximate the solution of 
\eqref{eq:Kiefer_Wolfowitz} via the finite-dimensional optimization problem 
\begin{equation}\label{eq:Kiefer_Wolfowitz_approx1}
\inf_{\textsf{f} \in \mc{F}_{\varphi, \sigma}^{\mathbb{G}}} -\su \log \textsf{f}(Y_i), \quad \mc{F}_{\varphi, \sigma}^{\mathbb{G}} \coloneq \left \{ \textsf{f} = \int \varphi_{\sigma}(y - \theta) d\nu(\theta): \;\,  \nu \, \, \text{distribution on $\mathbb{G}$} \right\},  
\end{equation}       
where $\mathbb{G}$ is a finite set of points in $\R^d$. Problem \eqref{eq:Kiefer_Wolfowitz_approx1} can be rewritten as 
\begin{equation}\label{eq:Kiefer_Wolfowitz_approx2}
\inf_{\alpha \in \Delta^{|\mathbb{G}|}} - \sum_{i = 1}^n \log \left(\sum_{j = 1}^{|\mathbb{G}|} \alpha_j \varphi_{\sigma}(Y_i - \theta_j) \right), 
\end{equation}
where $\Delta^r \coloneq \{x \in \R_+^{r}: \sum_{j = 1}^r x_j = 1\}$ denotes the probability simplex in $\R^r$, $r \geq 1$. There is 
a variety of convex optimization algorithms that can be used to solve \eqref{eq:Kiefer_Wolfowitz_approx2}. Our experiments
are based on a primal-dual interior point method \cite{boyd2004convex} that yields fast and highly accurate results even if $|\mathbb{G}|$ includes several thousand points. Regarding $\mathbb{G}$, our default is choice is $\mathbb{G} = \mc{Y}_n$ for
$d \geq 2$ and $\mathbb{G}$ being a set of $g_n$ linearly spaced points in the interval $[\min_i Y_i, \max_i Y_i]$ with
$g_n = n$ or $g_n = n^{1/2}$. In the paper \cite{Dicker2016} it is shown that the latter choice suffices to ensure
comparable statistical performance to the solution of the infinite-dimensional problem \eqref{eq:Kiefer_Wolfowitz}. 

Solving \eqref{eq:Kiefer_Wolfowitz_approx1} yields the estimator $\wh{\nu} = \sum_{j = 1}^p \wh{\alpha}_j \delta_{\wh{\theta}_j}$,
where $\{ \wh{\alpha}_j \}_{j = 1}^p$ represent the non-zero entries of the resulting minimizer of \eqref{eq:Kiefer_Wolfowitz_approx2} and $\{ \wh{\theta}_j \}_{j = 1}^p \subseteq \mathbb{G}$ represent the corresponding atoms. Computing an optimal coupling between the two finitely supported measures $\mu_n$ and $\wh{\nu}$ according to problem \eqref{eq:Kantorovich_main} amounts to solving the linear program 
\begin{equation}\label{eq:Kantorovich_finite}
\min_{\Gamma \in \R_+^{n \times p}} \tr(C^{\T} \Gamma) \quad \text{subject to} \; \sum_{i = 1}^n \Gamma_{ij} = \wh{\alpha}_j, \; 1 \leq j \leq p, \;\;\, \sum_{j = 1}^p \Gamma_{ij} = \frac{1}{n}, \; 1 \leq i \leq n, 
\end{equation}
where $C = (\nnorm{X_i - \wh{\theta}_j}_2^2/2)_{1 \leq i \leq n, 1 \leq j \leq p}$, and the row and column sum constraints represent the requirements on the two marginal distributions. Solving \eqref{eq:Kantorovich_finite} exhibits similar computational complexity to the linear assignment problem \eqref{eq:LAP}. For the numerical examples presented in this paper, we used the routine \texttt{cplexlp} in \textsf{CPLEX} \cite{cplex}. Fast approximate solution can be obtained via Sinkhorn iterations \cite{Cuturi2019}. For $d = 1$, problem \eqref{eq:Kantorovich_finite} becomes considerably simpler due to the natural
ordering of the real line, and can be solved in time $O(n + p)$ via the so-called ``Northwest Corner Rule" \cite[][$\S$3.4.2]{COT2019} after sorting the $\{ X_i \}_{i = 1}^n$ and $\{ \wh{\theta}_j \}_{j = 1}^p$. 

Finally, given a minimizer $\wh{\Gamma}$ of \eqref{eq:Kantorovich_finite}, the barycentric projections \eqref{eq:barycentric_proj_estimator} can be computed as 
\begin{equation*}
\wh{f}(X_i) = \sum_{j = 1}^p \wh{\Gamma}_{ij} \wh{\theta}_j \, \Big/ \sum_{j = 1}^p \wh{\Gamma}_{ij} = n \sum_{j = 1}^p \wh{\Gamma}_{ij} \wh{\theta}_j, \quad 1 \leq i \leq n.
\end{equation*}
\section{Main results}\label{sec:main_results}
In this section, we first analyze permutation recovery ({\bfseries T1}) based on the linear assignment problem in \eqref{eq:permutation_min_general} with the distinction that $\{ Y_i \}_{i = 1}^n$ may be 
contaminated by Gaussian additive noise, i.e., $Y_i = f(X_{\pi^*(i)}) + \epsilon_i$, $1 \leq i \leq n$, with 
$\{ \epsilon_i \}_{i = 1}^n$ being i.i.d.~$N(0,\sigma^2 I_d)$-distributed random variables. The Gaussianity
assumption is not essential; generalizations to the non-isotropic case or other noise distributions satisfying various tail conditions (sub-Gaussian, sub-Exponential, $\ldots$) appear rather straightforward, and are not pursued in this paper to simplify the exposition and
to facilitate the comparison to related results in previous literature, specifically \cite{Ma2020, ZhangSlawskiLi2021, Flammarion16}.

The main technical contribution of this paper is the analysis of Algorithm \ref{alg:denoising} for the purpose of 
denoising ({\bfseries T2}), which is presented subsequently. 

\subsection{Permutation recovery}\label{subsec:permutation_recovery}
Consider the following linear assignment problem under the permuted regression setup \eqref{eq:permuted_regression}:
\begin{equation}\label{eq:LAP_permutation_recovery}
\min_{\pi} \frac{1}{2} \su \nnorm{Y_i - X_{\pi(i)}}_2^2, 
\end{equation}
where the minimization is over all permutations $\pi$ of $\{1,\ldots, n\}$.
Let $\wh{\pi}$ denote the minimizer of \eqref{eq:LAP_permutation_recovery}. Assuming i.i.d.~Gaussian errors, the following result (Proposition~\ref{prop:permutation_recovery}) states sufficient conditions for exact permutation recovery, i.e., the event $\{\wh{\pi} = \pi^* \}$, to occur with high probability. Comparison to existing results will indicate that the 
required conditions cannot substantially be relaxed.

The discussion below Theorem \ref{theo:rockafellar} in $\S$\ref{sec:problem_approach} has indicated the necessity of
the requirement that $\psi_{f^*}$ be strictly convex already in the absence of noise. A further strengthening to strong convexity, i.e., 
\begin{equation}\label{eq:strong_convexity}
\psi_{f^*}(z) \geq \psi_{f^*}(x) + \nscp{\nabla \psi_{f^*}(x)}{ z - x} + \frac{\lambda}{2} \nnorm{x - z}_2^2 \quad \forall x,z \in \R^d    
\end{equation}
becomes necessary to counteract noise\footnote{To obtain more intuition, note that \eqref{eq:strong_convexity} is equivalent to $\scp{\nabla \psi_{f^*}(z) - \nabla \psi_{f^*}(x)}{z - x} \geq \lambda \nnorm{x - z}_2^2$;  the left hand
side of this expression corresponds to the non-noise contributions when comparing the objectives of the LAP \eqref{eq:LAP_permutation_recovery} for $n = 2$ at $\pi_1 = \textsf{id}$ and $\pi_2 = (2 \; 1)$, respectively.}. Equipped with strong convexity, we are in position to state the following result (proved in Appendix~\ref{sec:Proof-Prop-2}). 
\begin{prop}\label{prop:permutation_recovery}
Suppose that  $Y_i = f^*(X_{\pi^*(i)}) + \eps_i$, $1 \leq i \leq n$ with $f^* = \nabla \psi_{f^*}$ being the gradient 
of a $\lambda$-strongly convex function $\psi_{f^*}$, for fixed vectors $\{ X_i \}_{i = 1}^n \subset \R^d$ and i.i.d.~errors $\{ \eps_i \}_{i = 1}^n \sim N(0, \sigma^2 I_d)$. Let $\wh{\pi}$ denote the minimizer of the optimization 
problem \eqref{eq:LAP_permutation_recovery}. 
If $\min_{i < j} \nnorm{X_i - X_j}_2 > \lambda^{-1} \sigma \sqrt{6\log n}$, it holds with probability at least $1 - 1/n$ that $\wh{\pi} = \pi^*$.
\end{prop}

\noindent {\bfseries Discussion}. Comparison to previous work indicates that the separation condition 
\begin{equation}\label{eq:separation_condition}
\min_{i < j} \nnorm{X_i - X_j}_2 \geq \lambda^{-1}  \sigma \sqrt{6\log n}
\end{equation}
cannot be substantially relaxed. The paper \cite{ZhangSlawskiLi2021} considers the case in which 
$f^*(x) = B^* x$ is a linear transformation, which corresponds to 
$\psi_{f^*}(x) = \frac{1}{2} x^{\T} B^* x$ (up to an additive constant). Under the assumption of Gaussian noise as in Proposition \ref{prop:permutation_recovery} and Gaussian design, i.e., $\{ X_i \}_{i = 1}^n \overset{\text{i.i.d.}}{\sim} N(0, I_d)$, it is shown that permutation recovery \emph{fails}  for \emph{any estimator} with probability at least $1/2$ whenever 
\begin{equation}\label{eq:permutation_recovery_minimax}
\sum_{j = 1}^d \log \left(1 + \frac{\lambda_j^2}{\sigma^2} \right) \leq \log n,
\end{equation}
where $\{ \lambda_j \}_{j = 1}^{d}$ are the singular values of $B^*$. In the setting of this paper, $B^*$ is required
to be symmetric positive semidefinite. Suppose that $B^*$ has bounded condition number, i.e., 
$\lambda I_d \lec B^* \lec C \, \lambda I_d$ for some constant $C \geq 1$. In this case, the left hand side of 
\eqref{eq:permutation_recovery_minimax} becomes proportional to $d \log(1 + \frac{\lambda^2}{\sigma^2}) \leq d \frac{\lambda^2}{\sigma^2}$, and hence in summary, permutation recovery cannot succeed if 
$d \lesssim \lambda^{-2} \sigma^2 \log n$. On the other hand, for $\{ X_i \}_{i = 1}^n \overset{\text{i.i.d.}}{\sim} N(0, I_d)$, concentration results for Gaussian random vectors and the union bound yields that 
$\min_{i < j} \nnorm{X_i - X_j}_2 \gtrsim \sqrt{d} - \sqrt{\log n} \gtrsim \sqrt{d}$ for $d \gtrsim \log n$ with high probability, which, when substituted into \eqref{eq:separation_condition}, implies that the condition $d \gtrsim \lambda^{-2} \sigma^2 \log n$ suffices for permutation recovery to succeed. 

The above example shows that the condition in Proposition \ref{prop:permutation_recovery} is generally sharp, up to a constant factor. Moreover, the example reveals a ``blessing of dimensionality" phenomenon in the sense that permutation recovery can typically (only) be hoped for in the regime $d \gtrsim \log n$. Indeed, for sub-Gaussian random designs, in that regime
the scaling of the minimum separation $\min_{i < j} \nnorm{X_i - X_j}_2$ begins to outweigh the $\sqrt{\log n}$ factor
on the right hand side of the sufficient condition \eqref{eq:separation_condition}, cf.~\cite[][Lemma B.1]{SlawskiBenDavidLi2019}. By contrast, for $d = O(1)$, $\min_{i < j} \nnorm{X_i - X_j}_2$ may  exhibit polynomial decay in $n$ \cite[][Lemma 2]{SlawskiBenDavidLi2019}. 

Finally, the specialization of Proposition \ref{prop:permutation_recovery} to a linear map shows that so-called unlabeled sensing problems \cite{Unnikrishnan2015} (i.e., permuted regression problems with $f^*$ linear) can be solved \emph{efficiently} via the linear assignment problem \eqref{eq:LAP_permutation_recovery} if the underlying linear map is \emph{positive definite}. So far, no computationally efficient approach to unlabeled sensing problems with provable recovery guarantees was known except for the case of ``sparse shuffling" in which $\pi^*$ is known to permute only a somewhat small
fraction of $\{1,\ldots,n\}$ \cite{SlawskiBenDavidLi2019, SlawskiRahmaniLi2018, ZhangSlawskiLi2021, ZhangLi2020, Peng2021}. \vskip2ex    

\noindent {\bfseries Connection to recovery results in the ``permuted monotone matrix model"}. The paper \cite{Ma2020}
considers the model
\begin{equation}\label{eq:permuted_monotone_matrix}
\M{Y} = \Pi^* \Theta^* + \M{Z},
\end{equation}
where $\Pi^*$ and $\Theta^*$ are unknown permutation and ``signal" matrices of dimension $n$-by-$n$ and 
$n$-by-$d$, respectively, and the entries of the noise matrix $\M{Z}$ are i.i.d.~$N(0,\sigma^2)$-distributed. Moreover,
the entries of each of the columns of $\Theta^*$ are arranged in increasing order, i.e., for all $1 \leq j \leq d$, it holds that $\Theta^*_{ij} < \Theta^*_{(i+1)j}$, for $1 \leq i \leq n-1$. 

The paper \cite{Ma2020} studies the problem of 
recovering $\Pi^*$ from $\M{Y}$. One can think of the entries $(\Theta_{ij}^*)$ as evaluations of monotone
increasing functions $\{ f_j^* \}_{j = 1}^d$ at (unknown) design points $X_{ij}$, i.e., $\Theta_{ij}^* = f_j^*(X_{ij})$, $1 \leq i \leq n$, $1 \leq j \leq d$. Observe that functions of the form $f^*(x) \equiv f^*(x_1, \ldots, x_d) = (f_1^*(x_1), \ldots, f_d^*(x_d))$ with
$\{ f_j^* \}_{j = 1}^d$ monotone increasing equal the gradient of a sum of \emph{univariate} convex functions, i.e.,
$f^*(x) = \nabla (\psi_{f_1^*}(x_1) + \ldots + \psi_{f_d^*}(x_d))$ with $\{ \psi_{f_j^*} \}_{j = 1}^d$ convex, which constitutes
an important special case of the class of functions that are gradients of convex functions. As opposed to the setup
under consideration in this paper, the setting in \cite{Ma2020} does not involve any design points $\{ X_i \}_{i = 1}^n$. 
However, specific (user-designed) choices of those points in conjunction with the linear assignment problem \eqref{eq:LAP_permutation_recovery} with $Y_i$ (the $i$-th row of $\M{Y}$), $1 \leq i \leq n$, can lead to specific approaches for
recovering $\Pi^*$. Perhaps the most straightforward choice is given by $X_i = x_i \M{1}_d$, $1 \leq i \leq n$,  for any 
increasing sequence of scalars $\{ x_i \}_{i = 1}^n \subset \R$; in this case, the LAP \eqref{eq:LAP_permutation_recovery} reduces to 
sorting the rows of $\M{Y}$ according to their row sums, which is also a rather intuitive strategy. In \cite{Ma2020} the leading right singular vector of $\M{Y}$ is used instead of $\bm{1}_d$, which yields improved recovery results. 
\begin{rem}\label{rem:Comp-Ma}{\bf (Comparison to results in \cite{Flammarion16, Ma2020})} 
The conditions for permutation recovery in \cite{Ma2020} very much align with our condition \eqref{eq:separation_condition}. The agreement can be seen best if $\{ f_j^* \}_{j = 1}^d$ are linear functions with non-negative slopes $\{ \eta_j \}_{j = 1}^d$ and $\Theta_{ij}^* \equiv  f_{j}^*(X_{ij}) = f_j^*(x_i)$, $1 \leq i \leq n$, $1 \leq j \leq d$, for scalars $x_1 < \ldots < x_n$, in which case $\Theta^* = \M{x} \bm{\eta}^{\T}$ with $\M{x} = (x_1, \ldots, x_n)^{\T}$ and 
$\bm{\eta} = (\eta_1, \ldots, \eta_d)^{\T}$. 
It is shown in \cite{Ma2020} that the condition $\nnorm{\bm{\eta}}_2 \gtrsim \sigma \sqrt{\log n}$ is necessary (in a minimax sense) for exact permutation recovery. Observe that 
$\nnorm{\bm{\eta}}_2 \asymp \sqrt{d} \min_{1 \leq j \leq d} \eta_j$ as long 
as the slopes are of the same order, which agrees with the recovery condition \eqref{eq:separation_condition} up to constant factors
noting that here $\lambda = \min_{1 \leq j \leq d} \eta_j$ and assuming the scaling $\min_{i < j} \nnorm{X_i - X_j}_2 \asymp \sqrt{d}$ as explained above. In particular, the requirement $d \gtrsim \log n$ becomes manifest once more, and also appears as a crucial condition in the paper \cite{Flammarion16}. The latter studies model \eqref{eq:permuted_monotone_matrix} with the goal of estimating the signal $\Theta^*$ rather than the permutation $\Pi^*$. The authors of \cite{Flammarion16} show that the excess error in estimating $\Theta^*$ relative to an oracle that is equipped with knowledge of $\Pi^*$ is proportional to $\log(n) / d$.
\end{rem}

\subsection{Denoising}\label{subsec:denoising}
In this subsection, we present our main results on the denoising task {\bfseries (T2)} based on
Algorithm \ref{alg:denoising}. In particular, we provide upper bounds on the mean squared error that indicate that this task can indeed be accomplished, albeit at slow rates. 

The subsection is organized as follows: (i) we first present
a result under the assumption of Gaussian errors for the permuted regression setting \eqref{eq:permuted_regression}, which is readily extended to (ii) the \emph{unlinked regression}
setting with samples $\mc{X}_n$ and $\mc{Y}_{m}$ of different size; (iii) the univariate case $d = 1$ admits faster rates, relaxed assumptions, and a considerably simpler proof. We then discuss (iv) how
these results can be extended to errors from elliptical distributions with ``benign tails" for which the 
associated NPMLE $\wh{\textsf{f}}_n$ can be expected to behave similarly as in the Gaussian case. 

The following theorem addresses item (i). We first list the key assumptions on $f^* = \nabla \psi_{f^*}$. 
\begin{itemize}
\item[{\bfseries(A1)}] The function $\psi_{f^*}$ is $\lambda$-strongly convex, i.e., \eqref{eq:strong_convexity} holds.
\item[{\bfseries(A2)}] The function $\psi_{f^*}$ is $L$-smooth, i.e.,
                       \begin{equation}\label{eq:Lsmooth}
                       \psi_{f^*}(z) \leq \psi_{f^*}(x) + \nscp{\nabla \psi_{f^*}(x)}{z-x} + \frac{L}{2} \nnorm{x - z}_2^2 \quad \forall x,z \in \R^d.    
                       \end{equation}
\item[{\bfseries(A3)}] The sequence $f^*(X_i) \equiv \theta_i^*$, $1 \leq i \leq n$, is uniformly bounded, i.e., 
                       there exists $0 < B < \infty$ such that $\max_{1 \leq i \leq n} \nnorm{\theta_i^*}_2 \leq B$. 
\end{itemize}
\begin{theo}\label{theo:permuted_generald} Consider the permuted regression problem \eqref{eq:permuted_regression} and 
suppose that  $\{ \epsilon_i \}_{i = 1}^n$ are i.i.d.~Gaussian errors with zero mean and covariance $\sigma^2 I_d$.
Let $\{ \wh{f}(X_i) \}_{i = 1}^n$ be the output of Algorithm \ref{alg:denoising}. 
Then if $n \gtrsim_{d, B, \sigma} (\log n)^{d+1}$, with probability at least $1 - 5/n$, it holds that 
\begin{equation*}
\frac{1}{n} \su \nnorm{\wh{f}(X_i) - f^*(X_i)}_2^2 \lesssim_{\sigma, d, B} \frac{L}{\lambda} \frac{1}{\log n},
\end{equation*}
where $\gtrsim_{[\ldots]}$ and $\lesssim_{[\ldots]}$ indicate the presence of a positive multiplicative constants depending
only on the quantities $[\ldots]$ given in the subscripts.
\end{theo}
The above theorem (proved in Appendix~\ref{sec:Proof:Thm-2-3}) indicates a rather slow rate of convergence proportional to $1/\log n$. For ease of exposition, we refrain from elaborating on the constants in terms of $\sigma$, $d$, and $B$; details can be found in the Appendix containing the proofs. Even though this paper does not present a (minimax) lower bound, rates faster than logarithmic decay generally appear little plausible in view of results in the deconvolution literature \cite[e.g.,][]{Hall2008, Fan1991, Dattner2011}. Our simulation results
in $\S$\ref{sec:numerical_results} largely corroborate the rate in Theorem \ref{theo:permuted_generald}. 

\vskip1.5ex
\noindent {\bfseries Unlinked Regression}. Our next result (proved in Appendix~\ref{sec:Proof:Thm-2-3}) constitutes a counterpart to Theorem \ref{theo:permuted_generald} in the unlinked regression
setting. 
\begin{theorem}\label{theo:unlinked_generald} 
Consider random variables $X \sim \mu$ and $Y \overset{\mc{D}}{=} f^*(X) + \eps$ with $f^* = \nabla \psi_{f^*}$ such
that \emph{{\bfseries (A1)}},~\emph{{\bfseries (A2)}} and 
\begin{equation*}
\emph{{\bfseries (A3$^{\prime}$)}}: \qquad \p_{X \sim \mu}(\nnorm{f^*(X)}_2 \leq B) = 1
\end{equation*}
hold true, and $\eps \sim N(0, \sigma^2 I_d)$. Let $\{ \wh{f}(X_i) \}_{i = 1}^n$ denote the output of Algorithm 
\ref{alg:denoising} given samples $\mc{X}_n = \{ X_i \}_{i = 1}^n$ and $\mc{Y}_{m} = \{ Y_i \}_{i = 1}^{m}$
consisting of i.i.d.~copies of $X$ and $Y$, respectively. Then if $m \geq  C_1(d, B, \sigma) (\log m)^{d + 1}$, with probability at least $1 - 5/m - C (n^{-c} + m^{-c})$, it holds that 
{\small \begin{align*}
\frac{1}{n} \su \nnorm{\wh{f}(X_i) - f^*(X_i)}_2^2 &\leq 
 \frac{2L}{\lambda} \Bigg( \frac{C_2(\sigma, d, B)}{\log m}  + 2 \sqrt{\frac{\log n}{n}} \vee \Big(\frac{\log n}{n}\Big)^{2/d} + 2 \sqrt{\frac{\log m}{m}} \vee \Big(\frac{\log m}{m}\Big)^{2/d} \Bigg)
\end{align*}}
for absolute constants $C, c > 0$ and constants $C_1 > 0$ and $C_2 > 0$ depending only on the quantities in the parentheses. 
\end{theorem}
\noindent The above statement indicates that the unlinked regression case does not behave fundamentally differently from the permuted regression setting. Specifically, as long as $n \asymp m$ the extra terms in Theorem \ref{theo:unlinked_generald} incurred in distinction to Theorem \ref{theo:permuted_generald} are lower order terms; they reflect the Wasserstein distance between the two measures $\frac{1}{n} \su \delta_{f^*(X_i)}$ and $\frac{1}{m} \sum_{i = 1}^{m} \delta_{\theta_i^*}$ with $\theta_i^* \overset{\mc{D}} = f^*(X_1)$, $1 \leq i \leq n$. This Wasserstein distance decays more rapidly than the Wasserstein deconvolution rate of the NPMLE, which reflects the error incurred in step 1 in Algorithm \ref{alg:denoising}.
\vskip1ex
\noindent We now state a separate result for the case $d = 1$; see Appendix~\ref{sec:Proof-Prop-3} for a proof. Even though the rates remain unchanged, it is noteworthy
that assumptions {\bfseries(A1)} and {\bfseries(A2)} are no longer required.
\begin{prop}\label{prop:denoising_d1}
Suppose that $d = 1$. Then in the situation of Theorem \ref{theo:permuted_generald},
\begin{equation*}
\frac{1}{n} \su |\wh{f}(X_i) - f^*(X_i)|^2 \lesssim_{\sigma, B} \frac{1}{\log n}.
\end{equation*}
Furthermore, in the situation of Theorem \ref{theo:unlinked_generald}, 
\begin{equation*}
\frac{1}{n} \su |\wh{f}(X_i) - f^*(X_i)|^2 \leq C(\sigma, B) \frac{1}{\log m} + 4 \left( \sqrt{\frac{\log n}{n}} + \sqrt{\frac{\log m}{m}} \right)    
\end{equation*}
where $C(\sigma, B)$ is a constant depending only on $\sigma$ and $B$.
\end{prop}
\noindent At this point, it is worth comparing the rates in Proposition \ref{prop:denoising_d1}, in the case $d = 1$,
to previous results in the literature. Regarding the permuted regression setting, the rate in Proposition~\ref{prop:denoising_d1} falls slightly short of the minimax rate $\{\log \log n / \log n\}^{2}$ in \cite{Weed2018}. At the same time, the
approach taken herein yields slightly faster rates in the unlinked regression setting than \cite{Balabdoui2020}. The authors
of \cite{Balabdoui2020} bound the mean absolute error rather than the mean squared error; for Gaussian errors, they obtain 
the rate $1 / (\log n)^{1/4}$, whereas a minor adoption of the proof of Proposition \ref{prop:denoising_d1} yields the 
rate $1 / (\log n)^{1/2}$ for the mean absolute error for the proposed estimator. 
\vskip1.5ex
\noindent {\bfseries Proof techniques and extension to other noise distributions}. We anticipate that results similar to Theorems \ref{theo:permuted_generald}, \ref{theo:unlinked_generald} and Proposition \ref{prop:denoising_d1} can be 
obtained for other (isotropic) noise distributions. In fact, in our proofs Gaussianity is used explicitly 
only via a specific upper bound taken from \cite{Saha2020} on the Hellinger distance $\dH(\textsf{f}_n, \wh{\textsf{f}}_n)$ for
the Kiefer-Wolfowitz NPMLE \eqref{eq:Kiefer_Wolfowitz}. Outside the Gaussian distribution, such bounds
do not appear to be available in the current literature. In the permuted regression setting, a key intermediate
result is a bound on the squared Wasserstein distance of the form (modulo constants) 
\begin{equation}\label{eq:W2bound_general}
\dW_2^2(\nu_n^*, \wh{\nu}) \leq \min_{\delta > 0} \left\{ \delta^2 +  \left[\textsf{h}^2 (1/\delta)^{d} \Psi(1/\delta) \right]^{\eta(d)} \right\},
\end{equation}    
where $\textsf{h}$ is an upper on the Hellinger distance $\dH(\wh{\textsf{f}}_n, \textsf{f}_n)$, $\Psi(1/\delta)$ is an upper
bound on the reciprocal of the Fourier transform\footnote{Recall that we use the superscript $^{\wh{}}$ to indicate the Fourier transform of a function.} of the density of the errors $1/ \wh{\varphi_{\sigma}}^2$ over $[-1/\delta, 1/\delta]^d$, and $\eta(d) = 2/ (d + 8)$. 

For Gaussian errors, we use $\Psi(\delta) = \exp((1/\delta)^2 \, c)$ and 
in turn $(1/\delta^d) \Psi(1/\delta) \leq \exp((1/\delta)^2 c')$ for constants $c, c' > 0$. Choosing $\delta^{-2} = \frac{1}{c'} \{-\log \textsf{h} \}$,
\eqref{eq:W2bound_general} becomes
\begin{equation*}
\dW_2^2(\nu_n^*, \wh{\nu}) \leq c' \frac{1}{\log \frac{1}{\textsf{h}}} + \textsf{h}^{\eta(d)} \lesssim \frac{1}{\log n}, 
\end{equation*}
using the upper bound on $\textsf{h}$ according to Lemma \ref{lem:saha} in Appendix \ref{app:hellinger_npmle}.

Faster polynomial rates can be obtained for error distributions  for which $\Psi(1/\delta) = d (1/\delta)^{\beta} \lesssim (1/\delta)^{\beta}$; an example ($\beta = 4$) is given by the multivariate 
Laplace distribution with $\wh{\varphi_{\sigma}}(\omega) = \frac{1}{1 + \sigma^2 \nnorm{\omega}_2^2}$. With the choice 
$\delta^2 = \{ \textsf{h}^2 \delta^{-(\beta + d)} \}^{\eta(d)}$, \eqref{eq:W2bound_general} becomes $2\textsf{h}^{\frac{8}{2(d + 8) + (\beta + d) \cdot 2}}$ (cf.~Remark \ref{rem:polynomial_decay} in Appendix \ref{app:deconvolution}). Similar deconvolution rates for finite mixtures of Laplace distributions   
for $d = 1$ are shown in \cite{Gao2016}. We also note that in prior work on unlinked regression \cite{Balabdoui2020}, error 
bounds are derived in terms of the decay of $\wh{\varphi_{\sigma}}$. 

Establishing the bound \eqref{eq:W2bound_general} requires additional conditions on the tails of $\varphi_{\sigma}$ so that
$\wh{\nu}$ can be shown to have sufficient moments (specifically, of order $4$) with high probability. The latter property
is easiest to verify for suitable spherical distributions with $\varphi(z) = \varphi(\nnorm{z}_2)$, and particularly among those,
for suitable scale mixtures of Gaussian distributions such as the aforementioned Laplace distribution (cf.~Remark \ref{rem:truncation} in Appendix \ref{app:deconvolution}). At the same time, heavy-tailed scale mixtures such a the Cauchy distribution will not lend themselves to a bound of the form \eqref{eq:W2bound_general}.

\section{Numerical Results}\label{sec:numerical_results}
In this section, we corroborate key aspects of our rationale and analysis in the preceding sections via numerical
examples. The empirical performance of the proposed approach with regard to denoising ({\bfseries T2}) will also be 
investigated in detail, and compared to two competing methods \cite{Balabdoui2020, Weed2018} proposed previously for the case $d = 1$. 

\subsection{Permutation Recovery}\label{subsec:numerical_permutation}
This subsection is intended as an illustration of Proposition \ref{prop:permutation_recovery} concerning
task ({\bfseries T1}), i.e., exact permutation recovery. Three different settings are considered: 
\vskip1.5ex
\noindent \texttt{psd}: $f^*(x) = B x$, where $B$ is a symmetric positive definite matrix, corresponding to
the gradient of the convex function $x \mapsto \frac{1}{2} x^{\T} B x$. In each replication, 
we generate $B \sim \text{df}^{-1}\text{Wishart}(I_d, \text{df} = 2\cdot d)$, where ``$\text{df}$" is short
for ``degrees of freedom", $\{ X_i \}_{i = 1}^n \overset{\text{i.i.d.}}{\sim} N(0, I_d)$, 
and finally $Y_i = f^*(X_i) + \sqrt{3/2}\, \eps_i$, $1 \leq i \leq n$.
\vskip1.5ex
\noindent \texttt{sep}: $f^*(x) = (3/2) \cdot (\sqrt{x_1}, \ldots, \sqrt{x_d})$, corresponding to the gradient
of the separable convex function $x \equiv (x_1, \ldots, x_d)\mapsto \sum_{j = 1}^d x_j^{3/2}$ on $\R_+^d$. In each replication, 
we generate $\{ X_i \}_{i = 1}^n \overset{\text{i.i.d.}}{\sim} \textsf{U}([0,1]^d)$ 
and $Y_i = f^*(X_i) + \sqrt{2/7} \, \eps_i$, $1 \leq i \leq n$, where $\textsf{U}(\ldots)$ denotes the 
uniform distribution.
\vskip1.5ex
\noindent \texttt{exp-norm}: $f^*(x) = \frac{1}{2} \frac{x}{\nnorm{x}_2} \cdot \exp(\nnorm{x}_2/2)$, 
corresponding to the gradient of the convex function $x \mapsto \exp(\nnorm{x}_2/2)$; convexity follows
from the composition rules given in \cite[][$\S$3.2.4]{boyd2004convex}. In each replication, we generate
$\{ X_i \}_{i = 1}^n \overset{\text{i.i.d.}}{\sim} N(0, I_d)$, and $Y_i = f^*(X_i) + 4 \eps_i$, $1 \leq i \leq n$. 
\vskip1.5ex
\noindent In all three settings, we fix $n = 1,000$ and the noise terms $\{\epsilon_i \}_{i = 1}^n$ are drawn
i.i.d.~from the $N(0, I_d)$-distribution. The noise variance is chosen specifically for each setting, to ensure
comparable signal-to-noise ratios\footnote{The signal-to-noise ratio can be formally defined via the left and right hand side in the recovery condition of Proposition \ref{prop:permutation_recovery}} across the three settings. The dimension $d$ is varied between $10$ and $70$ in steps of $10$. For each setting
and each value of $d$, we perform $100$ independent replications. In each replication, we solve the linear assignment
problem \eqref{eq:LAP_permutation_recovery}, and obtain the scaled Hamming distance
$\frac{1}{n} \su \mathbb{I}(\wh{\pi}(i) \neq i)$ (note that here $\pi^*(i) = i$, $1 \leq i \leq n$). The results are shown in 
Figure \ref{fig:permutation_recovery}, and confirm the central insight that results from Proposition \ref{prop:permutation_recovery}, namely that permutation recovery becomes considerably easier 
as the dimension $d$ increases in view of the scaling of $\min_{i < j} \nnorm{X_i - X_j}_2$. Ultimately,
for $d$ large enough, permutation recovery succeeds in all replications for all three settings.  
\begin{figure}[t]
    \begin{tabular}{lll}
    \hspace*{15ex}\texttt{psd} & \hspace*{15ex}\texttt{sep} & \hspace*{12ex}\texttt{exp-norm} \\
    \hspace*{-2ex}\includegraphics[width = .32\textwidth]{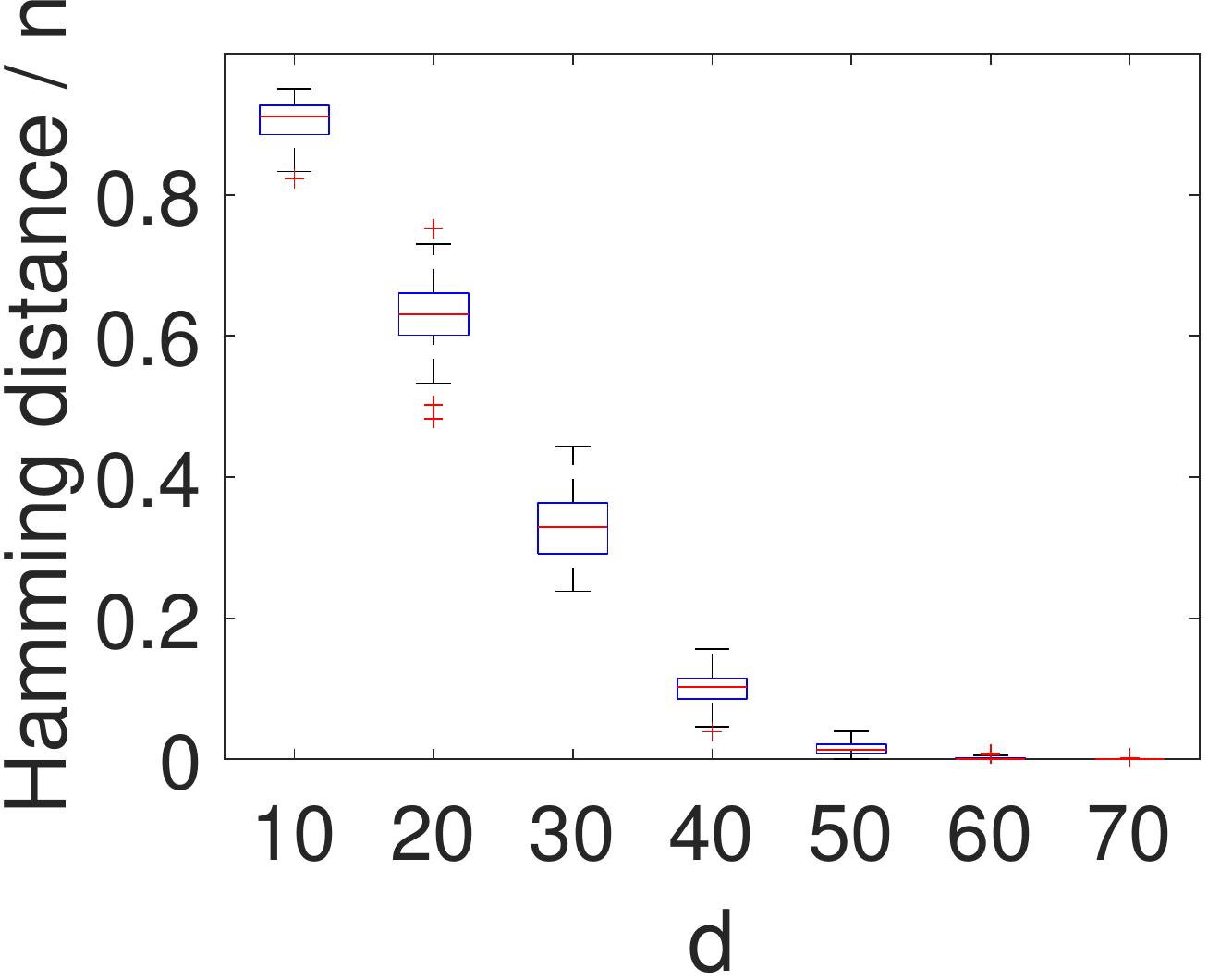} &
    \includegraphics[width = .32\textwidth]{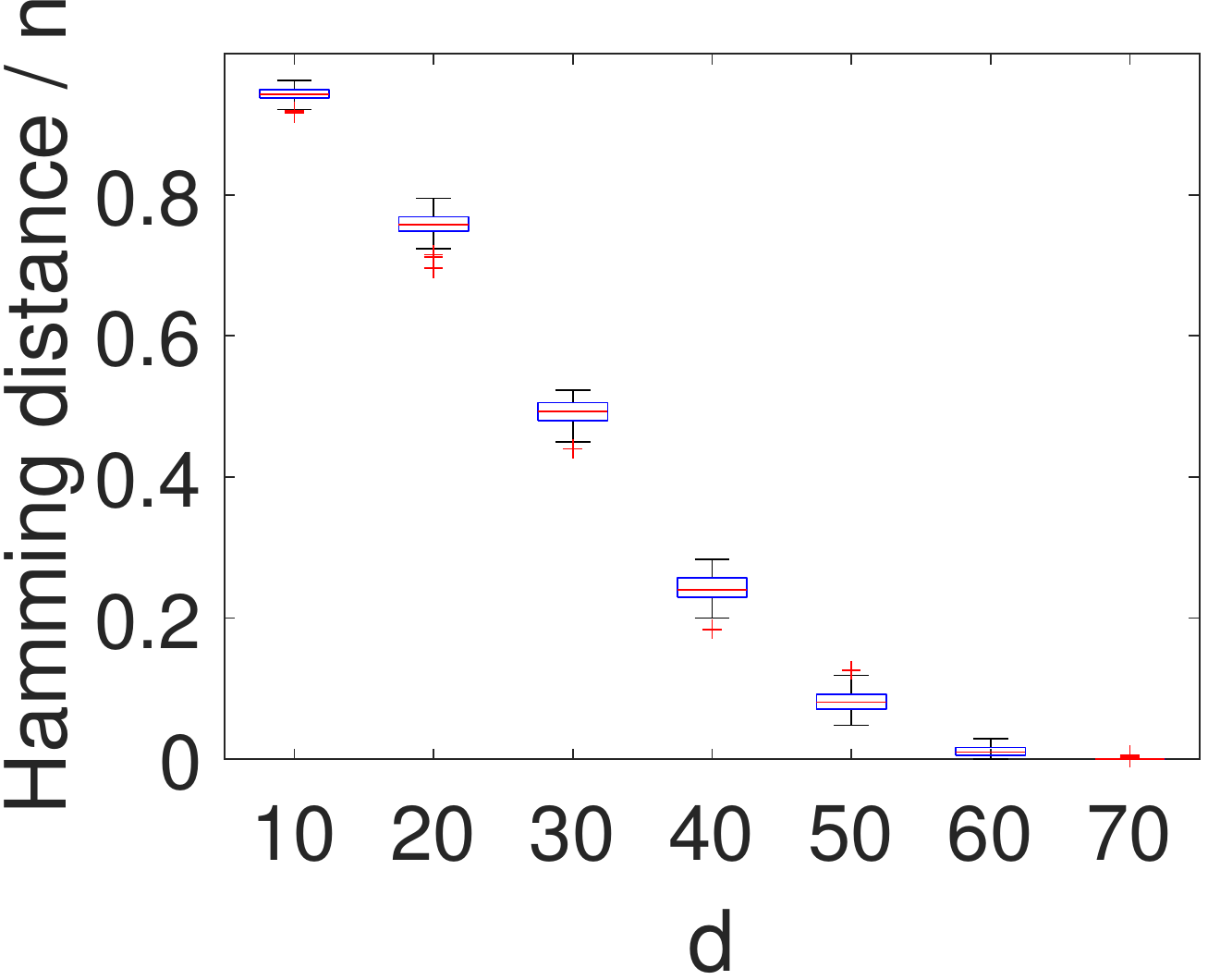} &
    \includegraphics[width = .32\textwidth]{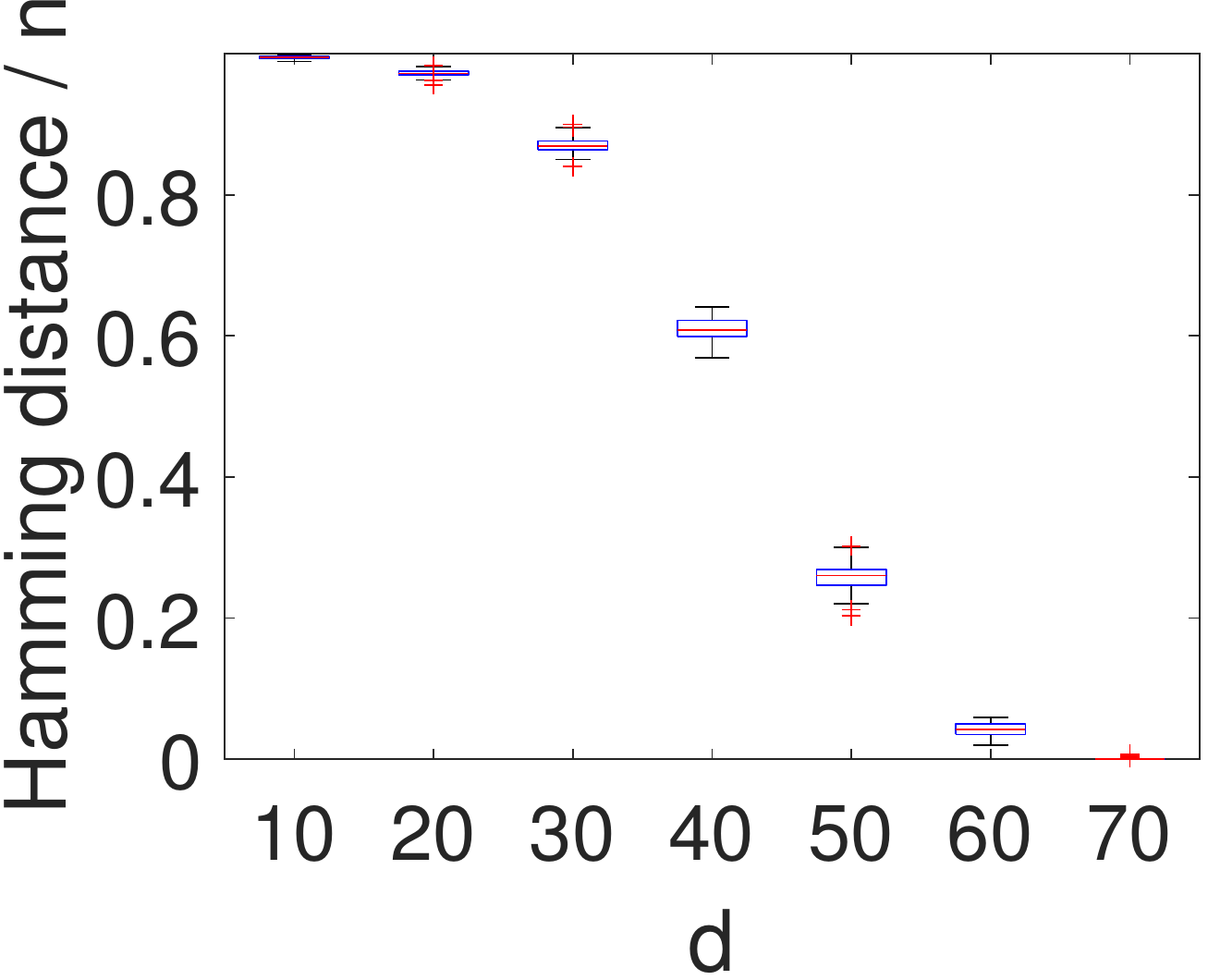}
    \end{tabular}
    \vspace*{-2ex}
    \caption{Boxplots of the scaled Hamming distance $\frac{1}{n} \su \mathbb{I}(\wh{\pi}(i) \neq \pi^*(i))$ between
    $\wh{\pi}$ (from \eqref{eq:LAP_permutation_recovery}) and the ground truth $\pi^*$ based on 
    100 replications for each setting (plot) and each value of $d$ (horizontal axis).}
    \label{fig:permutation_recovery}
\end{figure}
\subsection{Denoising, $d=1$}\label{subsec:numerical_d1}
In this subsection, we compare the performance of the proposed approach with regard to denoising {\bfseries(T2)} 
to two methods proposed in earlier work \cite{Weed2018, Balabdoui2020}. These two competing methods only
discuss the case $d = 1$, hence our comparison is confined to this case. For our comparison, we adopt the five settings
for the function $f^*$ considered in \cite{Balabdoui2020} and depicted in the left panel of Figure \ref{fig:denoising_results_d1}. 
Specifically, these five setting are given by 
\begin{itemize}
\item[1.] \texttt{linear}: $f^*(x) = x$, $x \in [0, 10]$,
\item[2.] \texttt{constant}: $f^*(x) = 0$, $x \in [0,10]$,
\item[3.] \texttt{step2}: $f^*(x) = 2 \mathbb{I}(x \in [0, 5)) + 8 \mathbb{I}(x \in [0, 10])$,
\item[4.] \texttt{step3}: $f^*(x) = 5 \mathbb{I}(x \in [10/3, 20/3)) + 10 \mathbb{I}(x \in [20/3, 10])$,
\item[5.] \texttt{power}: $f^*(x) = -(x-5)^4 \mathbb{I}(x \in (0,5]) + (x-5)^4 \mathbb{I}(x \in [5, 10])$. 
\end{itemize}
The design points $\{ X_i \}_{i = 1}^n$ are sampled i.i.d.~uniformly from the interval $[0,10]$, and 
$Y_i = f^*(X_i) + \epsilon_i$, $1 \leq i \leq n$ (without loss of generality, we choose the permutation
$\pi^*$ as the identity). For the errors, we consider both Gaussian noise with zero mean and unit variance
as well as Laplacian noise with zero mean and scale parameter equal to one. We consider $n = 100$ and $n = 1000$; comparison
for larger $n$ were not considered since the approach in \cite{Balabdoui2020} does not scale favorably with $n$, incurring a runtime complexity of $O(n^2)$ per gradient iteration. Hundred independent replications are performed for each configuration in terms of the setting for 
$f^*$, noise distribution, and sample size. 

The proposed approach (Slawski \& Sen, short \textsf{SS}) is run by solving the approximate NPMLE problem \eqref{eq:Kiefer_Wolfowitz_approx1} with $\mathbb{G}$ chosen as a linearly spaced grid of size $2 \lceil n^{1/2} \rceil$ between $\min_i Y_i$ and $\max_i Y_i$, and the resulting deconvolution estimate $\wh{\nu} = \sum_{j = 1}^p \wh{\alpha}_j \delta_{\wh{\theta}_j}$ is used for the Kantorovich problem \eqref{eq:Kantorovich_finite}. The competitor \textsf{BDD} (initials of the last names of the authors of \cite{Balabdoui2020}) is run based on an in-house implementation of the gradient descent method in that paper, using the 
starting values $\wh{f}(X_{(i)}) = Y_{(i)}$, $1 \leq i \leq n$. Gradient descent is performed with constant step size; for the 
sake of fair comparison, six different values for the step size between $0.01$ and $0.5$ are considered, and for each
configuration we report the result of the specific step size achieving minimum average error over the respective replications. 
The competitor \textsf{RW} (Rigollet \& Weed, \cite{Weed2018}) is run based on an in-house implementation of a subgradient
descent method to solve the (discretized) Wasserstein deconvolution problem considered in that paper (cf.~$\S$2.2 therein). 
The size of the quantization alphabet is taken as $\lceil 2 \sqrt{n} \rceil$, linearly spaced between $\min_i Y_i$ and $\max_i Y_i$. 
Each optimal transport problem required for subgradient computation is approximated via Sinkhorn's algorithm \cite[][$\S$4]{COT2019} with regularization parameter $\varepsilon = 0.1$. As for \textsf{BDD}, we consider six different values for the step size between
$5 \cdot 10^{-5}$ and $2 \cdot 10^{-3}$, and select the results achieving minimum average error over these six choices. 
\vskip1.5ex
\noindent {\bfseries Results}. The results of our comparison are visualized in Figure \ref{fig:denoising_results_d1} via boxplots showing the mean squared denoising errors over 100 replications. The general picture is that \textsf{BDD} achieves the best empirical performance (with optimized step size), while the performance of the proposed approach \textsf{SS} is often on par with \textsf{BDD}. In our comparison, the relative performance of  \textsf{SS} is worse for ``smooth" $f^*$ (settings \texttt{linear} and \texttt{power}). By contrast, \textsf{RW}  performs rather poorly for the settings \texttt{constant}, \texttt{step2}, and \texttt{step3}. Somewhat surprisingly, \texttt{RW} does not exhibit any noticeable decrease in error as the sample size is increased from $100$ to $1000$ with the exception of 
setting \texttt{linear} and Gaussian errors. Despite careful monitoring of convergence and inspection of potential computational issues, it is quite well possible that the performance of \textsf{RW} can be improved substantially with a refined implementation\footnote{The authors of that method did not publish their implementation/code.} since in
fact all three approaches compared herein follow rather similar rationales, and gaps in performance are thus not expected.

\begin{figure}
    \begin{tabular}{lll}
   $\begin{array}{c} \\[-3ex] \includegraphics[height = 0.12\textheight]{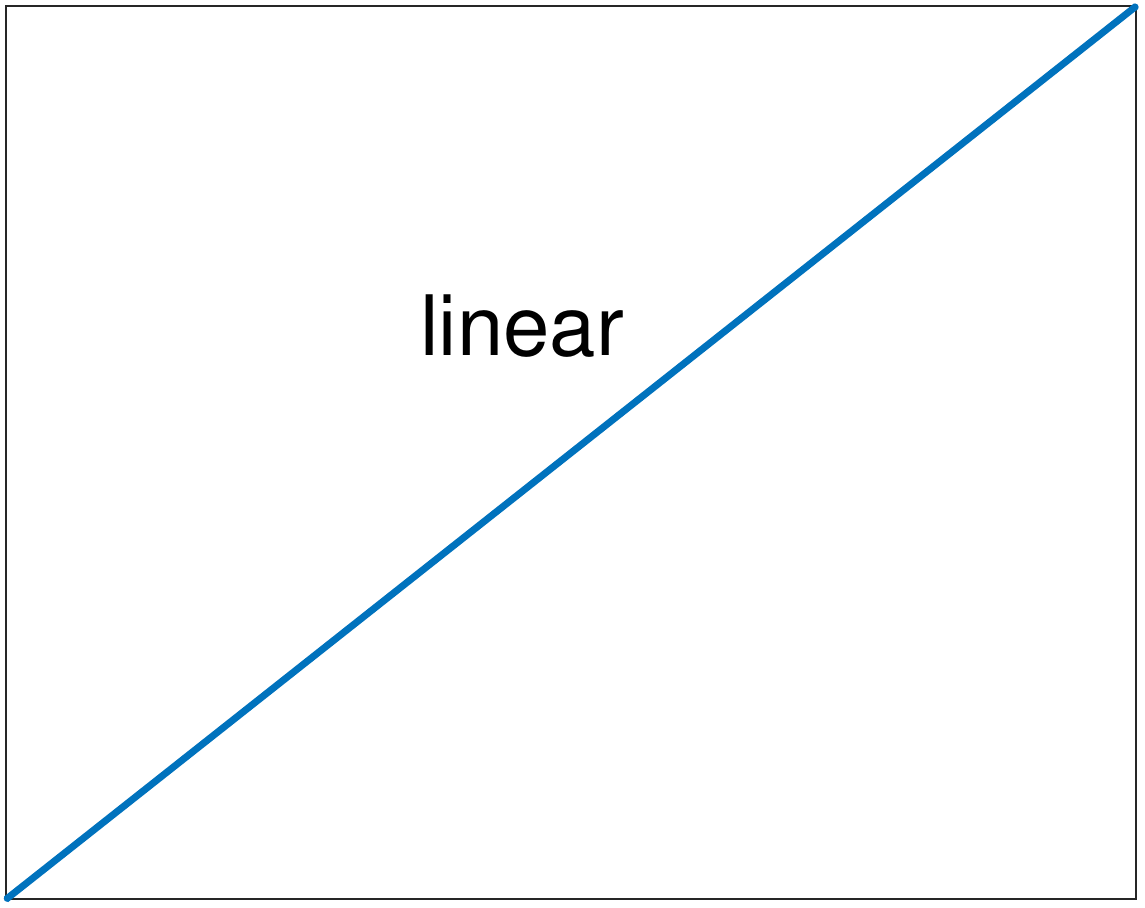} \end{array}$ & $\begin{array}{l} \hspace{8ex} \text{Gaussian errors} \\ \includegraphics[height = 0.16\textheight]{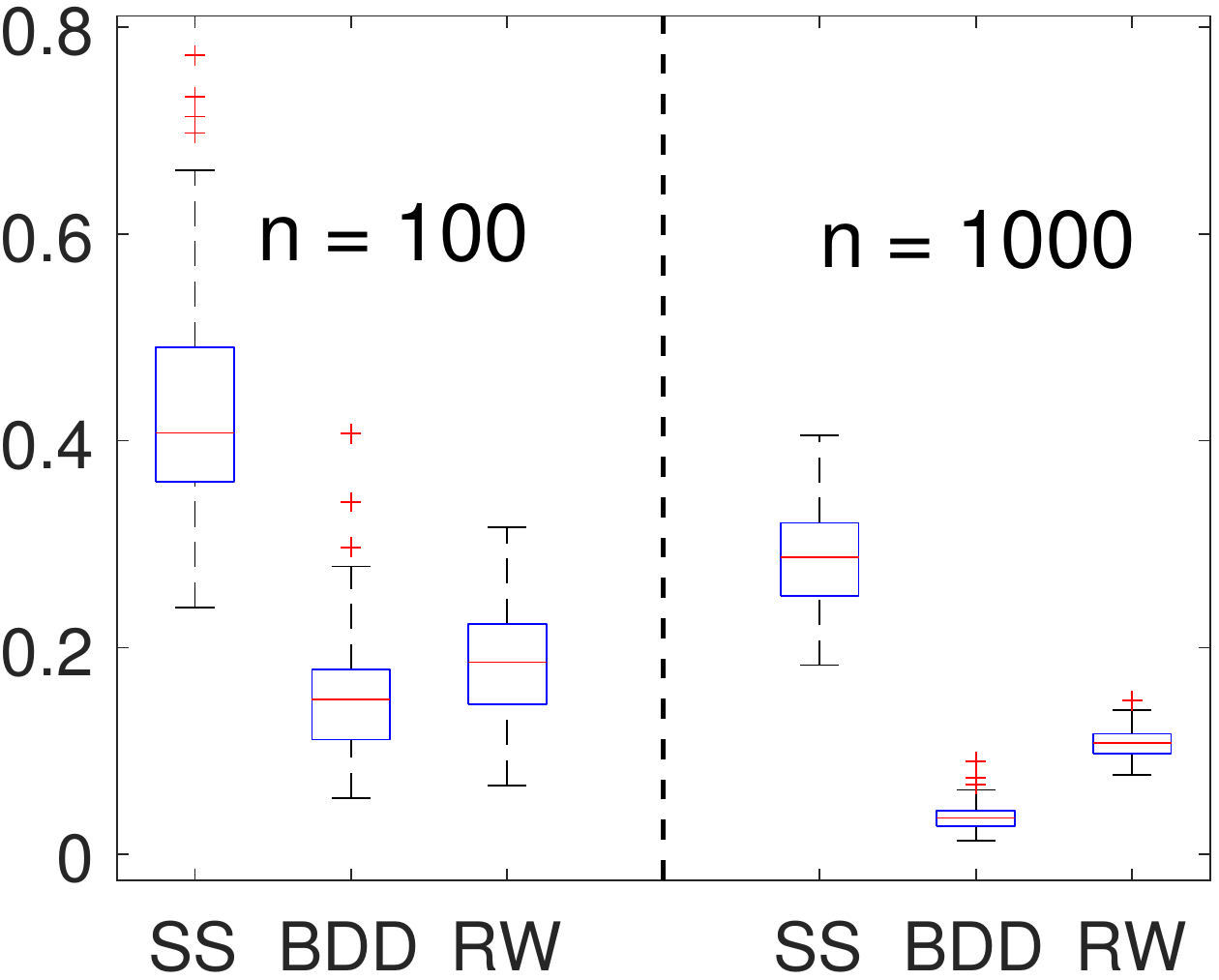} \end{array}$ &  $\begin{array}{l} \hspace{8ex} \text{Laplace errors} \\ \includegraphics[height = 0.16\textheight]{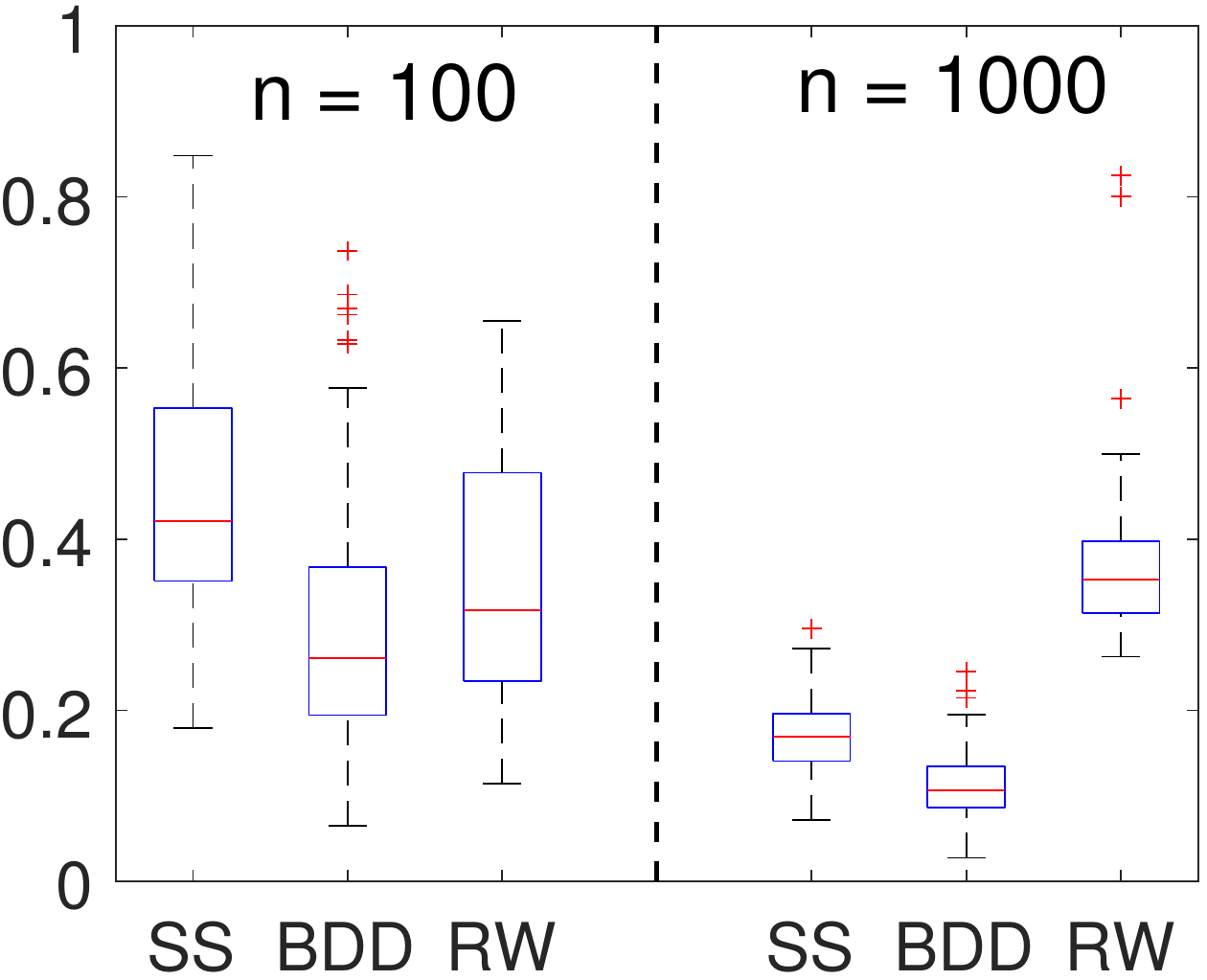} \end{array}$ \\
    
    $\begin{array}{c} \\[-28ex] \includegraphics[height = 0.12\textheight]{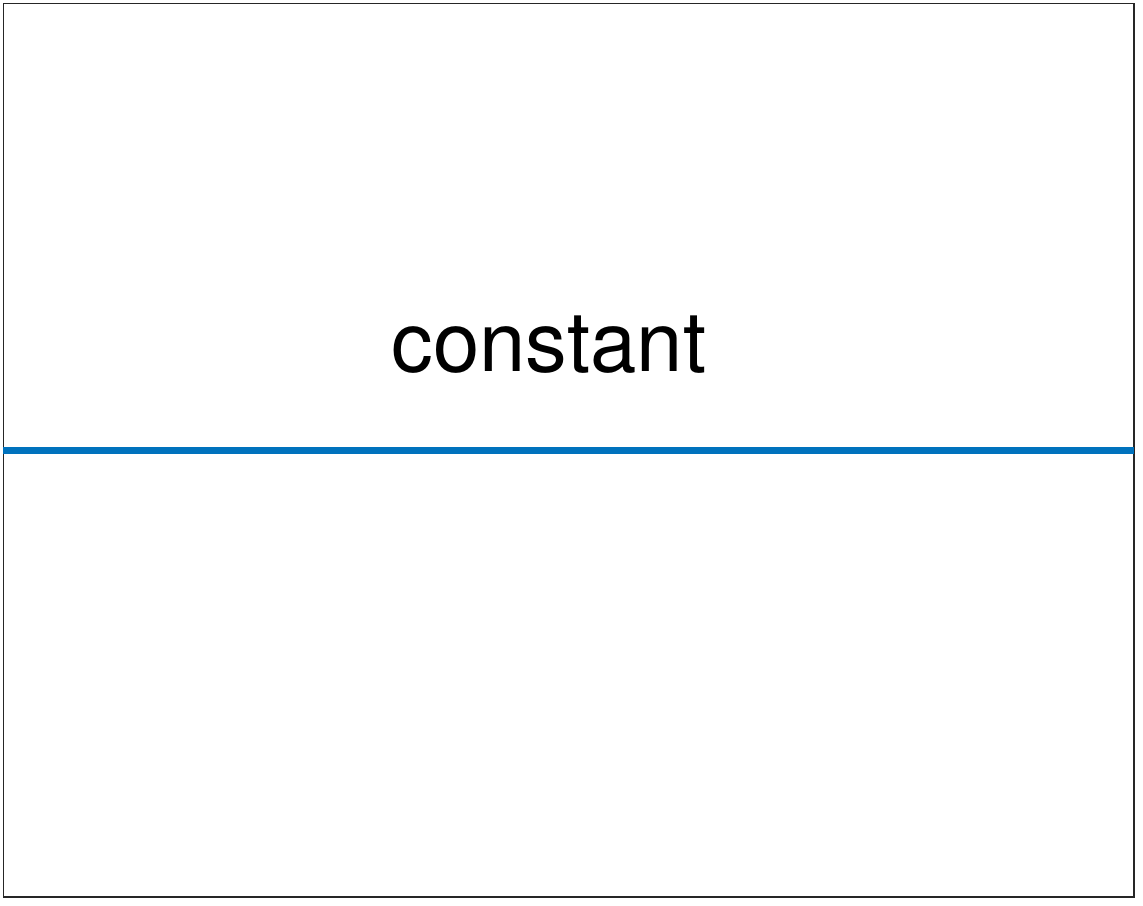} \end{array}$ &  \includegraphics[height = 0.16\textheight]{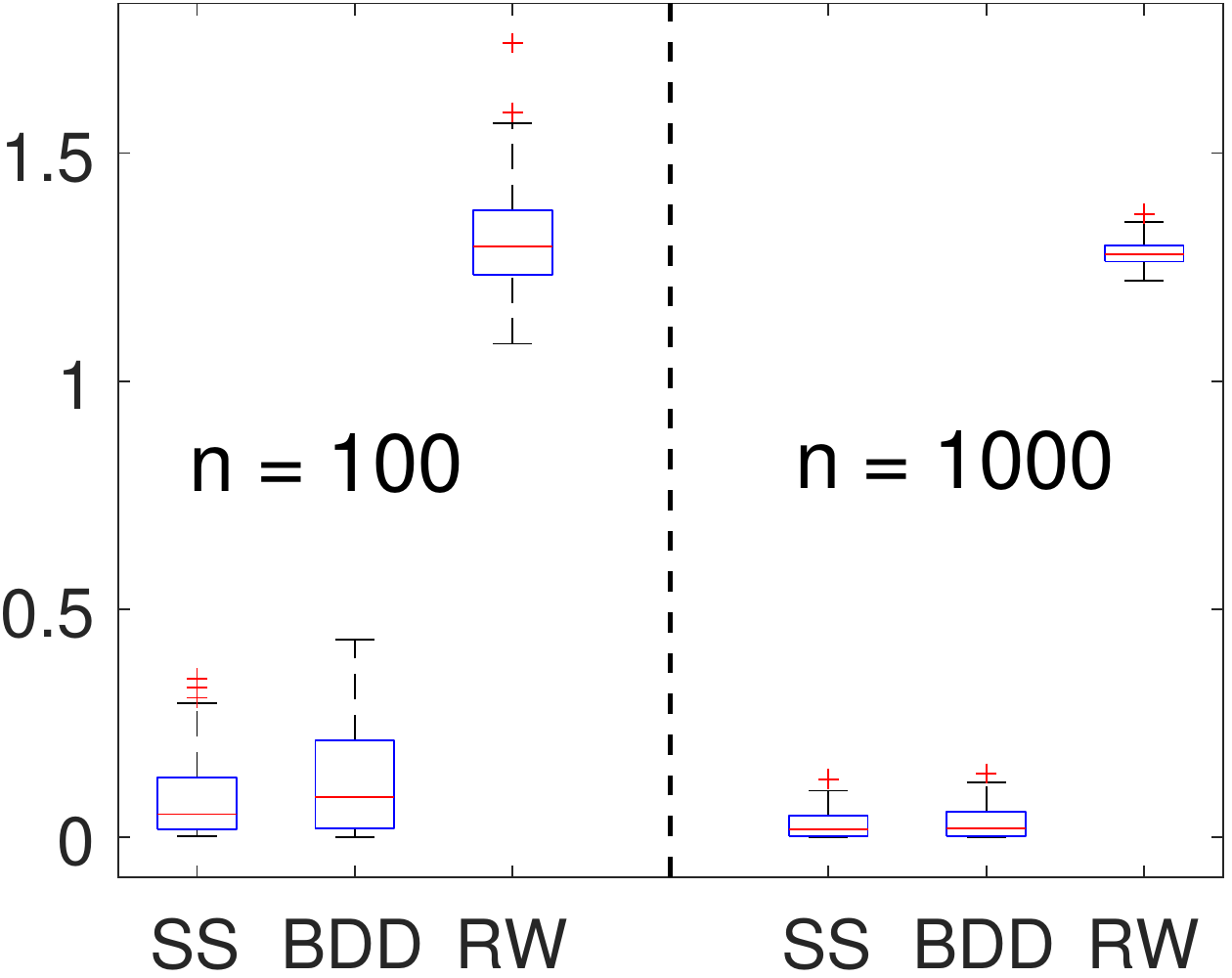} &  \includegraphics[height = 0.16\textheight]{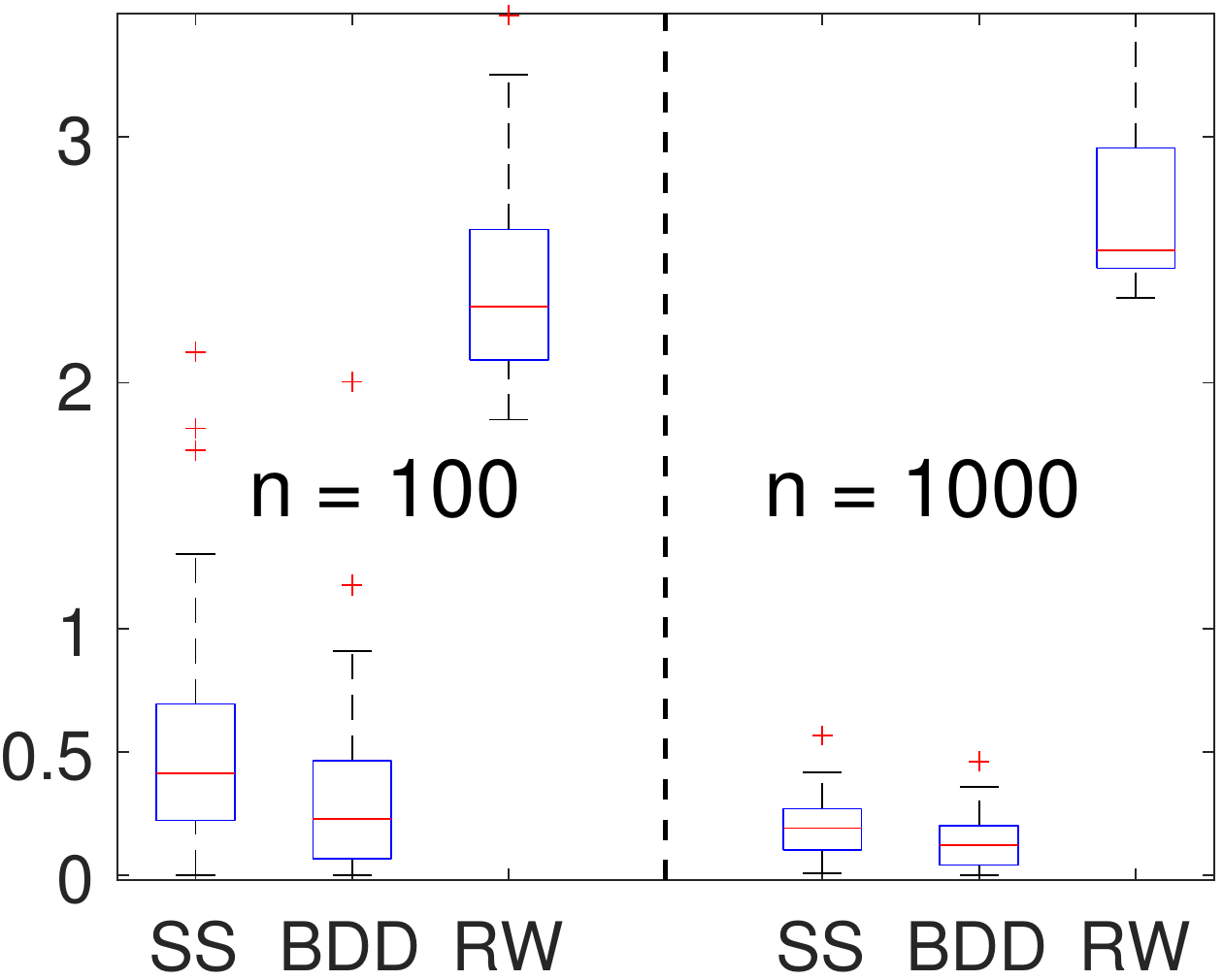} \\
     
      $\begin{array}{c} \\[-28ex] \includegraphics[height = 0.12\textheight]{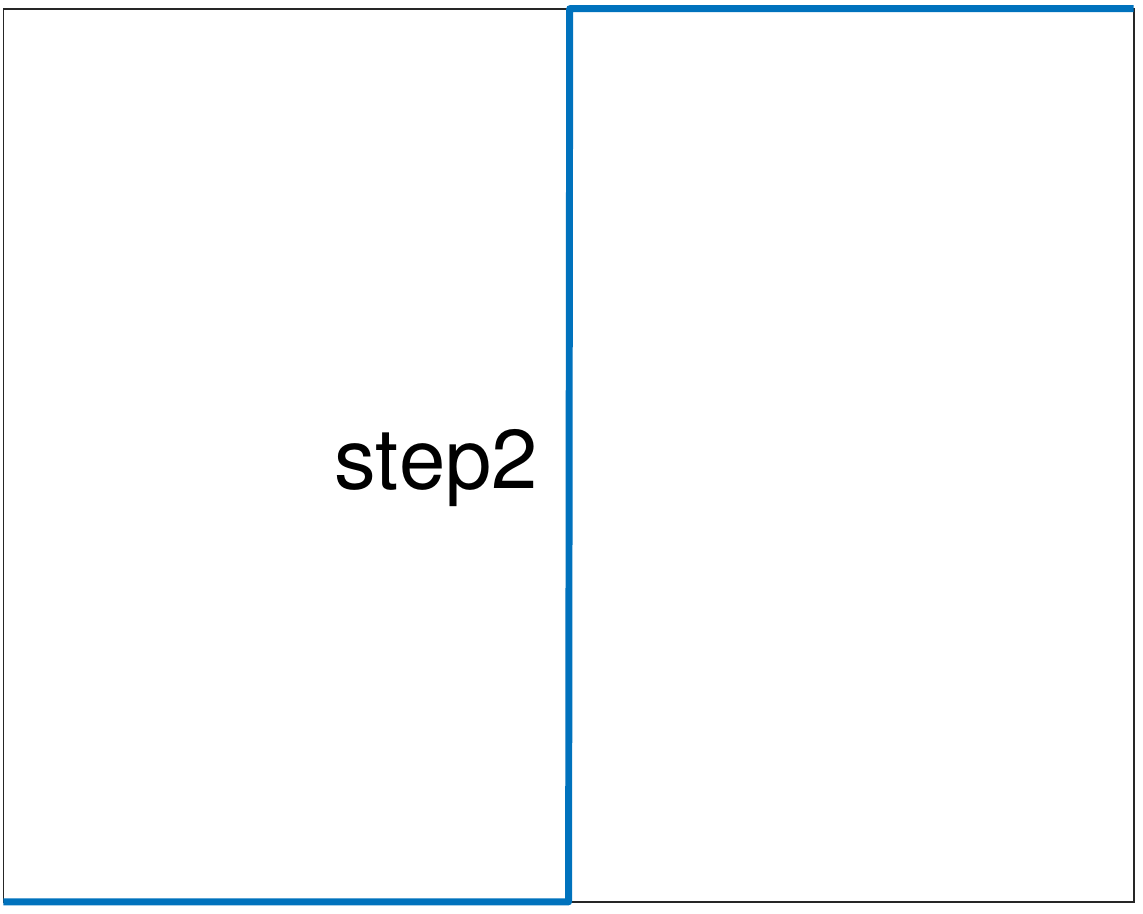} \end{array}$ & \includegraphics[height = 0.16\textheight]{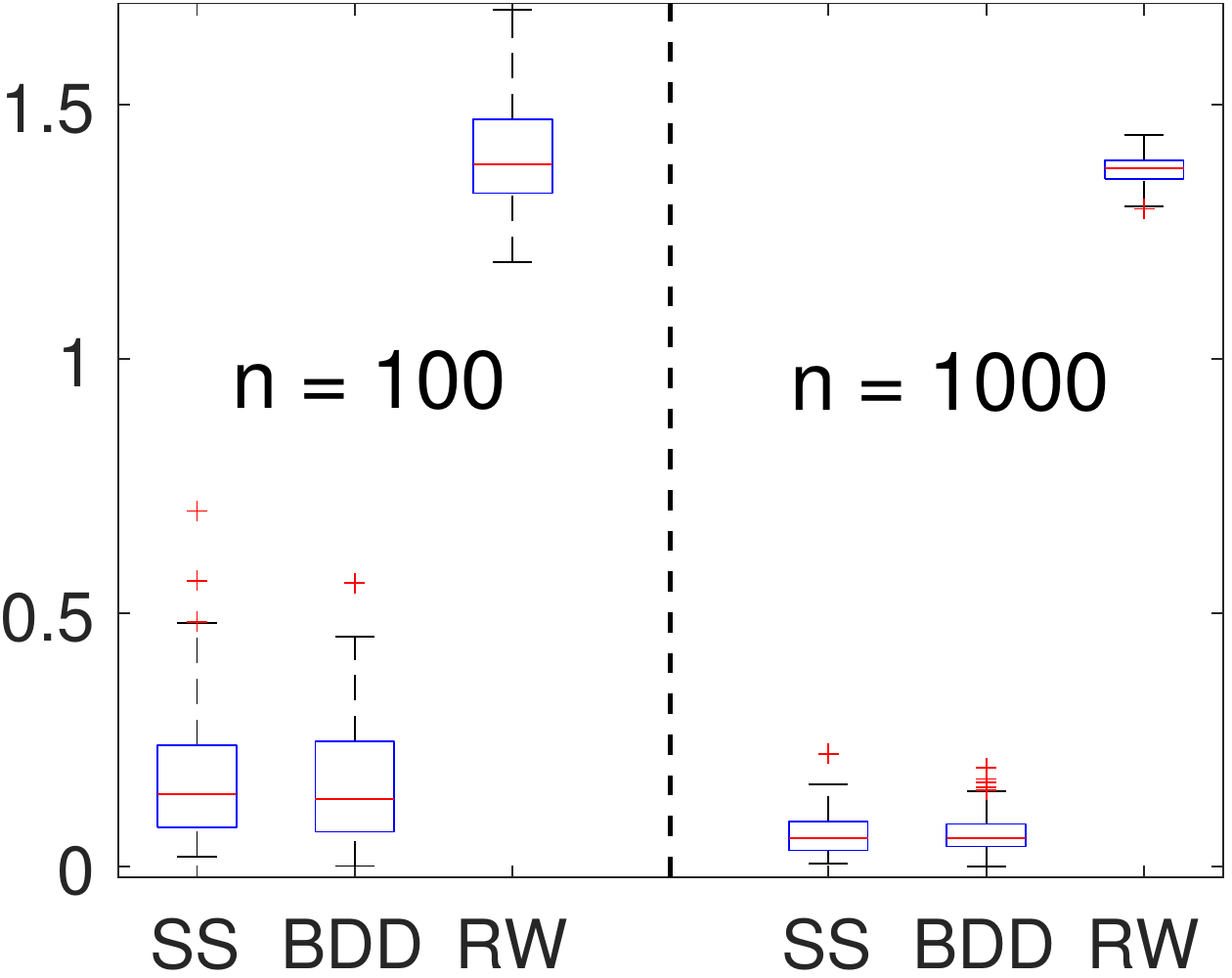} &  \includegraphics[height = 0.16\textheight]{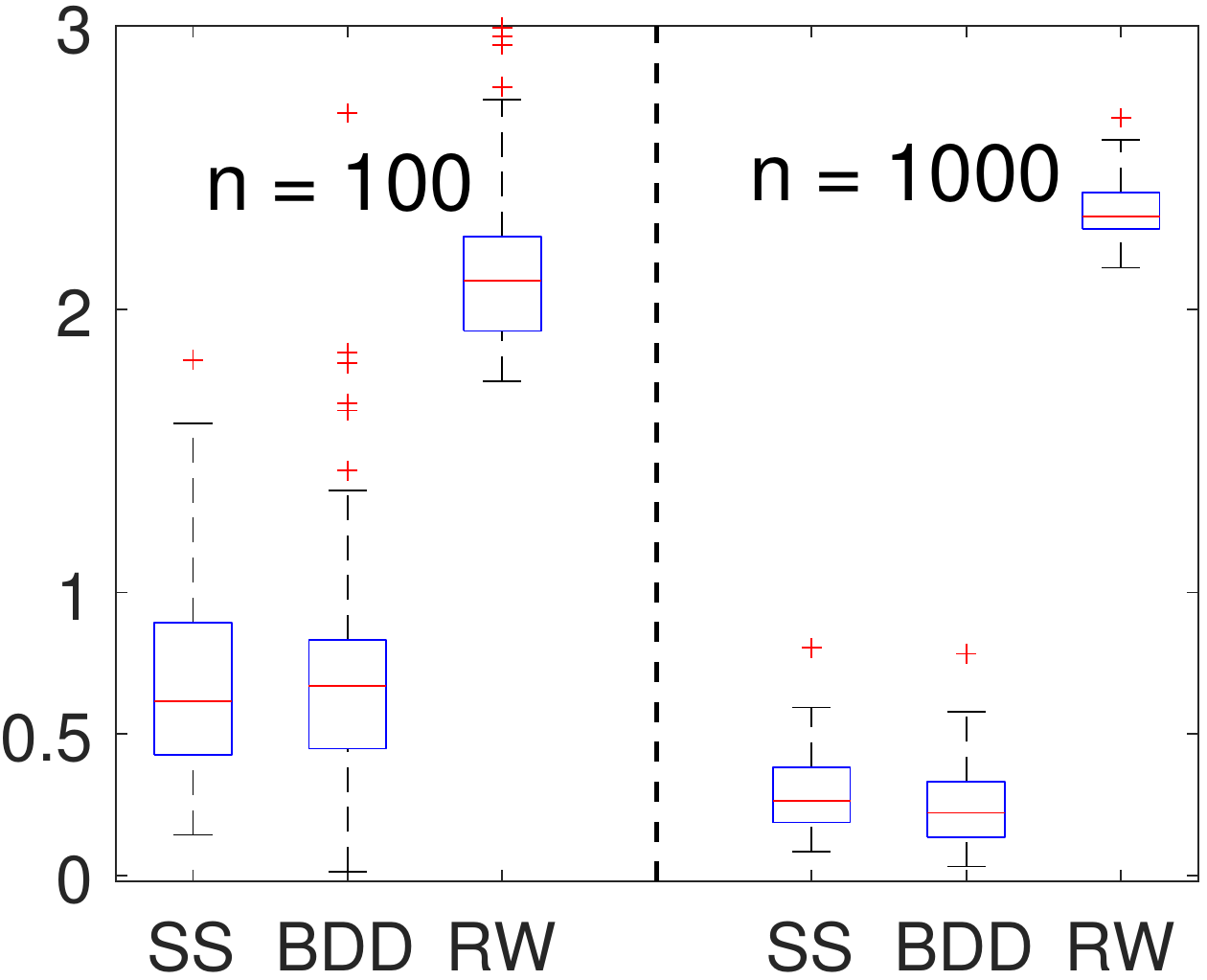} \\
     
     $\begin{array}{c} \\[-28ex] \includegraphics[height = 0.12\textheight]{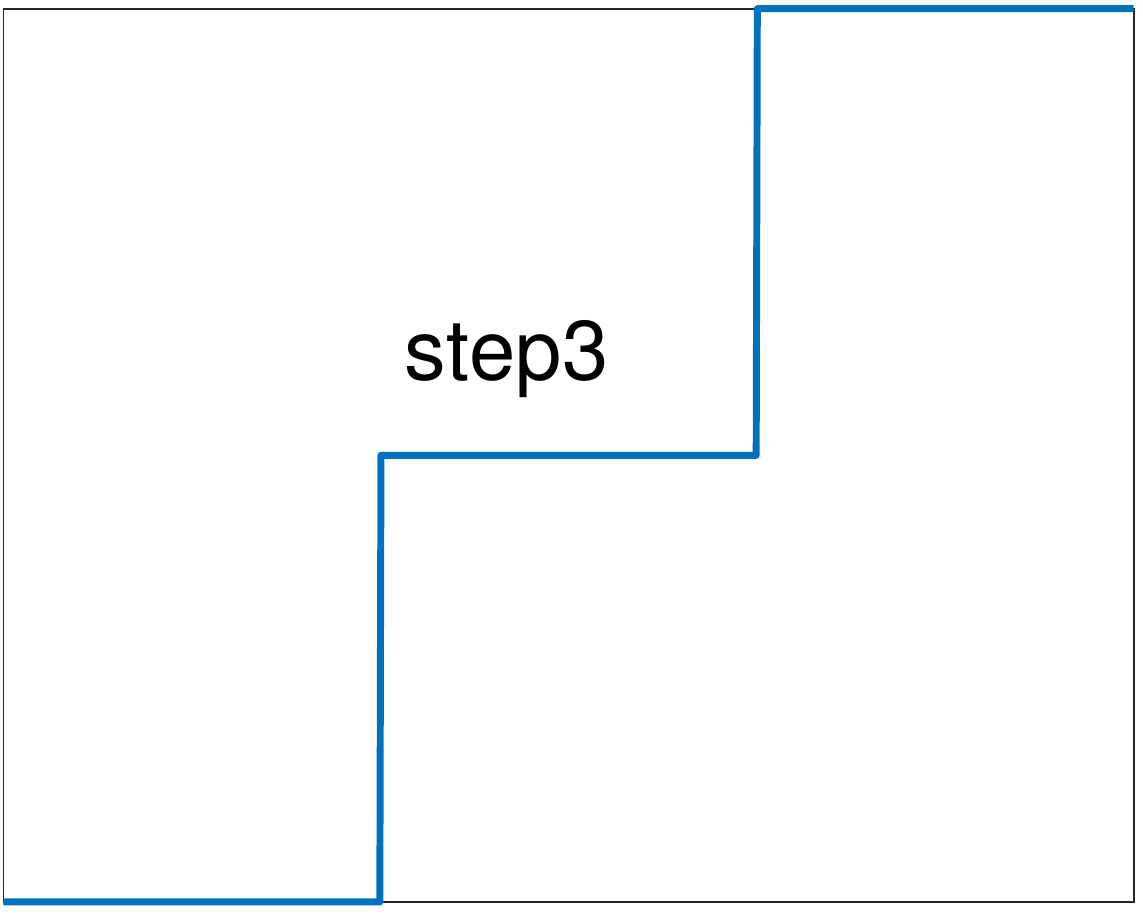} \end{array}$ &
     \includegraphics[height = 0.16\textheight]{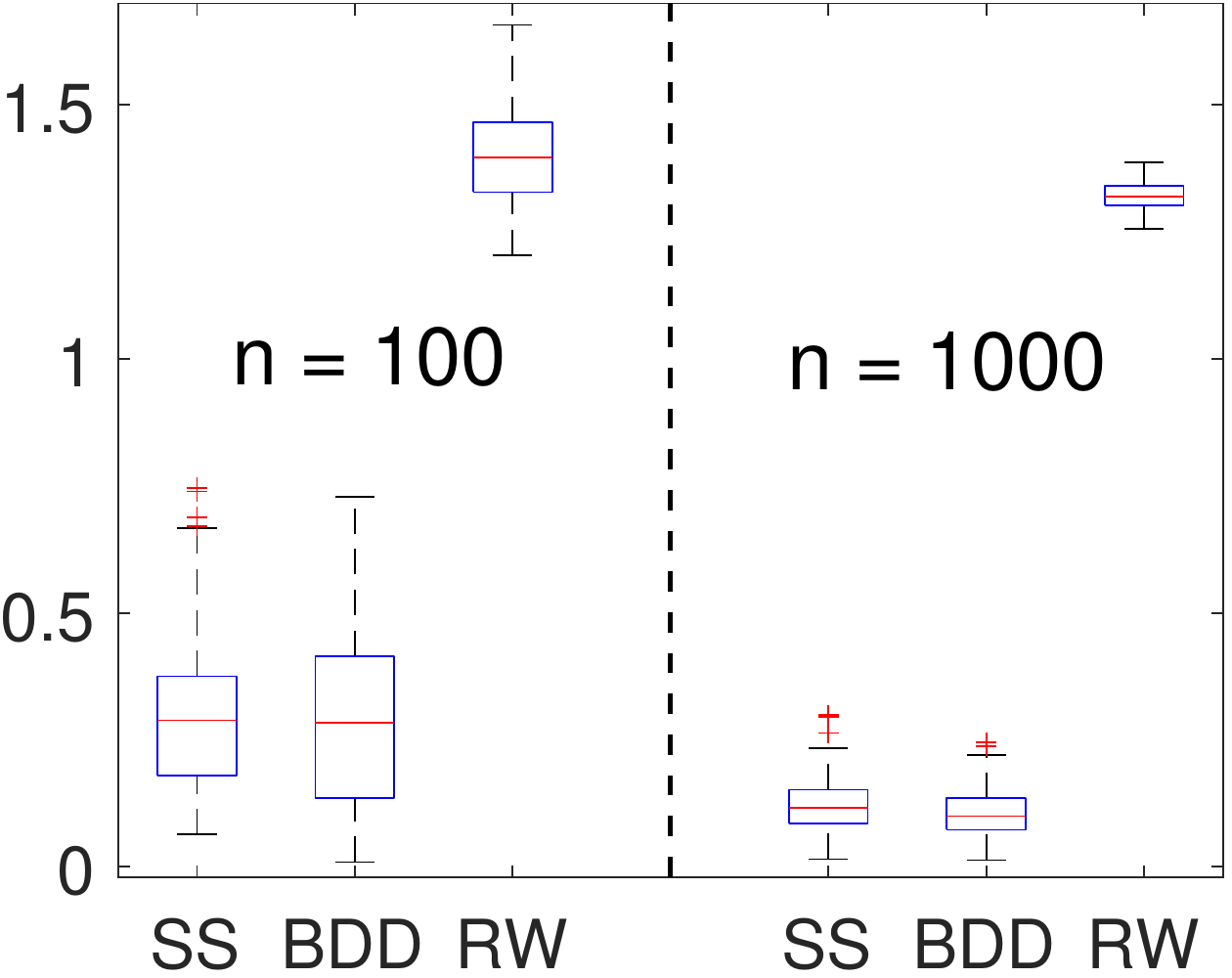} &  \includegraphics[height = 0.16\textheight]{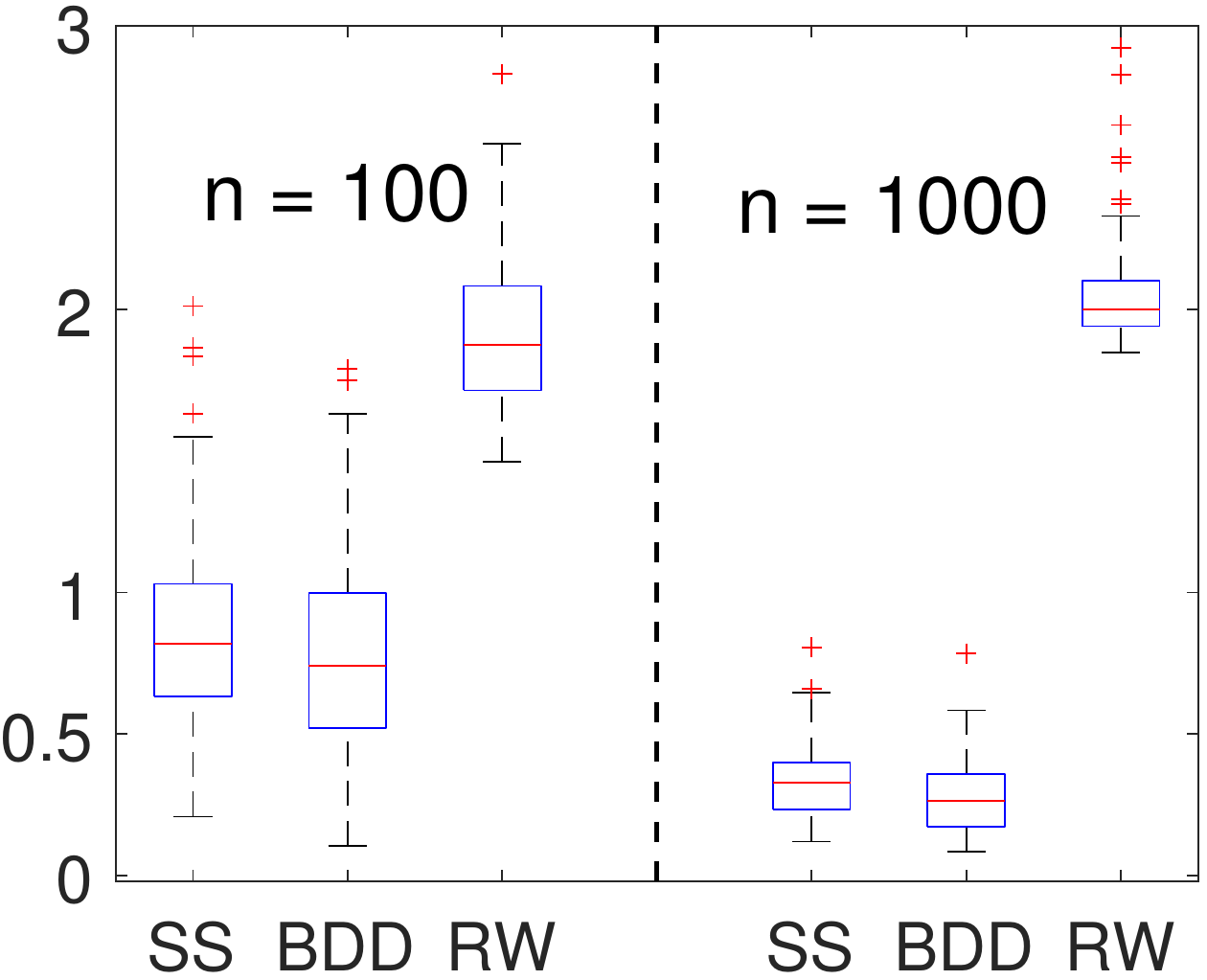} \\
     
     $\begin{array}{c} \\[-28ex] \includegraphics[height = 0.12\textheight]{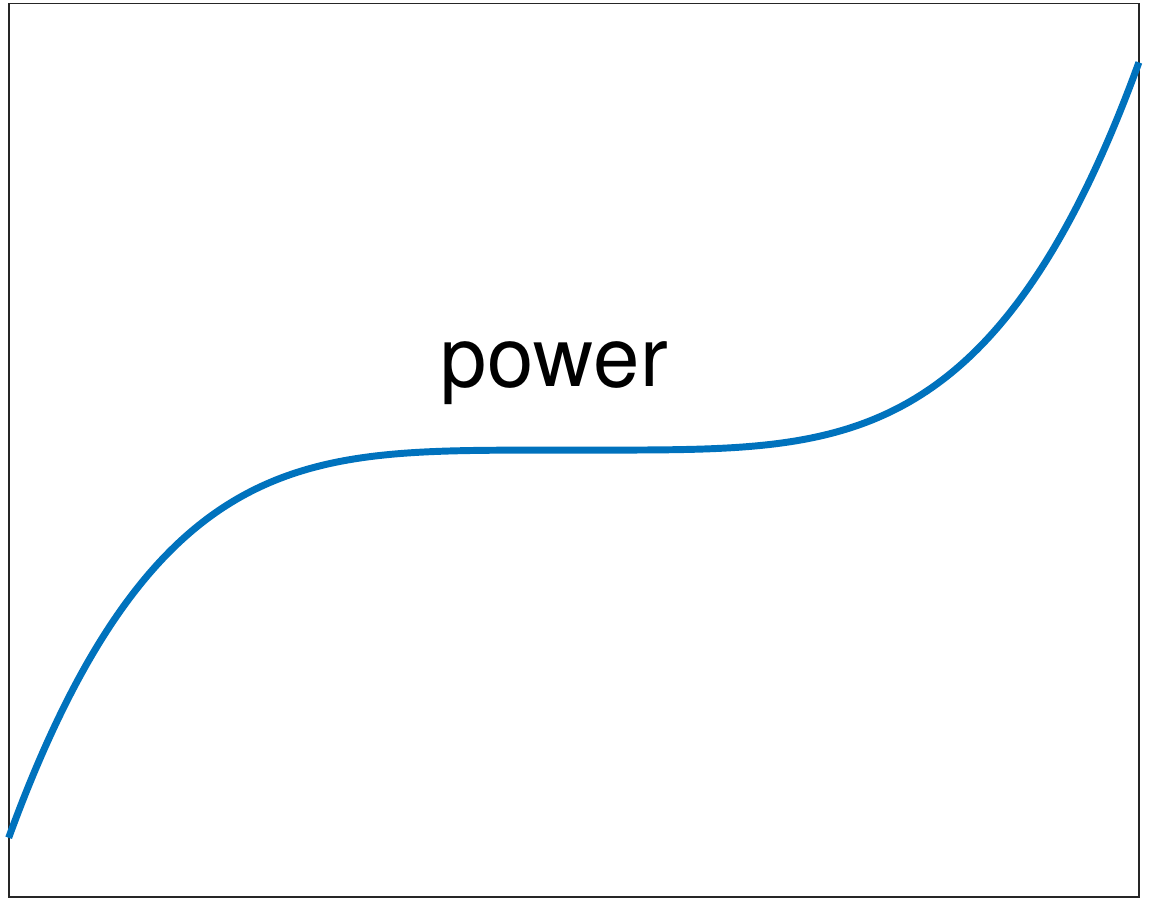} \end{array}$ &
     \includegraphics[height = 0.16\textheight]{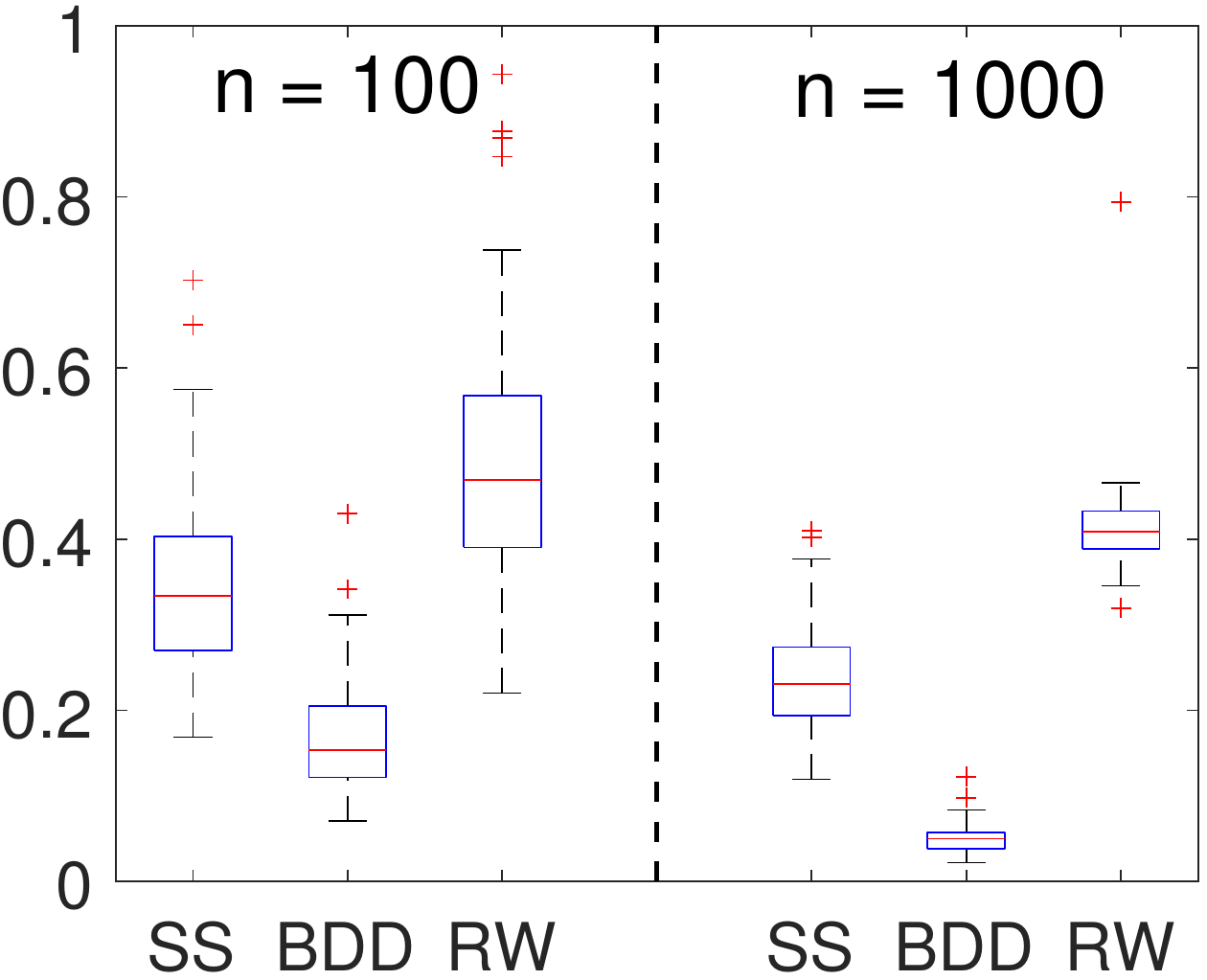} &  \includegraphics[height = 0.16\textheight]{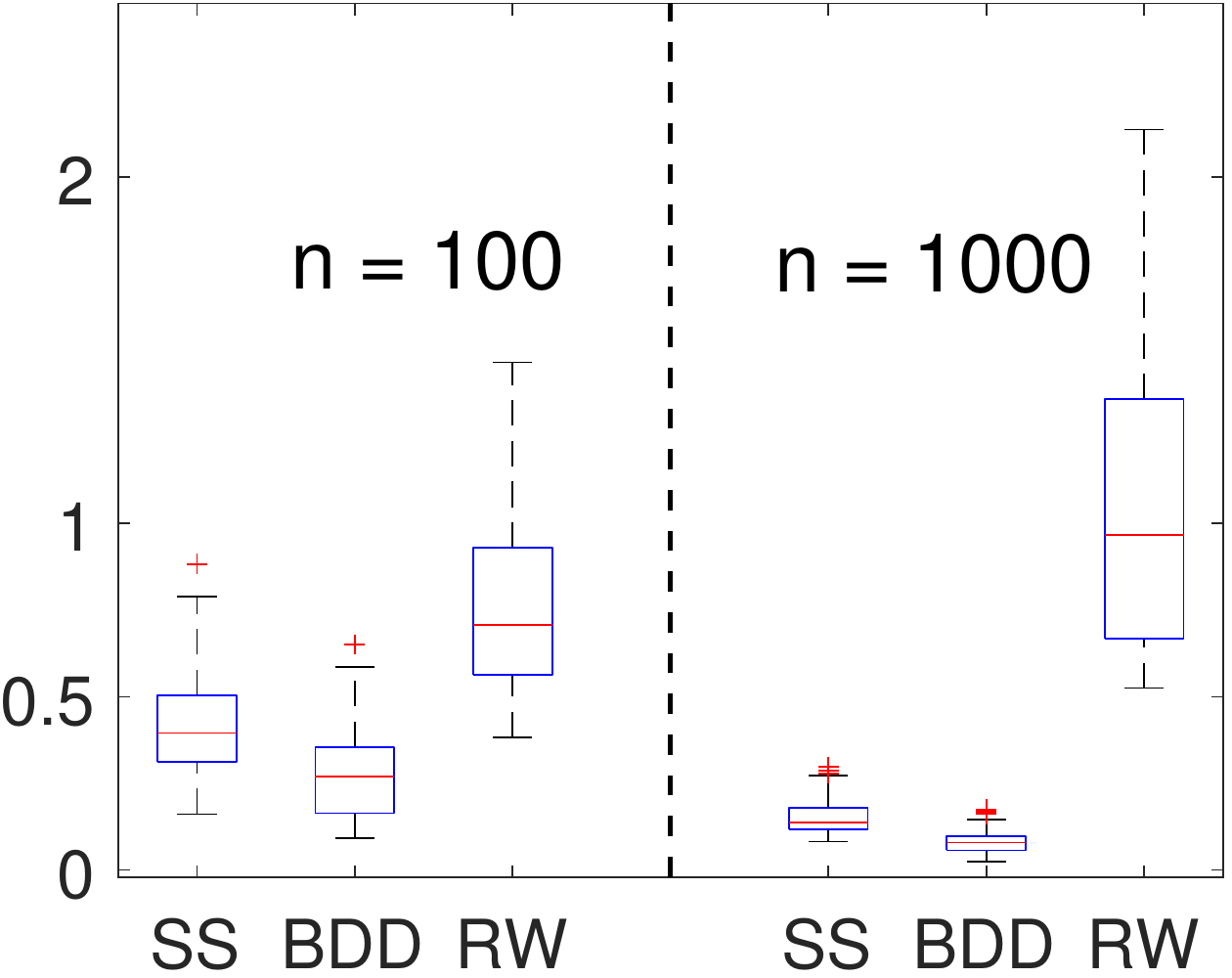} 
    \end{tabular}
    \caption{Results of the comparison of the three approaches under consideration for the denoising problem {\bfseries (T2)}. Left:
    underlying function $f^*$. Middle and right: Boxplots of mean squared errors for Gaussian and Laplace errors, respectively.}
    \label{fig:denoising_results_d1}
\end{figure}

\subsection{Denoising, $d > 1$}\label{subsec:numerical_generald}
This subsection is intended to corroborate and complement aspects of our theoretical results in $\S$\ref{subsec:denoising} regarding task {\bfseries (T2)} for general dimension. The competitors in the preceding section were developed for the case $d = 1$, hence
we confine ourselves to the proposed method. 

We generate data following the permuted regression setup \eqref{eq:permuted_regression}. The sample $\{ X_i \}_{i = 1}^n$ is sampled uniformly from the unit Euclidean ball in $\R^d$ ($d = 2, 4$), and subsequently we generate
$Y_i = f^*(X_i) + \sigma \epsilon_i$, $1 \leq i \leq n$, where $\sigma = 1/16$ and the $\{ \epsilon_i \}_{i = 1}^n$ are sampled i.i.d.~from the $N(0, I_d)$-distribution and alternatively from the multivariate Laplace distribution\footnote{Specifically, we generate $\epsilon_i = g_i \cdot \xi_i$, where $g_i \sim N(0, I_d)$-distribution and $\xi_i \sim \text{Exp}(1)$, $1 \leq i \leq n$, where $\text{Exp}(1)$ denotes the
exponential distribution with unit scale.}. The settings considered for $f^*$ are summarized in Table \ref{tab:simulation_settings}.
For the sample size, we consider $n = 2^8, 2^9, \ldots, 2^{12} = 4096$ (and in some selected settings $2^{13}$) in anticipation 
of slow rates as indicated by the results in $\S$\ref{subsec:denoising}. 

The proposed approach is run as follows: we solve the approximate NPMLE problem \eqref{eq:Kiefer_Wolfowitz_approx2} with $\mathbb{G} = \{ Y_i \}_{i = 1}^n$ equipped
with knowledge of $\varphi_{\sigma}$, and use the resulting deconvolution estimate $\wh{\nu}$ in the Kantorovich problem \eqref{eq:Kantorovich_finite} to obtain $\{ \wh{f}(X_i) \}_{i = 1}^n$. We then report the normalized MSE 
$\frac{1}{n \sigma^2} \su \nnorm{\wh{f}(X_i) - f^*(X_i)}_{2}^2$. The results depicted in Figure \ref{fig:results_denoising_generald} represent averages over 100 independent replications, with bars indicating $\pm$ standard error.   
\begin{table}[]
    \centering
    \begin{tabular}{|l|l|l|l|l|l|}
    \hline
                   & \texttt{cluster} & \texttt{linear} & \texttt{separable} & \texttt{sphere} & \texttt{radial} \\ 
                   \hline
    $\psi_{f^*}(x)$   & $\max \limits_{1 \leq j \leq k} \nscp{a_j}{x}$ & $\frac{1}{2} x^{\T} \sum_{j = 1}^k v_j v_j^{\T} x$ & $\frac{2}{3} \sum_{j = 1}^d (x_j + 1)^{3/2}$      &  $\nnorm{x}_2$ &  $\exp(\nnorm{x}_2^2/2)$                 \\[1ex]
    $f^*(x)$          & $a_{x}$     & $\sum_{j = 1}^k \nscp{x}{v_j}  v_j$  & $\sum_{j = 1}^d (x_j + 1)^{1/2}$  &  $\frac{x}{\nnorm{x}_2}$ &  $x \cdot \exp(\nnorm{x}_2^2/2)$ \\
    \hline
    \end{tabular}
    \caption{Summary of the simulation settings considered in $\S$\ref{subsec:numerical_generald}. In the 2nd column from the left, 
    $a_x$ is short for $\{a_j: \scp{x}{a_j} = \psi_{f^*}(x) \}$.}
    \label{tab:simulation_settings}
\end{table}
\vskip1.25ex
\noindent {\bfseries Results}. First, the results shown in Figure \ref{fig:results_denoising_generald} confirm that 
the rates are indeed slow as expected in light of the results in $\S$\ref{subsec:denoising}, with an error decay that is linear on a log-log scale for some instances (corresponding to a polynomial rate in $n$) and noticeably sublinear for others. While the discussion at the 
end of $\S$\ref{subsec:denoising} suggests that Laplacian errors will yield faster rates, this is not confirmed by
our simulations; the observed denoising error is often comparable if not higher than for Gaussian errors. Moreover, 
while the analysis in $\S$\ref{subsec:denoising} requires strong convexity of $\psi_{f^*}$, several of the 
settings considered here (\texttt{cluster}, \texttt{linear} with $k < d$ and \texttt{sphere}) do not comply with
that assumption, yet the empirical results shown here do not indicate that the lack of strong convexity prompts
a substantially different scaling of the denoising error. In fact, the setting \texttt{cluster} corresponds 
to a clustering problem with a finite number of clusters, i.e., the underlying problem is parametric 
rather than non-parametric, and one would hence intuitively expect even faster rates (this intuition
is confirmed by our results). In a similar vein, we also observe smaller errors if the ``intrinsic dimension"
of the problem is smaller than the ambient dimension: in the setting \texttt{linear}, the parameter $k$ reflects
the intrinsic dimension, and Figure \ref{fig:results_denoising_generald} indeed shows that the denoising 
error drops as $k$ is reduced. For several settings with $d = 4$ (in particular \texttt{sphere} and \texttt{radial}) the denoising error is essentially flat and starts decreasing only after $n$ becomes rather large. This behavior is not understood
at this point; one possible explanation for the setting \texttt{sphere} might be the lack of strong convexity in conjunction with the absence of additional structure such as in the setting \texttt{cluster}. 
\begin{figure}[!h]
    \begin{tabular}{lll}
    $\begin{array}{l}
    \hspace*{6ex} \text{\texttt{clusters}, $d=2$}\\
    \includegraphics[width = .32\textwidth]{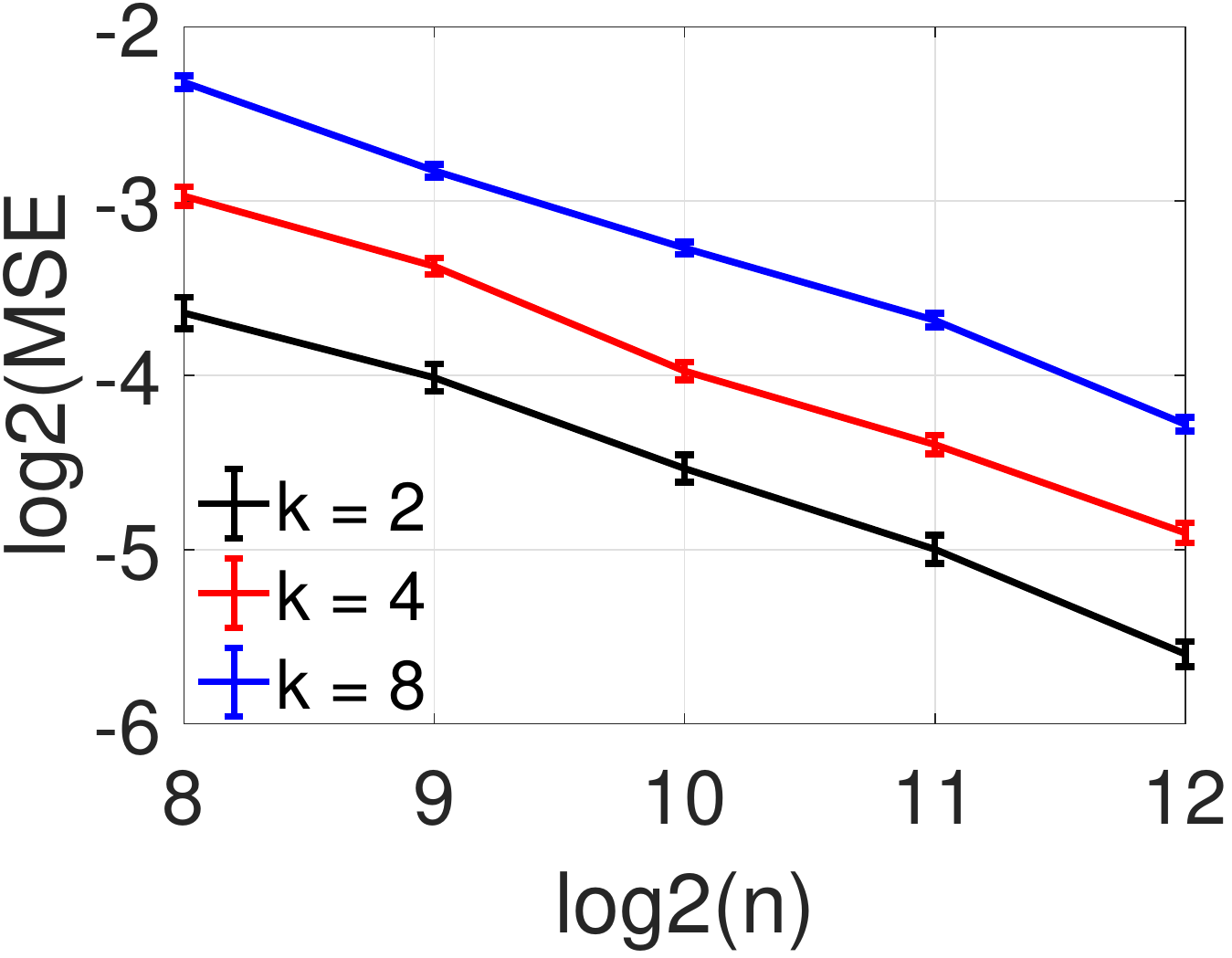} 
    \end{array}$ & \hspace*{-5ex} $\begin{array}{l}
    \hspace*{9ex} \text{\texttt{linear}, $d=2$}\\
    \includegraphics[width = .32\textwidth]{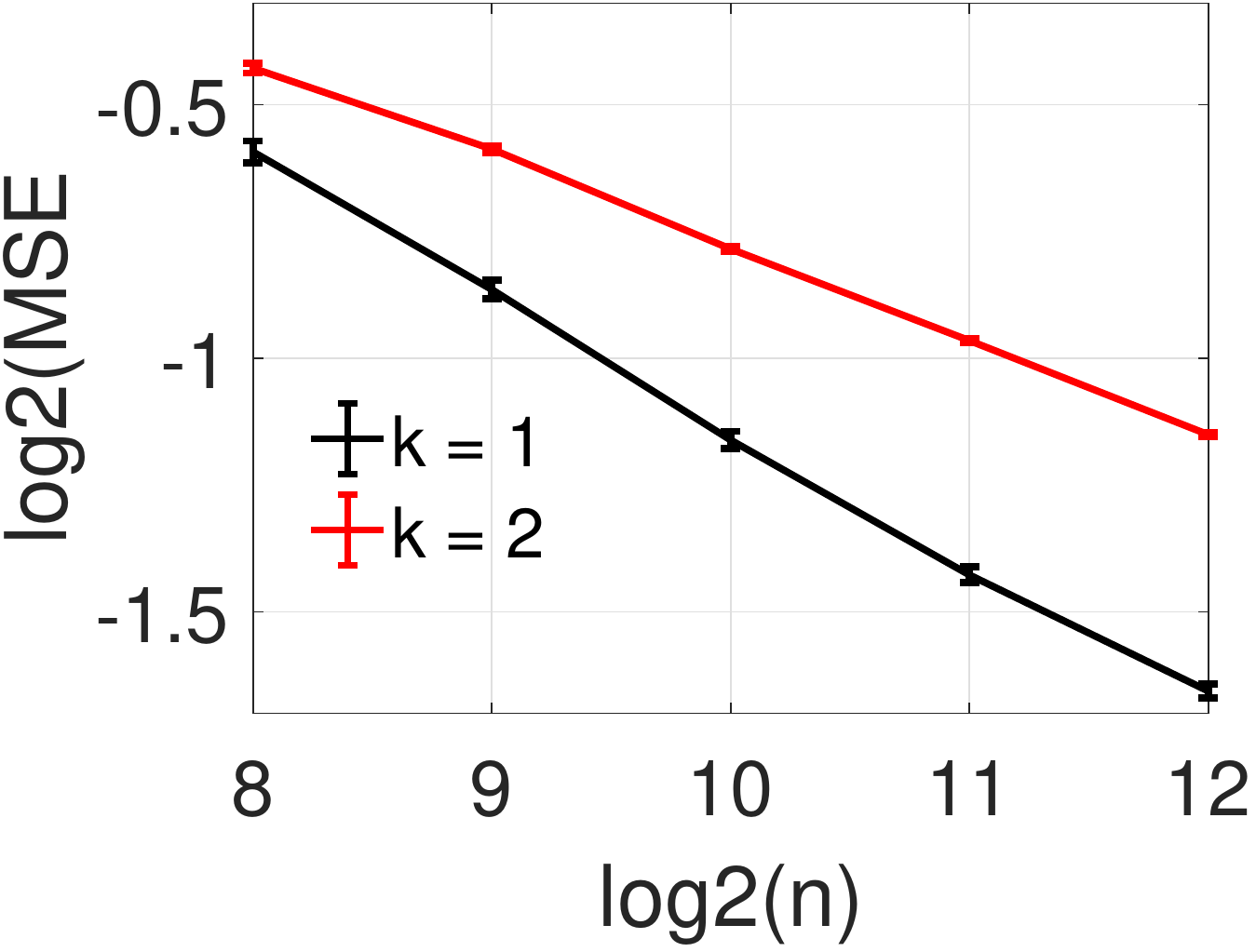} 
    \end{array}$ & \hspace*{-6ex} $\begin{array}{l}
    \hspace*{9ex} \text{\texttt{separable}}\\
    \includegraphics[width = .32\textwidth]{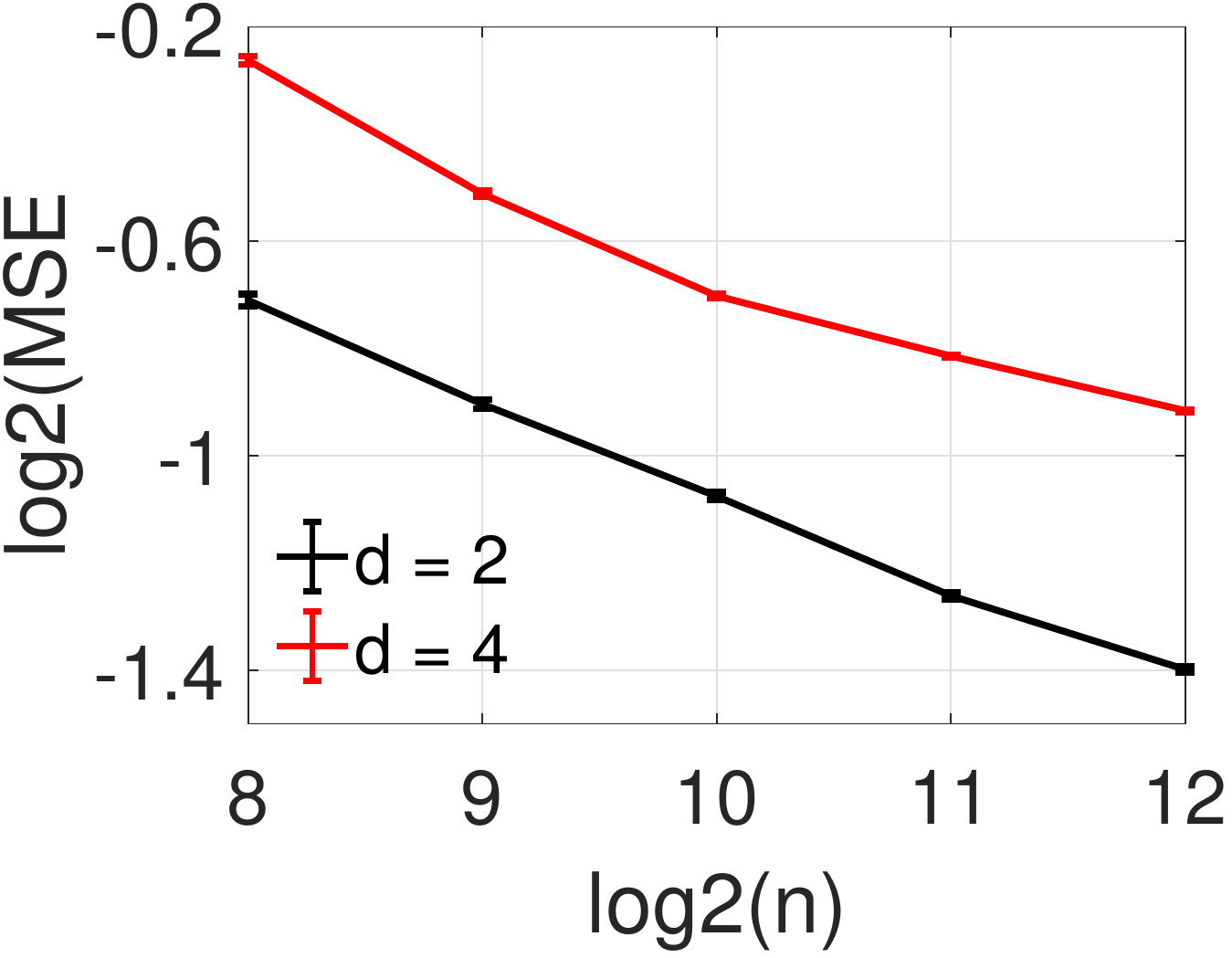} 
    \end{array}$ \\[.5ex]
    $\begin{array}{l}
    \hspace*{5ex} \text{\texttt{clusters}, Laplace, $d=2$}\\
    \includegraphics[width = .32\textwidth]{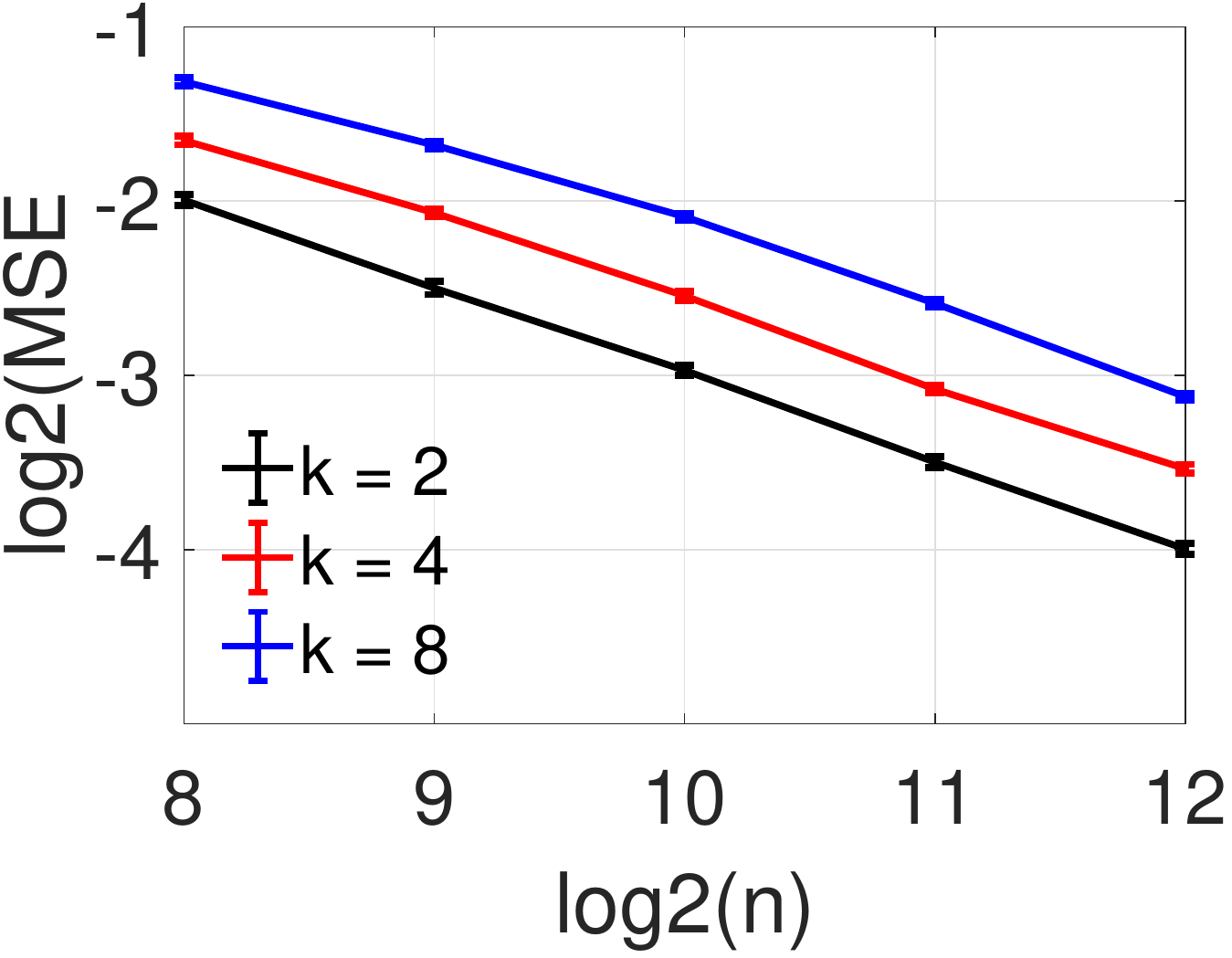} 
    \end{array}$ & \hspace*{-5ex} $\begin{array}{l}
    \hspace*{7ex} \text{\texttt{linear}, Laplace, $d=2$}\\
    \includegraphics[width = .32\textwidth]{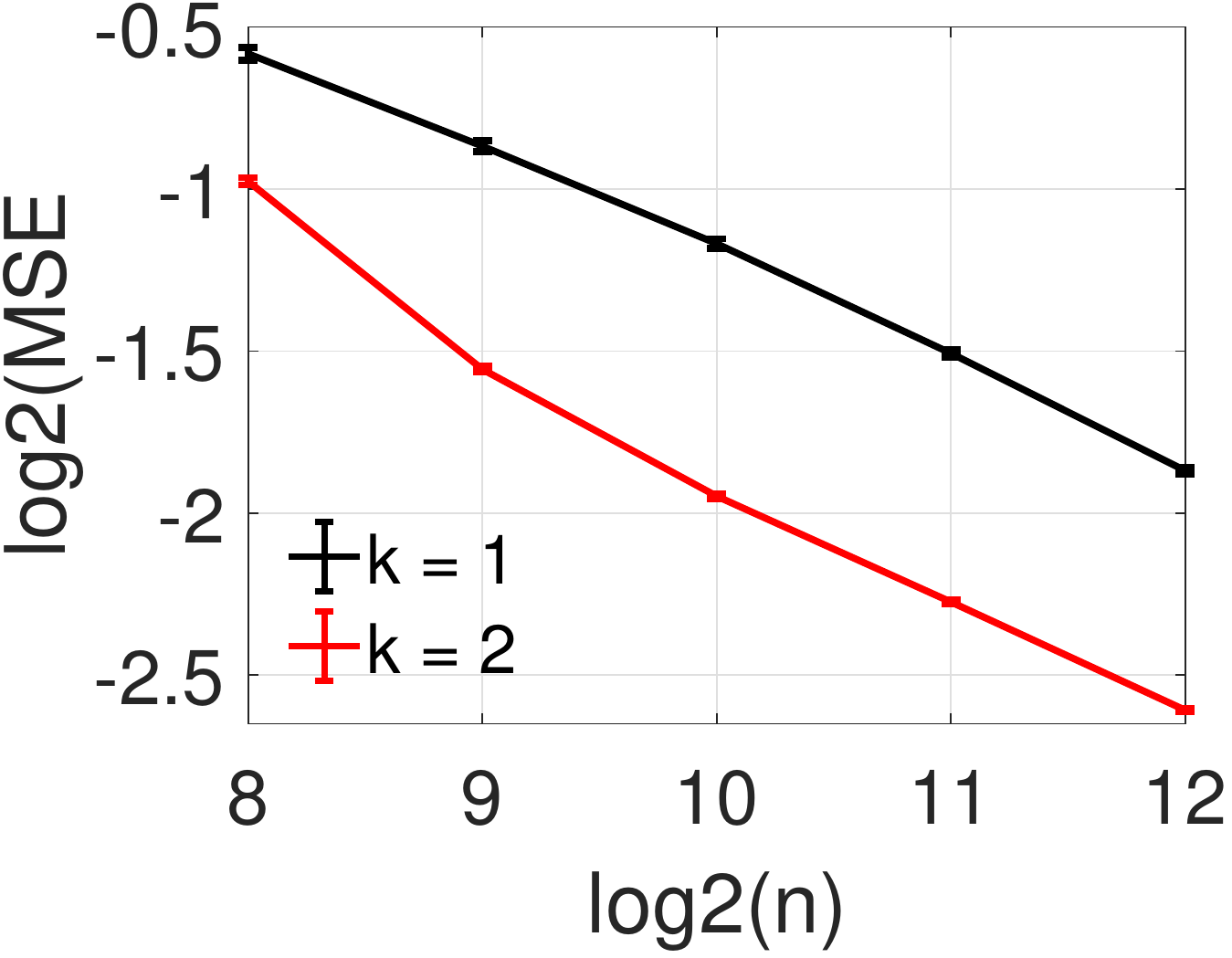} 
    \end{array}$ & \hspace*{-6ex} $\begin{array}{l}
    \hspace*{9ex} \text{\texttt{separable}, Laplace}\\
    \includegraphics[width = .32\textwidth]{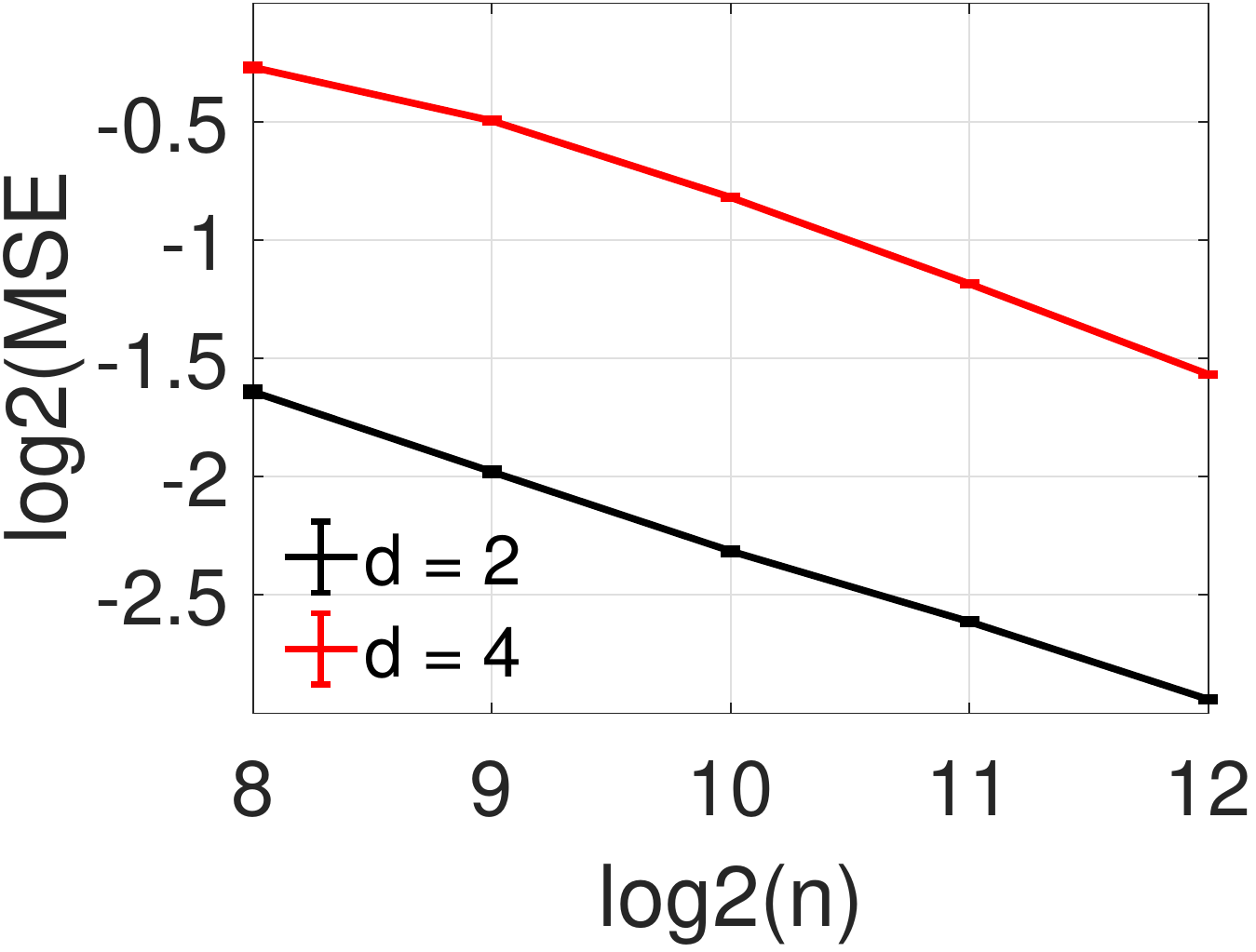} 
    \end{array}$\\[.5ex]
     $\begin{array}{l}
    \hspace*{8ex} \text{\texttt{clusters}, $d=4$}\\
    \includegraphics[width = .32\textwidth]{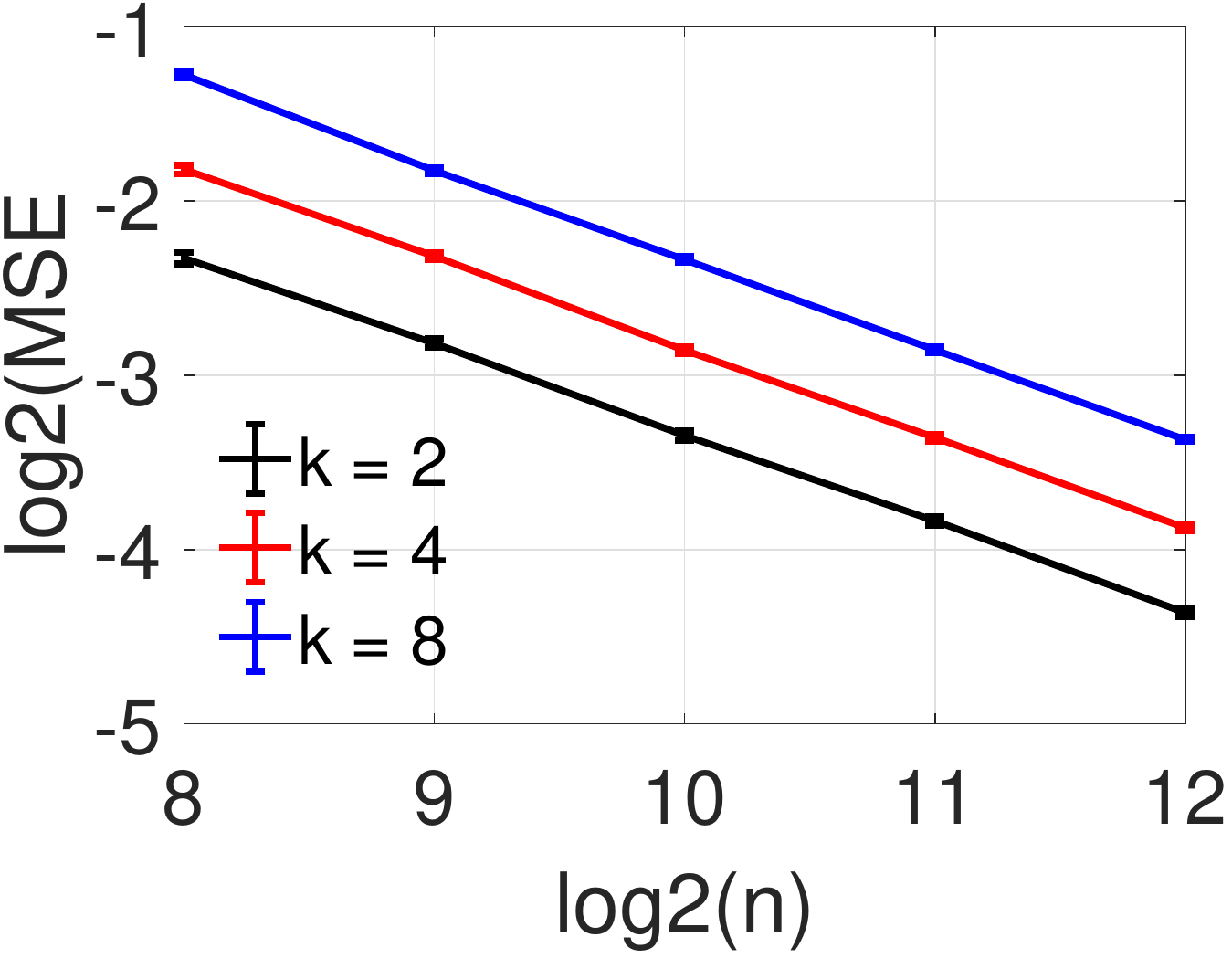} 
    \end{array}$ &  \hspace*{-5ex} $\begin{array}{l}
    \hspace*{9ex} \text{\texttt{linear}, $d=4$}\\
    \includegraphics[width = .32\textwidth]{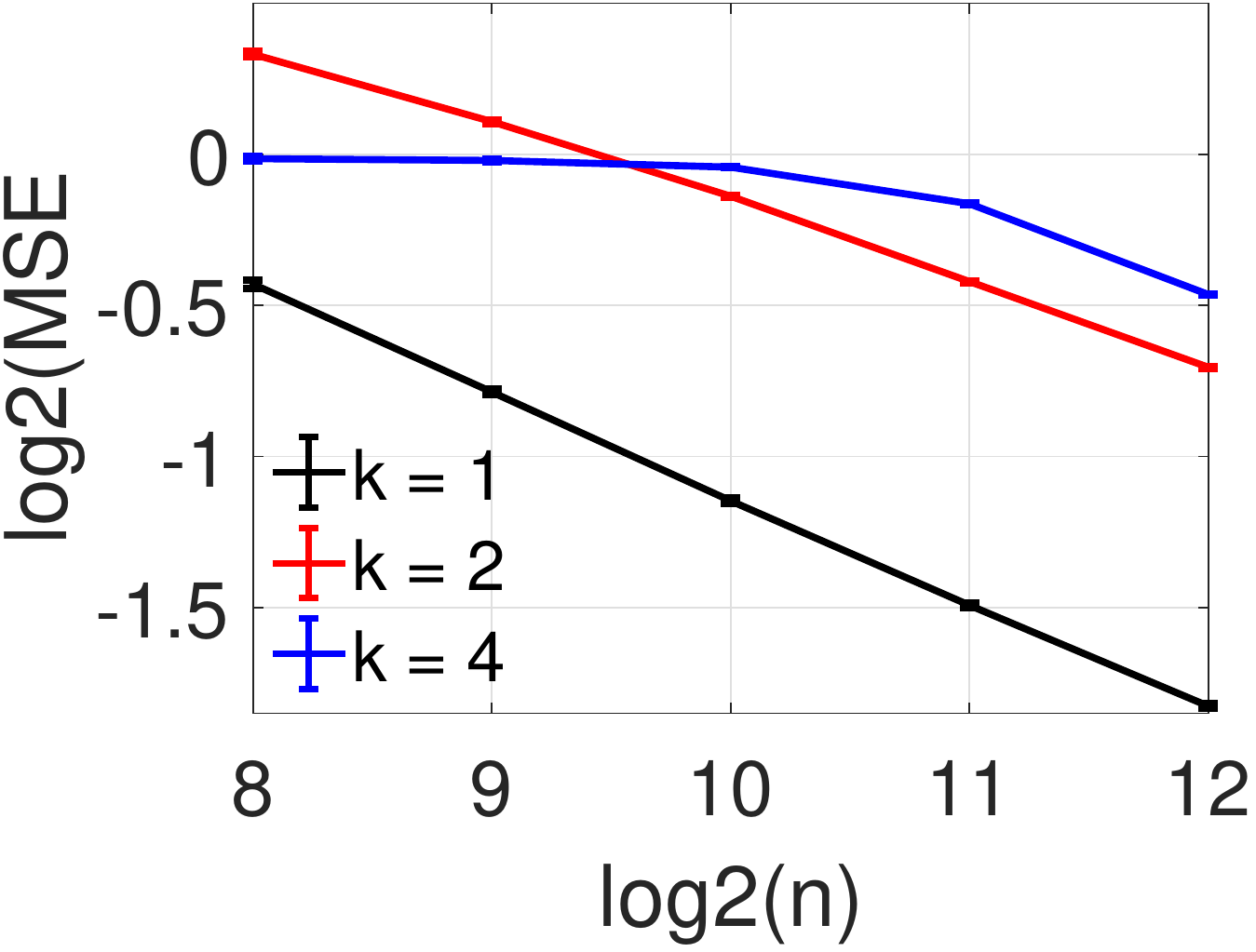} 
    \end{array}$ & \hspace*{-6ex} $\begin{array}{l}
    \hspace*{9ex} \text{\texttt{sphere}}\\
    \includegraphics[width = .32\textwidth]{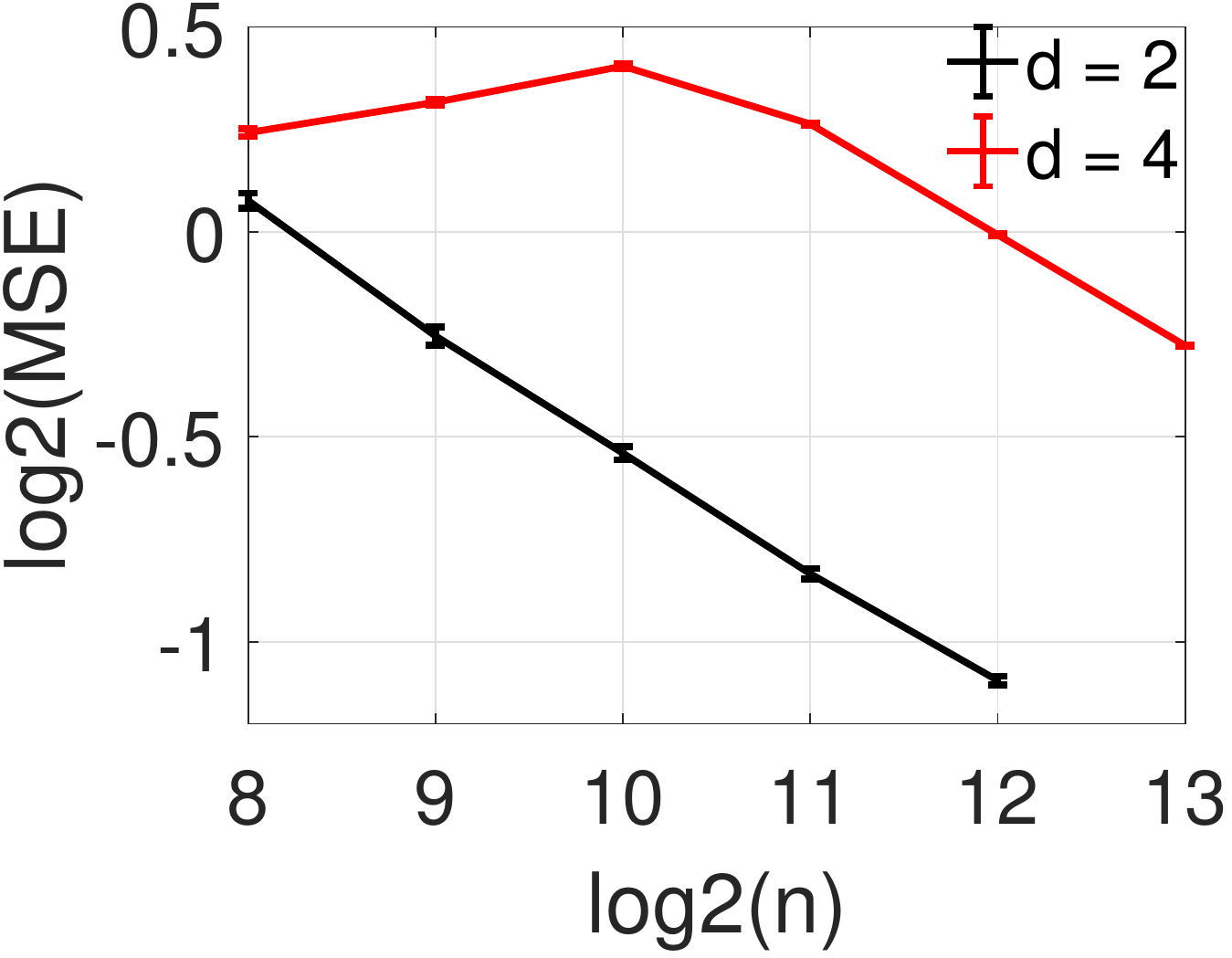} 
    \end{array}$\\[.5ex]
     $\begin{array}{l}
    \hspace*{5ex} \text{\texttt{clusters}, Laplace, $d=4$}\\
    \includegraphics[width = .32\textwidth]{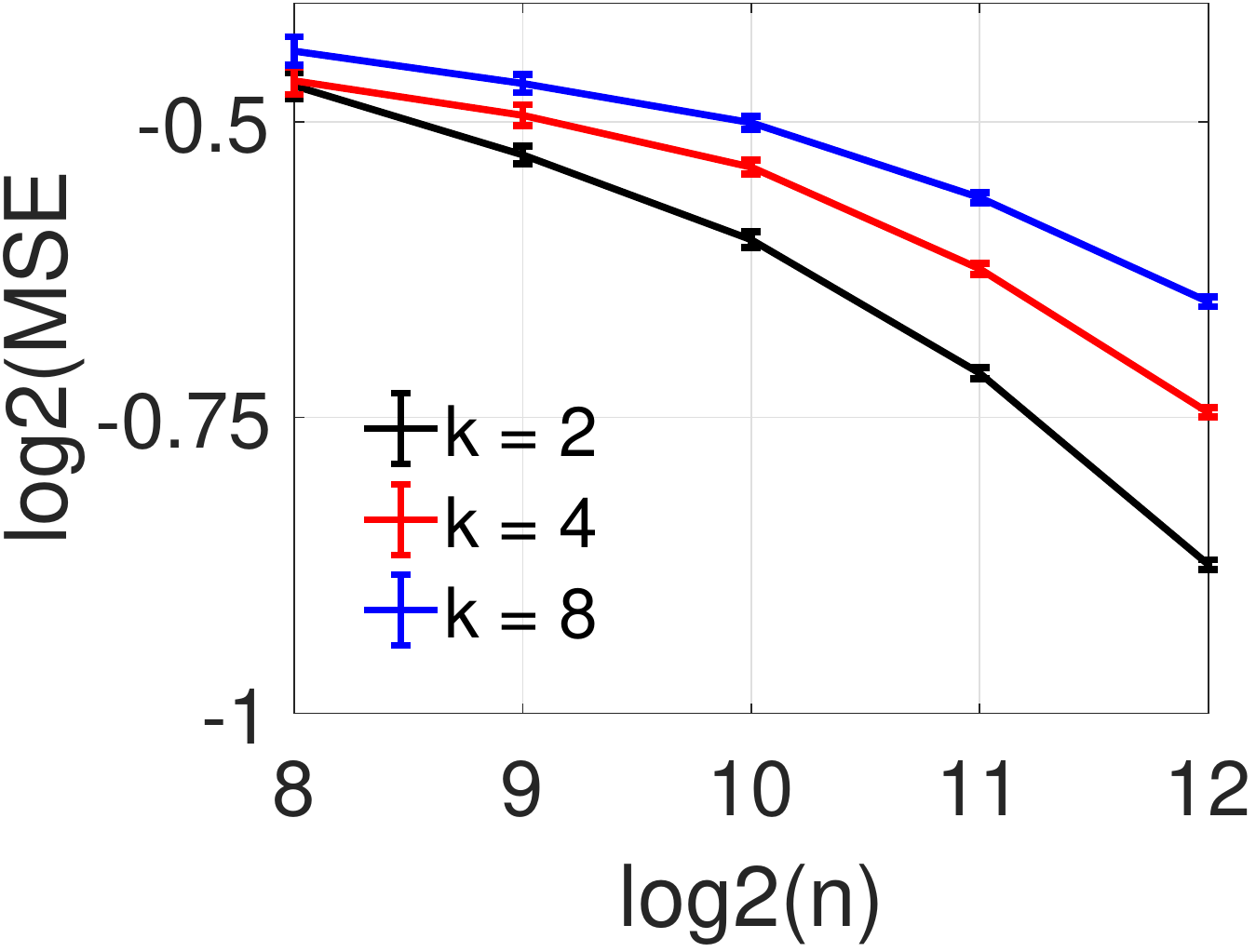} 
    \end{array}$ & \hspace*{-5ex} $\begin{array}{l}
    \hspace*{7ex} \text{\texttt{linear}, Laplace, $d=4$}\\
    \includegraphics[width = .32\textwidth]{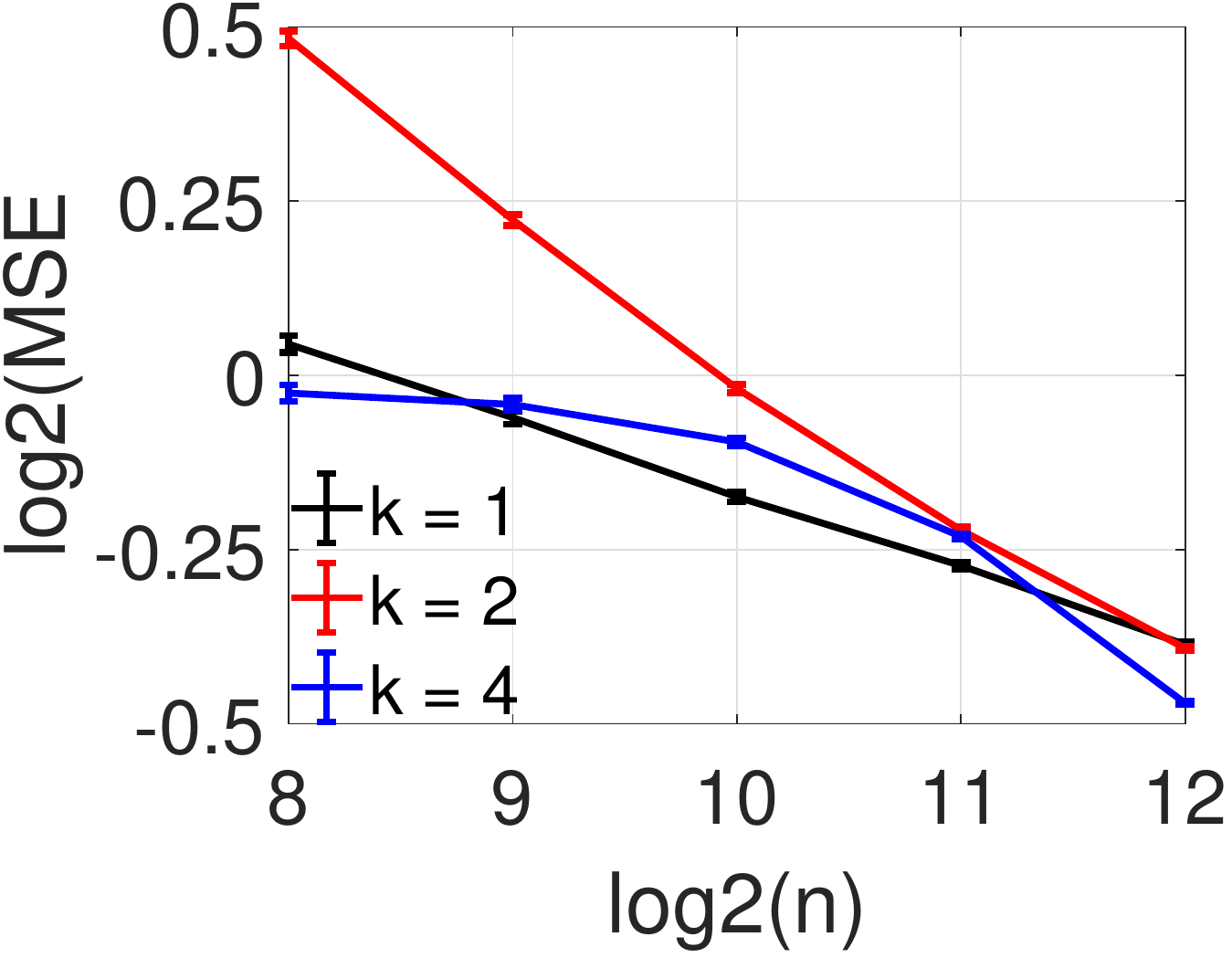} 
    \end{array}$ & \hspace*{-6ex} $\begin{array}{l}
    \hspace*{9ex} \text{\texttt{sphere}, Laplace}\\
    \includegraphics[width = .32\textwidth]{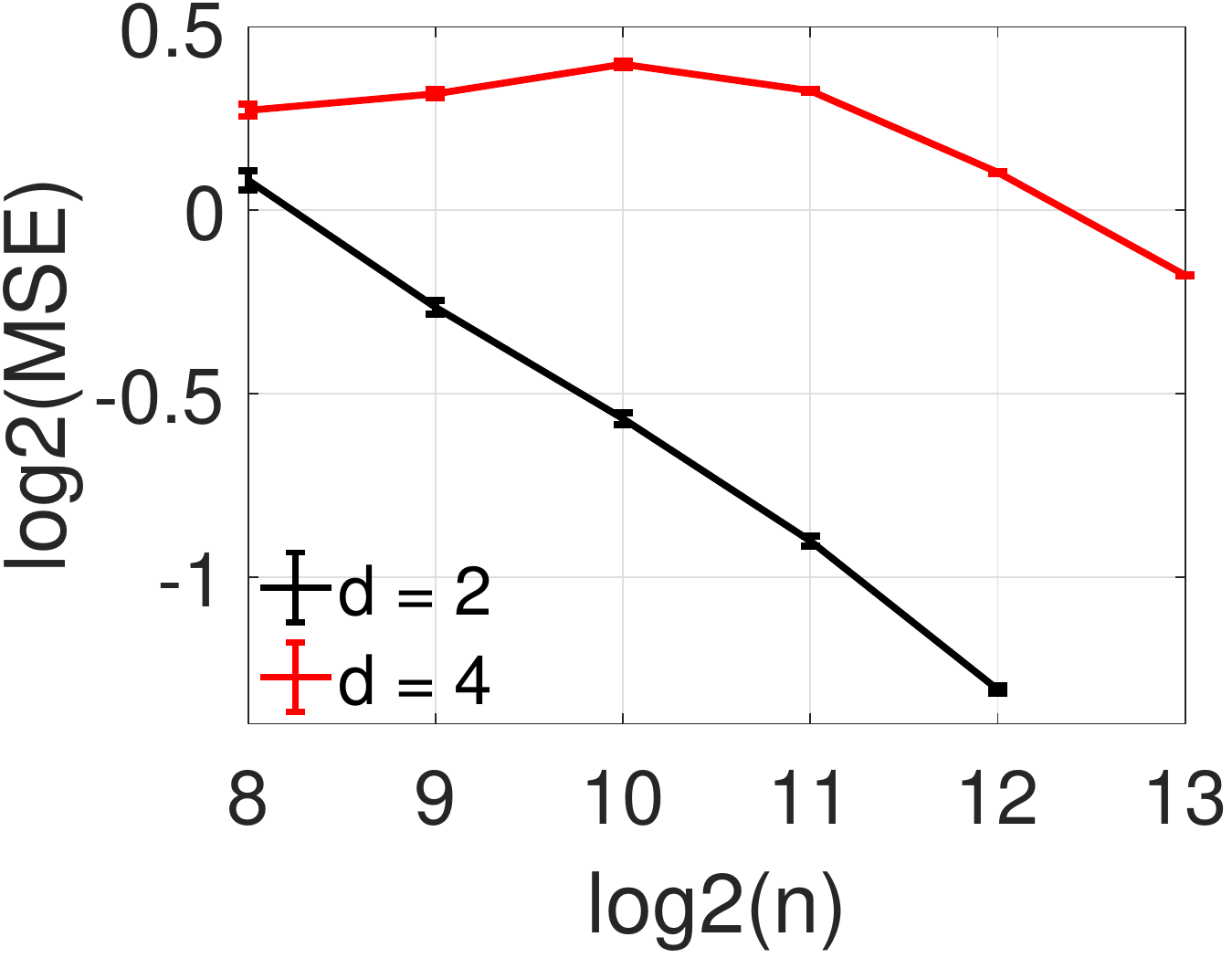} 
    \end{array}$\\
   &  \hspace*{-5ex}$\begin{array}{l}
    \hspace*{9ex} \text{\texttt{radial}}\\
    \includegraphics[width = .32\textwidth]{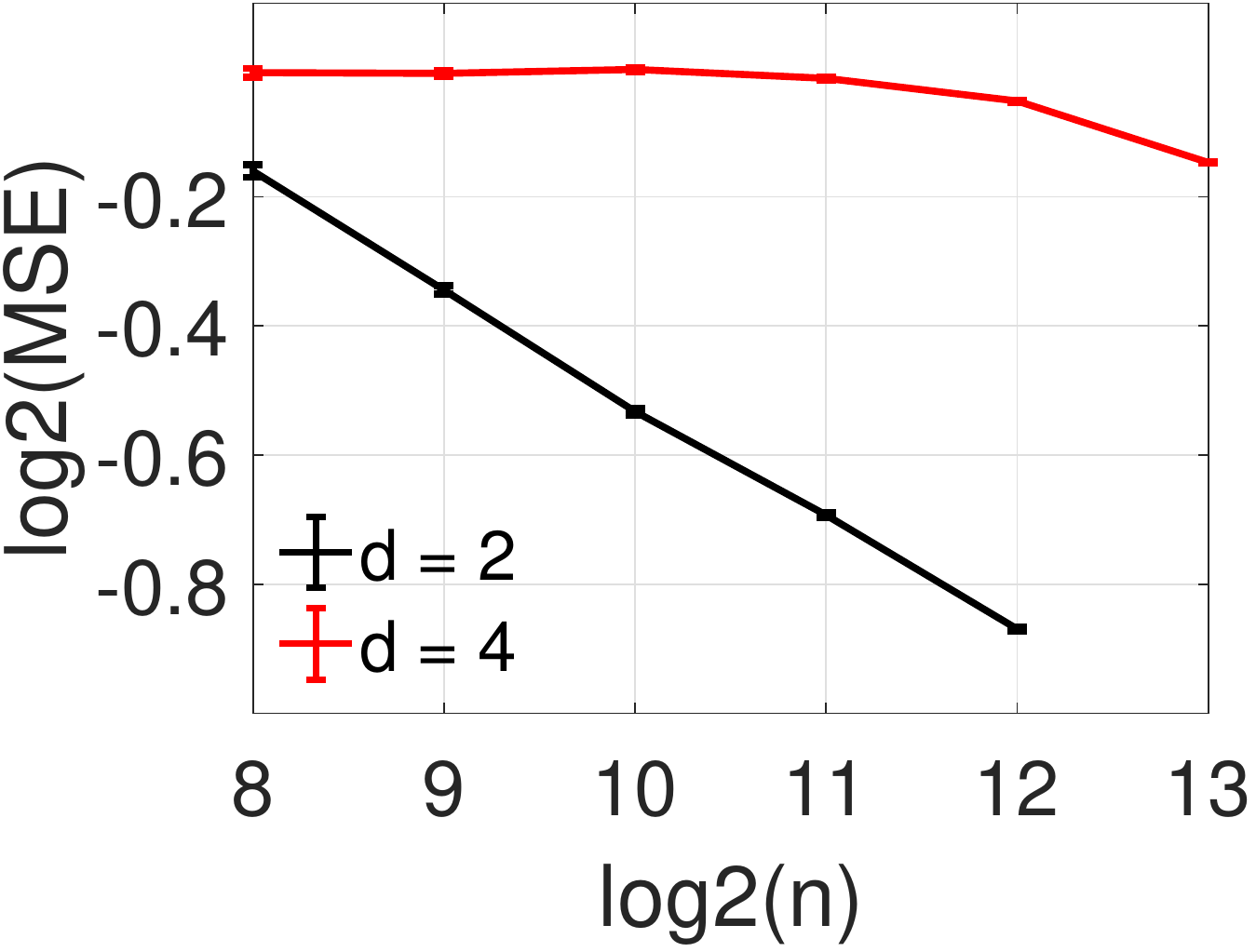}  \end{array}$ & \hspace*{-6ex} $\begin{array}{l}
    \hspace*{9ex} \text{\texttt{radial}, Laplace}\\
    \includegraphics[width = .32\textwidth]{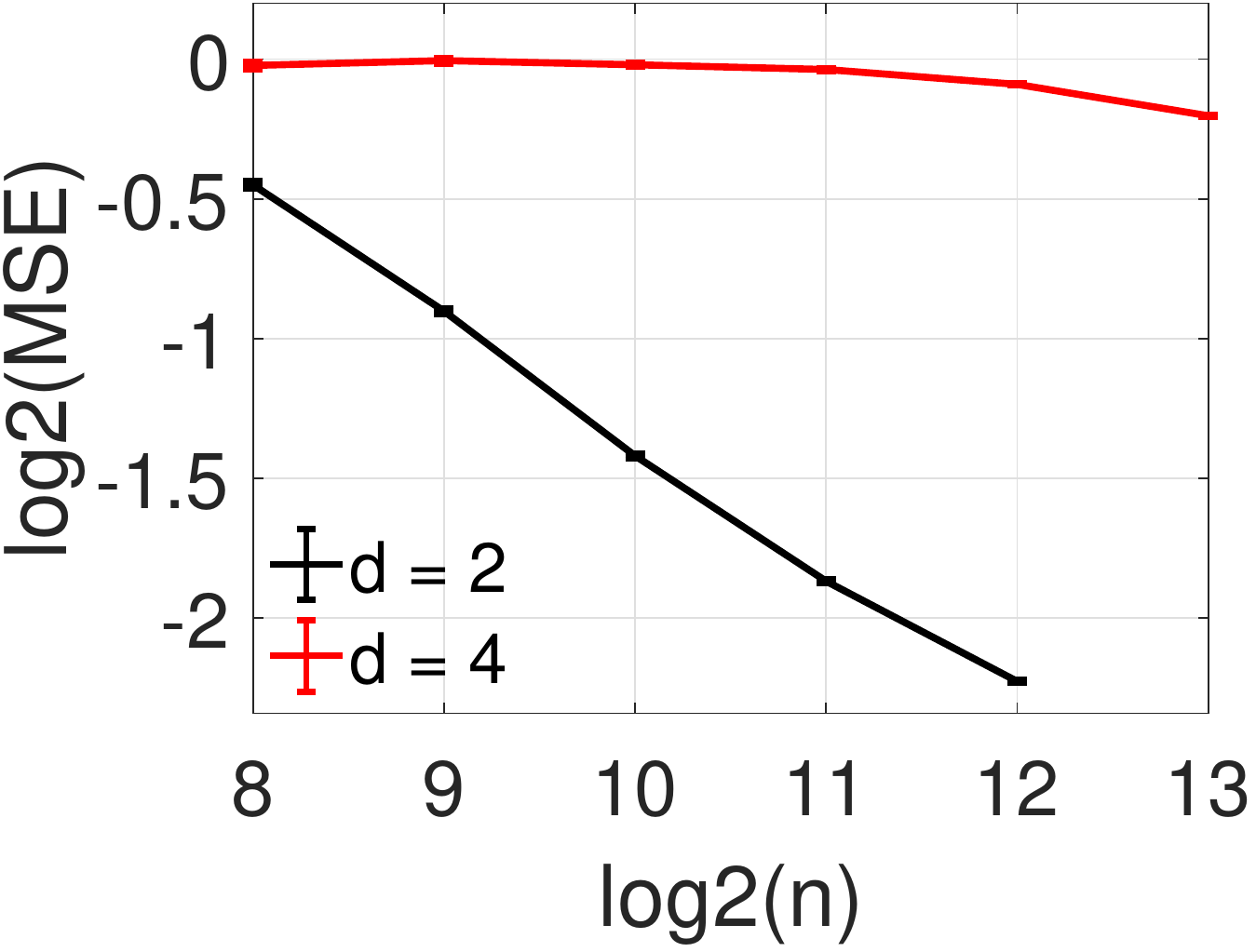} 
    \end{array}$
    \end{tabular}
    \caption{Denoising results for $d=2,4$; ``MSE" refers to $\frac{1}{n \sigma^2} \su \nnorm{\wh{f}(X_i) - f^*(X_i)}_2^2$. The results
    shown represent averages $\pm$1 standard error (the error bars are hardly visible for most instances) over 100 independent replications for the respective setting
    given in the plot captions (unless specified otherwise, the noise distribution is Gaussian). }
    \label{fig:results_denoising_generald}
\end{figure}
\clearpage
\section{Conclusion}\label{sec:conclusion}
In this paper, we have considered permuted and uncoupled regression for maps $f^*: \R^d \rightarrow \R^{d}$ that are gradients
of convex functions, the multi-dimensional analog of isotonic functions. This paper has studied exact permutation recovery
and denoising, and has established several connections to several recent works involving related permuted data problems. The task
of denoising is tackled via deconvolution based on the Kiefer-Wolfowitz NPMLE and optimal transport. The rich literature and 
the recently surging interest regarding the latter topic facilitates the analysis of the proposed approach. Compared to prior
work on one-dimensional permuted regression problems, the implementation of our approach is particularly convenient since
the underlying convex optimization problems are straightforward to solve and require almost no tuning; currently, the only
parameter to be specified is the grid for the approximate Kiefer-Wolfowitz problem, for which straightforward default options
are available that yield reasonable results empirically (cf.~$\S$\ref{sec:numerical_results}).

Despite the advances made in the current paper, there are several open problems and possible extensions from both practical
and theoretical viewpoints as discussed below.
\vskip1ex
\noindent (I)\emph{Towards deconvolution with unknown distribution of the errors}. Even though this objective appears not
to be achievable in general, it is of great practical importance to relax the somewhat unrealistic assumption that the
distribution of the error terms is fully known. As first steps, the following directions can be pursued: (i) the scale parameter
$\sigma$ is not known, and needs to be selected in a data-driven manner, and (ii) the model used for the errors is (mildly)
misspecified. 
\vskip1ex
\noindent (II) \emph{Beyond denoising}. In this paper, we only consider denoising, i.e., the estimation of the values of the unknown
function $f^*$ at the sample points $\{ X_i \}_{i = 1}^n$. A next step is to develop an approach that provides a (smooth) estimate
of $f^*$ over, say, a compact domain. 
\vskip1ex
\noindent (III) \emph{Minimaxity, adaptation, strong convexity}. Concerning our results obtained for denoising, the minimax
rate is yet unknown except for $d = 1$ \cite{Weed2018}. While slow rates appear inevitable in general, parts of our simulation
results indicate that faster rates can be obtained for instances with additional structure such as piecewise affine functions
and functions with low intrinsic dimensionality. In this context, it is of interest to study whether the proposed approach adapts to such
underlying low-complexity structure. Moreover, our results currently hinge on strong convexity, and it merits further investigation
whether this assumption can be relaxed.
\vskip1ex
\noindent (IV) \emph{Wasserstein vs.~maximum likelihood (ML) deconvolution}. The approach presented in this  paper is based on the Kiefer-Wolfowitz problem and thus ML deconvolution. Our analysis, however, is based on
bounding the distance to the underlying mixing measure in Wasserstein distance. This raises the question whether the use of the Wasserstein distance (as done in \cite{Weed2018} for $d = 1$) instead of the Kullback-Leibler divergence is more suitable for the
problem at hand. At the same time, ML deconvolution is considerably more convenient from a computational perspective. An interesting connection between ML deconvolution and entropic optimal transport is made in \cite{Rigollet2018}. It is of interest to study
whether that connection can be leveraged to facilitate the analysis of the proposed approach.
\vskip1ex
\noindent (V) \emph{Beyond equal dimensions}. The route taken in this paper requires $f^*$ to be a map from $\R^d$ to $\R^{d}$. This requirement can be limiting in applications in which the two samples $\mc{X}_n$ and $\mc{Y}_n$ live in different dimensions.   


\appendix

\bibliographystyle{IEEEtran}
\bibliography{references.bib}

\section{Proofs of main results}\label{app:proof_main}
This appendix contains proofs of our main results and additional technical background and discussion. The proofs of Theorems
\ref{theo:permuted_generald}, \ref{theo:unlinked_generald} and Proposition \ref{prop:denoising_d1} are decomposed into several key pieces which are presented in dedicated sections. The specific constituents and their dependencies are outlined in Figure \ref{fig:proofmap}.  
\begin{figure}[h!]
\centering
    \includegraphics[height = 0.1\textheight]{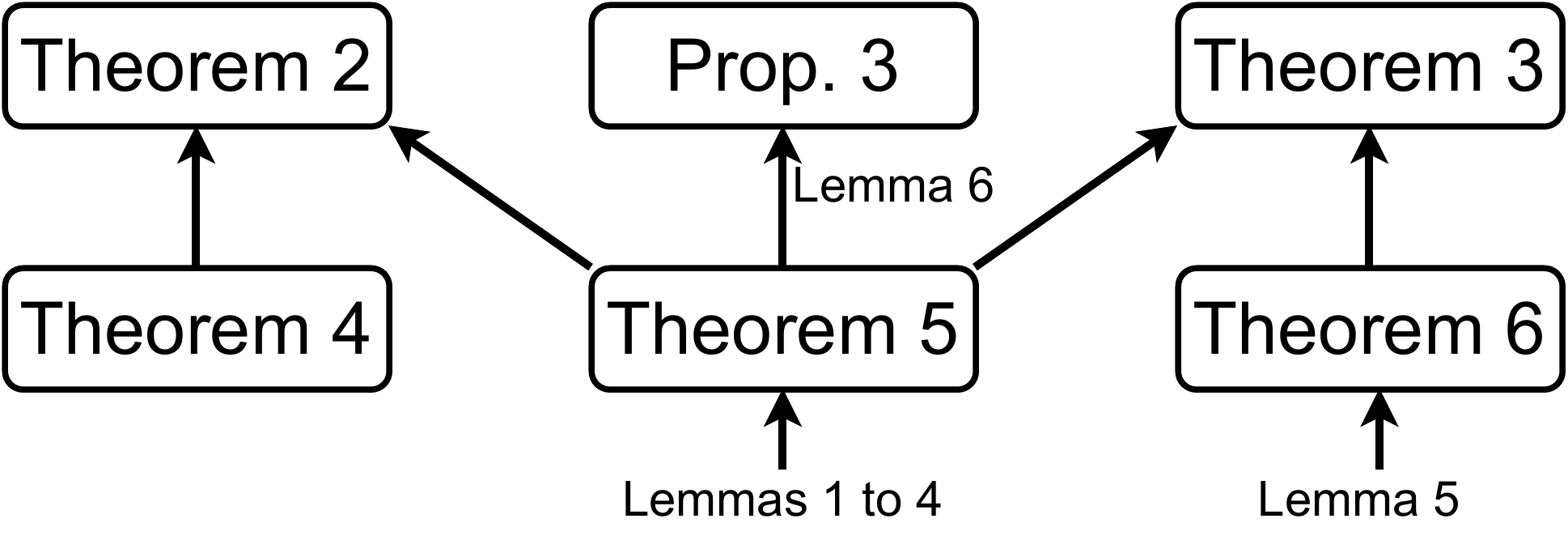}
    \caption{Chart summarizing the organization of the proofs of our main results on denoising {\bfseries(T2)}.}
    \label{fig:proofmap}
\end{figure}

\subsection{Proof of Proposition \ref{prop:permutation_recovery}}\label{sec:Proof-Prop-2}
Without loss of generality, we may assume that $\pi^*$ is the identity permutation, i.e., $\pi^*(i) = i$, $1 \leq i \leq n$. It then suffices to show that the set $\{ (X_i, Y_i) \}_{i = 1}^n$ is cyclically monotone with respect to the cost function
$c_0(x,y) := \nnorm{x - y}_2^2$, or equivalently the cost function $c(x,y) = -\scp{x}{y}$, under the conditions stated in the proposition. For this purpose, we need to show that for any subset $\{ (X_{i_j}, Y_{i_j}) \}_{j = 1}^k$ of $\{ (X_i, Y_i) \}_{i = 1}^n$ of size $k \geq 2$, it holds that \begin{equation}\label{eq:cyclic_monotonicity_proof}
-\sum_{j = 1}^k \scp{X_{i_j}}{Y_{i_j}} \leq  -\sum_{j = 1}^k \scp{X_{i_{j+1}}}{Y_{i_j}}, \quad i_{k+1} = i_1.    
\end{equation}
Expanding $Y_{i_j} = f^*(X_{i_j}) + \eps_{i_j}$, $1 \leq j \leq k+1$, the above inequality becomes 
\begin{equation*}
-\sum_{j = 1}^k \scp{X_{i_j}}{f^*(X_{i_j}) + \eps_{i_j}} \leq  -\sum_{j = 1}^k \scp{X_{i_{j+1}}}{f^*(X_{i_j}) + \eps_{i_j}}. 
\end{equation*}
Re-arranging in order to single out the contributions of the noise yields the following condition equivalent to \eqref{eq:cyclic_monotonicity_proof} 
\begin{equation}\label{eq:single_out_noise}
\sum_{j = 1}^k \scp{X_{i_{j+1}} - X_{i_j}}{f^*(X_{i_j})} + \sum_{j = 1}^k \scp{X_{i_{j+1}} - X_{i_j}}{\eps_{i_j}} \leq 0. 
\end{equation}
By $\lambda$-strong convexity of $\psi_{f^*}$, we have 
\begin{equation*}
\psi_{f^*}(X_{i_{j + 1}}) - \psi_{f^*}(X_{i_{j}}) - \nscp{\underbrace{\nabla \psi_{f^*}(X_{i_j})}_{f^*(X_{i_j})}}{X_{i_{j + 1}} - X_{i_j}} \geq \lambda \nnorm{X_{i_{j + 1}} - X_{i_{j}}}_2^2, \quad 1 \leq j \leq k. 
\end{equation*}
Summation of the above inequality over $j$ and using the cyclicity condition $i_{k + 1} = i_1$ yields 
\begin{equation}\label{eq:summing_over_j_strongconvex}
\sum_{j = 1}^k \scp{X_{i_{j+1}} - X_{i_j}}{f^*(X_{i_j})} \leq -\lambda \sum_{j = 1}^k \nnorm{X_{i_{j+1}} - X_{i_j}}_2^2.
\end{equation}
We now upper bound the second term in \eqref{eq:single_out_noise}. Conditional on the $\{ X_i \}_{i = 1}^n$ and using the independence of the errors, we have
\begin{equation*}
\sum_{j = 1}^k \scp{X_{i_{j+1}} - X_{i_j}}{\eps_{i_j}} \sim N\Big(0, \sigma^2 \sum_{j = 1}^k \nnorm{X_{i_{j+1}} - X_{i_j}}_2^2 \Big). 
\end{equation*}
Now define the quantity
\begin{equation}\label{eq:M_k}
M_k := \max_{\{i_1, \ldots, i_k\} \subseteq \{1,\ldots,n\}} \frac{1}{\sqrt{\sum_{j = 1}^k \nnorm{X_{i_{j+1}} - X_{i_j}}_2^2}} \sum_{j = 1}^k \scp{X_{i_{j+1}} - X_{i_j}}{\eps_{i_j}}.
\end{equation}
The standard Gaussian tail bound $\p(Z > z) \leq \exp(-z^2 /2)$ for $z \geq 0$, $Z \sim N(0,1)$, combined with the union bound and the inequality $\binom{n}{k} \leq \left( \frac{en}{k} \right)^k$ yields
\begin{equation*}
\p(M_k \geq t) \leq \exp \left(-\frac{t^2}{2\sigma^2} + k \log(en/k) \right).
\end{equation*}
Choosing $t = \sigma \sqrt{4\log n + 2k \log(en/k)}$, we obtain that
\begin{equation}\label{eq:M_k_tailbound}
\p\left(M_k \geq \sigma \sqrt{4\log n + 2k \log(en/k)}\right) \leq \frac{1}{n^2}. 
\end{equation}
Combining \eqref{eq:single_out_noise}, \eqref{eq:summing_over_j_strongconvex}, \eqref{eq:M_k},
we note that the desired condition \eqref{eq:cyclic_monotonicity_proof} is implied by the condition
\begin{equation*}
\forall k=2,\ldots,n: \quad \lambda \sqrt{k} \min_{i < j} \nnorm{X_i - X_j} \geq M_k. 
\end{equation*}
Using \eqref{eq:M_k_tailbound} along with the observation that the function $k \mapsto \sigma \sqrt{4 \log(n)/k + 2 \log(e n / k)}$ is
decreasing in $k$, a union bound over $k = 2,\ldots,n$, yields that if
\begin{equation*}
\min_{i < j} \nnorm{X_i - X_j}_2 \geq  \frac{\sigma \sqrt{6 \log n}}{\lambda}, 
\end{equation*}    
the required inequality \eqref{eq:cyclic_monotonicity_proof} for cyclic monotonicity holds with probability at least $1 - 1/n$.


\section{Proof of Theorems \ref{theo:permuted_generald} and~\ref{theo:unlinked_generald} and Proposition \ref{prop:denoising_d1}}\label{sec:Proof:Thm-2-3}
The proofs of these two theorems involve a few other results,  which we first state in the  following subsections. These results may be of independent interest. 

Let $\nu_n^* := \frac{1}{n} \su \delta_{\theta_i^*} = \frac{1}{n} \su \delta_{f^*(X_i)}$, and let
$\wh{\nu}$ denote the mixing measure associated with the NPMLE \eqref{eq:Kiefer_Wolfowitz}. Theorem
\ref{theo:kantorovich_permuted} below provides an {\it upper bound} on the empirical $L_2$-loss of the barycentric projection estimator $\{ \wh{f}(X_i) \}_{i = 1}^n$, obtained from an optimal coupling between $\nu_n^*$ and $\wh{\nu}$ (see~\eqref{eq:barycentric_proj_estimator}), in terms of the 2-Wasserstein distance between $\nu_n^*$ and $\wh{\nu}$. 
\subsection{Analysis of the Kantorovich problem \eqref{eq:Kantorovich_finite} for general $d$}
The result below is the central technical component in proving Theorem \ref{theo:permuted_generald}; see Appendix~\ref{sec:Proof-Thm-4} for its proof. 

\begin{theo}\label{theo:kantorovich_permuted} Consider the atomic measure $\nu_n^* := \frac{1}{n} \su \delta_{\theta_i^*}$, $\theta_i^* = f^*(X_i)$, $1 \leq i \leq n$, with $f^* = \nabla \psi_{f^*}$ such that assumptions \emph{{\bfseries(A1)}} through \emph{{\bfseries(A3)}} in $\S$\ref{subsec:denoising}
are satisfied, and let $\wh{\nu} := \sum_{j = 1}^p \wh{\alpha}_j \delta_{\wh{\theta}_j}$ be another atomic measure on $\R^d$.  
Let further $\mu_n := \frac{1}{n} \su \delta_{X_i}$, and consider the barycentric projection \eqref{eq:barycentric_proj_estimator} based on 
an optimal coupling \eqref{eq:Kantorovich_main} between $\mu_n$ and $\wh{\nu}$. We then have 
\begin{equation}\label{eq:thm_Kantorovich_main_conclusion}
\frac{1}{n} \su \nnorm{\wh{f}(X_i) - f^*(X_i)}_2^2 \leq \frac{L}{\lambda} \emph{\text{$\dW_2^2$}}(\nu_n^*, \wh{\nu}). 
\end{equation}
\end{theo}
We next upper bound $\text{$\dW_2^2$}(\nu_n^*, \wh{\nu})$.
\subsection{Wasserstein deconvolution rates}\label{app:deconvolution}
The following result provides an upper bound on the 2-Wasserstein distance between an underlying atomic mixing measure
$\nu_n^* = \frac{1}{n} \su \delta_{\theta_i^*}$ with uniformly bounded support and a deconvolution estimator $\wh{\nu}$ in terms of the Hellinger distance between the convolved 
measures $\nu_n^* \star \varphi_{\sigma}$ and $\wh{\nu} \star \varphi_{\sigma}$, along the route of the proof of Theorem 2 in \cite{Nguyen2013}; see Appendix~\ref{sec:Proof-Thm-5} for a proof. 
\begin{theo}\label{theo:wasserstein_deconvolution}
Let $\wh{\text{\emph{\textsf{f}}}}_n$ denote the NPMLE \eqref{eq:Kiefer_Wolfowitz} with $\varphi(z) = (2\pi)^{-d/2} \exp(-\nnorm{z}_2^2)$
given $\{ Y_i \}_{i = 1}^n \overset{\text{\emph{i.i.d.}}}{\sim} \varphi_{\sigma} \star \nu_n^*$ with 
$\nu_n^* = \frac{1}{n} \su \delta_{\theta_i^*}$ such that $\{ \theta_i^* \}_{i = 1}^n$ is contained in a Euclidean ball of radius $B$ centered at the origin, and let $\wh{\nu}$ be the mixing measure associated with the NPMLE, i.e., $\wh{\nu} \star \varphi_{\sigma} = \wh{\text{\emph{\textsf{f}}}}_n$. Choose $s > k \geq 1$, and suppose that $n \geq C_0(d, B, \sigma) (\log n)^{d+1}$. It then holds with probability at least $1 - 5/n$, 
\begin{equation*}
\text{\emph{$\dW_k^k$}}(\nu_n^*, \wh{\nu}) \leq C(s,k, d, \sigma, B) \left( \frac{1}{\log n} \right)^{k/2},   
\end{equation*}
where $C_0$ and $C$ are positive constants depending only on the quantities in the parentheses. 
\end{theo}

\subsection{Analysis of the Kantorovich problem \eqref{eq:Kantorovich_finite} for general $d$ when $m \ne n$}

\noindent The next result (proved in Appendix~\ref{sec:Proof-Thm-6}) extends Theorem \ref{theo:kantorovich_permuted} to the unlinked setting
based on samples $\mc{X}_n = \{X_1, \ldots, X_n\}$ and $\mc{Y}_m = \{Y_1,\ldots,Y_m \}$. The proof requires only one additional ingredient (Lemma \ref{lem:Wasserstein_concentration}) to the preceding proof. 
\begin{theo}\label{theo:kantorovich_unlinked} 
Let $X_1, \ldots, X_n \overset{\text{\emph{i.i.d.}}}{\sim} \mu$ and $\theta_1^*,\ldots,\theta_{m}^* \overset{\text{\emph{i.i.d.}}}{\sim} \nu$, where the support of $\nu$ is contained 
in a Euclidean ball of radius $B$, and let $f^* = \nabla \psi_{f^*}$ be the Brenier map
(cf.~Theorem \ref{theo:Brenier}) transporting $\mu$ to $\nu$ with $f^*$ satisfying \emph{{\bfseries(A1)}} and \emph{{\bfseries(A2)}}. Consider the atomic measures 
$\nu_n^* = \frac{1}{n} \su \delta_{f^*(X_i)}$, $\nu_{m}^* = \frac{1}{m} \sum_{i = 1}^{m} \delta_{\theta_i^*}$, and 
$\wh{\nu} = \sum_{j = 1}^p \wh{\alpha}_j \delta_{\wh{\theta}_j}$. 

Let further $\mu_n = \frac{1}{n} \su \delta_{X_i}$, and consider the barycentric projection \eqref{eq:barycentric_proj_estimator} based on 
an optimal coupling \eqref{eq:Kantorovich_main} between $\mu_n$ and $\wh{\nu}$. For positive constants $C, c > 0$, we then have, with probability at least  $1 - C (n^{-c} + m^{-c})$,
{\small \begin{align*}
\frac{1}{n} \su \nnorm{\wh{f}(X_i) - f^*(X_i)}_2^2 &\leq 
 \frac{2L}{\lambda} \Bigg( \emph{\text{$\dW$}}_2^2(\wh{\nu}, \nu_{m}^*) + 2 \sqrt{\frac{\log n}{n}} \vee \Big(\frac{\log n}{n}\Big)^{2/d} + 2 \sqrt{\frac{\log m}{m}} \vee \Big(\frac{\log m}{m}\Big)^{2/d} \Bigg).
\end{align*}}
\end{theo}

\subsection{Proof of Theorem \ref{theo:permuted_generald}}
Recall that $\nu_n^* = \frac{1}{n} \su \delta_{\theta_i^*} = \frac{1}{n} \su \delta_{f^*(X_i)}$, and note that $\wh{\nu}$ denotes the mixing measure associated with the NPMLE \eqref{eq:Kiefer_Wolfowitz}. Theorem
\ref{theo:kantorovich_permuted} above then yields that the barycentric projection 
estimator $\{ \wh{f}(X_i) \}_{i = 1}^n$, obtained from an optimal coupling between $\nu_n^*$ and 
$\wh{\nu}$ (see~\eqref{eq:barycentric_proj_estimator}), obeys the bound \eqref{eq:thm_Kantorovich_main_conclusion}.

Next, note that, under the stated condition on $n$, the squared 2-Wasserstein distance between $\nu_n^*$ and $\wh{\nu}$ can be bounded as $\textsf{W}_2^2(\nu_n^*, \wh{\nu}) \lesssim_{d, \sigma, B} \frac{1}{\log n}$ according to Theorem \ref{theo:wasserstein_deconvolution} with the stated probability. This completes the proof of the theorem. \pushQED{\qed} \popQED

\subsection{Proof of Theorem \ref{theo:unlinked_generald}}
The main modification relative to the previous proof is to consider both $\nu_n^* := \frac{1}{n} \su \delta_{f^*(X_i)}$
and $\nu_{m}^* := \frac{1}{m} \sum_{i = 1}^{m} \delta_{\theta_i^*}$, where $\{ \theta_{i}^* \}_{i = 1}^{m} \overset{\text{i.i.d.}}{\sim} \nu$. Theorem \ref{theo:kantorovich_unlinked} bounds the mean squared
denoising error in terms of the Wasserstein distance $\dW_2^2(\nu_{m}^*, \wh{\nu})$ and additional lower-order terms, where
$\wh{\nu}$ denotes the mixing measure associated with the NPMLE \eqref{eq:Kiefer_Wolfowitz} based on
the sample $\{ Y_i \}_{i = 1}^{m}$. We finally invoke Theorem \ref{theo:wasserstein_deconvolution}
to bound $\dW_2^2(\nu_{m}^*, \wh{\nu})$, with $\nu_n^*$ and $\{ \theta_i^* \}_{i = 1}^n$ replaced by 
$\nu_{m}^*$ and $\{ \theta_i^* \}_{i = 1}^{m}$, respectively. \pushQED{\qed} \popQED

\subsection{Proof of Proposition \ref{prop:denoising_d1}}\label{sec:Proof-Prop-3}
Let us consider the permuted regression setup \eqref{eq:permuted_regression}, and consider the two Kantorovich problems 
\begin{equation*}
\text{(i)}\; \min_{\gamma \in \Pi(\mu_n, \wh{\nu})} \int (x - \theta)^2 \, d\gamma(x, \theta), \qquad 
\text{(ii)}\; \min_{\gamma \in \Pi(\nu_n^*, \wh{\nu})} \int (\zeta - \theta)^2 \, d\gamma(\zeta, \theta). 
\end{equation*}
Let $\wh{\gamma}^1$ denote the so-called \emph{Northwest-corner} solution of (i), cf.~\cite[][$\S$3.4.2]{COT2019}, and 
let $\wt{\gamma}^1 = (f^*, \textsf{id})\#\wh{\gamma}^1$ the push-forward (cf.~Definition 1 in Appendix \ref{app:optimaltransport}) of $\wh{\gamma}^{1}$ under the transformation that pushes forward its two marginals to $f^* \# \mu_n = \nu_n^*$ and and 
$\textsf{id}\#\wh{\nu} = \wh{\nu}$, where \textsf{id} denotes the identity map. Since the $\{ X_i \}_{i = 1}^n$ and $\{ \theta_i^* \}_{i = 1}^n$ associated with $\mu_n$ and $\nu_n^*$ are related 
by the non-decreasing transformation $f^*$, $\wt{\gamma}^1$ is a minimizer of (ii) as follows, e.g., from Proposition 1 in \cite{Cuturi2019}. Consequently, letting $\wt{\theta}_{i} = \E_{(\theta, \zeta) \sim \wt{\gamma}^1}[\theta | \zeta = \theta_i^*]$, $1 \leq i \leq n$, denote the barycentric projections, we have 
\begin{equation*}
\wt{\theta}_i =  \frac{\int_{\theta} \theta \; d\wt{\gamma}^1(\theta_i^*, \theta)}{\int_{\theta} \; d\wt{\gamma}^1(\theta_i^*, \theta)} = \frac{\int_{\theta} \theta \; d\wh{\gamma}^1(X_i, \theta)}{\int_{\theta} \; d\wh{\gamma}^1(X_i, \theta)} = \wh{f}(X_i), \quad 1 \leq i \leq n,
\end{equation*}
where the last equality is simply the definition of the  $\{ \wh{f}(X_i) \}_{i = 1}^n$ (cf.~\eqref{eq:barycentric_proj_estimator}). 
On the other hand, by Lemma \ref{lem:barycentric}, $\frac{1}{n} \su (\wt{\theta}_i - \theta_i^*)^2 \leq \dW_2^2(\nu_n^*, \wh{\nu})$, which concludes the proof. 

\noindent The proof for the uncoupled regression setup is analogous to the proof of Theorem \ref{theo:unlinked_generald} (cf.~Theorem \ref{theo:kantorovich_unlinked} and its proof) and is hence omitted. \pushQED{\qed} \popQED

\section{Proof of Theorems~\ref{theo:kantorovich_permuted},~\ref{theo:wasserstein_deconvolution} and~\ref{theo:kantorovich_unlinked}}\label{sec:Proof-Thms}
\subsection{Proof of Theorem~\ref{theo:kantorovich_permuted}}\label{sec:Proof-Thm-4}

\noindent \begin{bew} 
\noindent
Consider an optimal coupling $\wh{\gamma}$ between $\wh{\nu}$ and $\mu_n$ minimizing 
\eqref{eq:Kantorovich_main}, and let $\wh{\gamma}_{ij}$ denote the resulting probability mass that
is assigned to $X_i$ and $\wh{\theta}_j$, $1 \leq i \leq n$, $1 \leq j \leq p$. Define further $\pi_j(X_i) = \wh{\Gamma}_{ij} n$, $1 \leq i \leq n$, $1 \leq j \leq p$. Accordingly, we have $\wh{\alpha}_j = \frac{1}{n} \su \pi_j(X_i) = \int \pi_j(x) \; d\mu_n(x)$, $1 \leq j \leq p$. Recall that $\psi_{f^*}^{\star}$ denotes the Legendre-Fenchel conjugate of $\psi_{f^*}$. We first bound $\int \psi_{f^*}^{\star}(\theta) \;d\wh{\nu}(\theta) - \int \psi_{f^*}^{\star}(\theta) \;d\nu_n^*(\theta)$ as
\begin{align}
  &\sum_{j = 1}^p \psi_{f^*}^{\star}(\wh{\theta}_j) \wh{\alpha}_j -  \int \psi_{f^*}^{\star}(\theta) \;d\nu_n^*(\theta) \notag \\
  =&\int \sum_{j = 1}^p \pi_j(x) \psi_{f^*}^{\star}(\wh{\theta}_j) \;d\mu_n(x) - \int \psi_{f^*}^{\star}(f^*(x)) \;d\mu_n(x) \notag \\
  \geq&\int  \psi_{f^*}^{\star}\left(\sum_{j = 1}^p \pi_j(x) \wh{\theta}_j \right) \;d\mu_n(x)  -  \int \psi_{f^*}^{\star}(f^*(x)) \;d\mu_n(x) \notag \\
  =& \int \psi_{f^*}^{\star}(\wh{f}(x)) \; d\mu_n(x) - \int \psi_{f^*}^{\star}(f^*(x)) \;d\mu_n(x) \notag \\
  \geq& \int \nabla \psi_{f^*}^{\star}(f^*(x))^{\T}(\wh{f}(x) - f^*(x)) \; d\mu_n(x) \; + \frac{1}{2L} \int \nnorm{\wh{f}(x) - f^*(x)}_2^2 \; d\mu_n(x), \notag \\
  =& \int x^{\T}(\wh{f}(x) - f^*(x)) \; d\mu_n(x) + \frac{1}{2L} \int \nnorm{\wh{f}(x) - f^*(x)}_2^2 \; d\mu_n(x)    \label{eq:main_lower_bound}
\end{align}
where the two inequalities follow from convexity and $L$-smoothness of $\psi_{f^*}$ in virtue of {\bfseries(A2)}, which implies $\frac{1}{L}$-strong convexity of its conjugate $\psi_{f^*}^{\star}$ \cite{Kakade2009}; the last equality follows from Brenier's theorem (Theorem \ref{theo:Brenier} in Appendix \ref{app:optimaltransport}) in light of which $\nabla \psi_{f^*}^{\star}$ is the inverse map of $f^* = \nabla \psi_{f^*}$. 
\vskip1.5ex
\noindent Moreover, the squared 2-Wasserstein distance between $\wh{\nu}$ and $\mu_n$, i.e., $ \dW_2^2(\wh{\nu}, \mu_n)$, can be expressed as 
\begin{align}
\su \sum_{j = 1}^p \nnorm{\wh{\theta}_j - X_i}_2^2 \wh{\Gamma}_{ij} &= \sum_{j = 1}^p \wh{\alpha}_j \nnorm{\wh{\theta}_j}_2^2 + \frac{1}{n} \su \nnorm{X_i}_2^2 - 2 \su \sum_{j = 1}^p \nscp{\wh{\theta}_j}{X_i} \wh{\Gamma}_{ij}  \notag \\
&= \int \nnorm{\theta}_2^2 \;d\wh{\nu}(\theta) + \int \nnorm{x}_2^2 \;d\mu_n(x) - \frac{2}{n} \su \scp{X_i}{\sum_{j = 1}^p n \wh{\Gamma}_{ij} \wh{\theta}_j} \notag \\
&= \int \nnorm{\theta}_2^2 \;d\wh{\nu}(\theta) + \int \nnorm{x}_2^2 \;d\mu_n(x) - 2 \int x^{\T} \wh{f}(x) \; d\mu_n(x) \label{eq:main_aux1}. 
\end{align}
Similarly, 
\begin{equation}\label{eq:main_aux2}
\dW_2^2(\nu_n^*, \mu_n) =\int \nnorm{\theta}_2^2 \;d\nu_n^*(\theta) + \int \nnorm{x}_2^2 \;d\mu_n(x) - 2 \int x^{\T} f^*(x) \;d\mu_n(x)
\end{equation}
where we note that $f^*$ is the optimal transport map from $\nu_n^*$ to $\mu_n$ (as $f^* \#\mu_n = \nu_n^*$ and $f^*$ is the gradient of a convex function).  
Combining \eqref{eq:main_lower_bound}, \eqref{eq:main_aux1}, \eqref{eq:main_aux2}, we obtain that 
\begin{align}
\int \nnorm{\wh{f}(x) - f^*(x)}_2^2 \; d\mu_n(x) &\leq L \Big[ \dW_2^2(\wh{\nu}, \mu_n) - \dW_2^2(\nu_n^*, \mu_n) + 2 \int \psi_{f^*}^{\star}(\theta) \;d (\wh{\nu} -\nu_n^*) (\theta) \notag \\
&\qquad + \int \nnorm{\theta}_2^2 \; d(\nu_n^*- \wh{\nu})(\theta) \Big]. \label{eq:main_intermediate}
\end{align}
Let $\wh{\eta}$ be an optimal coupling between $\nu_n^*$ and $\wh{\nu}$, and let further $\eta = (\nabla \psi_{f^*}^{\star}, \textsf{id}) \# \wh{\eta}$ be the push-forward (cf.~Definition \ref{def:pushforward}) of the coupling $\wh{\eta}$ under the transformation that pushes forward 
its two marginals to $\nabla \psi_{f^*}^{\star} \# \nu_n^* = \mu_n$ and $\textsf{id} \# \wh{\nu} = \wh{\nu}$, where we have 
used that $\nabla \psi_{f^*}^{\star}(\theta_i^*) = X_i$, $1 \leq i \leq n$, by Brenier's theorem, with $\textsf{id}$ denoting
the identity map. 

\noindent Accordingly, by the definition of the 2-Wasserstein distance in terms of optimal couplings (cf.~Appendix \ref{app:optimaltransport}), we obtain that 
\begin{equation*}
\dW_2^2(\mu_n, \wh{\nu}) \leq \int \nnorm{x - \theta}_2^2 \;d\eta(x, \theta) = 
\int \nnorm{\nabla \psi_{f^*}^{\star}(\zeta) - \theta}_2^2 \;d\wh{\eta}(\zeta, \theta). \label{eq:remainder_term_0}
\end{equation*}
Adding and subtracting $\zeta$ inside the norm on the right hand side and expanding the square, it follows that 
\begin{align}
\dW_2^2(\mu_n, \wh{\nu}) &\leq \int \nnorm{\nabla \psi_{f^*}^{\star}(\zeta) - \zeta}_2^2 \, d\nu_n^*(\zeta) + 
\int \nnorm{\theta - \zeta}_2^2 \, d\wh{\eta}(\zeta, \theta) + 2 \int \nscp{\nabla \psi_{f^*}^{\star}(\zeta) - \zeta}{\zeta - \theta} \, d\wh{\eta}(\zeta, \theta) \notag \\
&= \dW_2^2(\nu_n^*, \mu_n) + \dW_2^2(\nu_n^*, \wh{\nu}) + 2 \int \nscp{\nabla \psi_{f^*}^{\star}(\zeta) - \zeta}{\zeta - \theta} \, d\wh{\eta}(\zeta, \theta) \label{eq:remainder_term_bound1},   
\end{align}    
where we have used that $\psi_{f^*}^{\star}$ is the optimal transport map pushing forward $\nu_n^*$ to $\mu_n$, the definition 
of the 2-Wasserstein distance in terms of optimal transport and optimal couplings, and the definition of $\wh{\eta}$
as optimal coupling between $\nu_n^*$ and $\wh{\nu}$. 

In order to bound the rightmost term in the preceding display, we invoke {\bfseries(A1)} which implies \cite{Kakade2009} that the function $\psi_{f^*}^{\star}$ is $(1 / \lambda)$-smooth in the sense of \eqref{eq:Lsmooth}. This yields 
\begin{align}
2 \int \nscp{\nabla \psi_{f^*}^{\star}(\zeta)}{\zeta - \theta} \, d\wh{\eta}(\zeta, \theta) &\leq 
2 \int \left\{\psi_{f^*}^{\star}(\zeta) - \psi_{f^*}^{\star}(\theta) + \frac{1}{2 \lambda} \nnorm{\zeta - \theta}_2^2 \right\} \, d\wh{\eta}(\zeta, \theta) \notag \\
&= 2 \int \psi_{f^*}^{\star}(\zeta) \, d\nu_n^*(\zeta) - 2 \int \psi_{f^*}^{\star}(\theta) \, d\wh{\nu}(\theta) + \frac{1}{ \lambda} \dW_2^2(\nu_n^*, \wh{\nu}), \label{eq:remainder_term_bound2}
\end{align}
using the same argument as for the preceding display to obtain the rightmost term. 

Finally, we note that 
\begin{align}
2 \int \nscp{-\zeta}{\zeta - \theta} \, d\wh{\eta}(\zeta, \theta) &= \int \left\{ \nnorm{\theta}_2^2  - \nnorm{\theta - \zeta}_2^2 - \nnorm{\zeta}_2^2 \right \} \,d \wh{\eta}(\zeta, \theta) \notag \\
&= \int \nnorm{\theta}_2^2 \, d\wh{\nu}(\theta) - \int \nnorm{\zeta}_2^2 \, d\nu_n^*(\zeta) - \dW_2^2(\nu_n^*, \wh{\nu}). \label{eq:remainder_term_bound3}
\end{align}
Combining \eqref{eq:remainder_term_bound1}, \eqref{eq:remainder_term_bound2}, and \eqref{eq:remainder_term_bound3}, we obtain that 
\begin{align*}
\dW_2^2(\mu_n, \wh{\nu}) &\leq  \dW_2^2(\nu_n^*, \mu_n) + \frac{1}{\lambda} \dW_2^2(\nu_n^*, \wh{\nu})  +  2 \int \psi_{f^*}^{\star}(\zeta) \, d\nu_n^*(\zeta) - 2 \int \psi_{f^*}^{\star}(\theta) \, d\wh{\nu}(\theta) \\
&\qquad \quad + \int \nnorm{\theta}_2^2 \, d\wh{\nu}(\theta) - \int \nnorm{\zeta}_2^2 \, d\nu_n^*(\zeta). 
\end{align*}
Substituting this bound back into \eqref{eq:main_intermediate}, we observe that all but the term $\frac{L}{\lambda} \dW_2^2(\nu_n^*,  \wh{\nu})$ cancel, yielding the assertion of the theorem.  
\end{bew}

\subsection{Proof of Theorem~\ref{theo:kantorovich_unlinked}}\label{sec:Proof-Thm-6}
\begin{bew} We first note that the argument in the previous proof continues to apply with $\nu_n^* = \frac{1}{n} \su \delta_{f^*(X_i)}$, which yields 
\begin{equation*}
\frac{1}{n} \su \nnorm{\wh{f}(X_i) - f^*(X_i)}_2^2 \leq \frac{L}{\lambda} \dW_2^2(\wh{\nu}, \nu_n^*) 
\end{equation*}
We then use the triangle inequality 
\begin{equation*}
\dW_2(\wh{\nu}, \nu_n^*) \leq \dW_2(\wh{\nu}, \nu_{m}^*) + \dW_2(\nu_{m}^*, \nu_n^*) \leq 
 \dW_2(\wh{\nu}, \nu_{m}^*) + \dW_2(\nu_n^*, \nu) +  \dW_2(\nu_{m}^*, \nu),
\end{equation*}
and accordingly 
\begin{equation*}
\dW_2^2(\wh{\nu}, \nu_n^*) \leq 2 \dW_2^2(\wh{\nu}, \nu_{m}^*) + 4 (\dW_2^2(\nu_n^*, \nu) +  \dW_2^2(\nu_{m}^*, \nu)).  
\end{equation*}
\noindent The proof of the result now follows by invoking Lemma \ref{lem:Wasserstein_concentration} with the choices $t = \sqrt{\log n / n} \vee n^{-2/d} (\log n)^{2/d}$ and $t = \sqrt{\log m / m} \vee m^{-2/d} (\log m)^{2/d}$ to control the second and the third term of the above display, respectively, with the stated probability; for the second term, we use that $\{ f^*(X_i)\}_{i = 1}^n \overset{\text{i.i.d.}}{\sim} \nu$ since
$f^*$ pushes forward $\mu$ to $\nu$ (cf.~Definition \ref{def:pushforward} and Theorem \ref{theo:Brenier}). 
\end{bew}

\subsection{Proof of Theorem~\ref{theo:wasserstein_deconvolution}}\label{sec:Proof-Thm-5}
\begin{bew}
For a Lebesgue density $h$ on $\R^d$ and $q > 0$, let $\textsf{M}_{h}^q := \int \nnorm{x}_2^q \, h(x) \, dx$ denote the $q$-th moment
associated with $h$. 

Let $s > k$ be arbitrary and let $K: \R^d \rightarrow (0,\infty)$ be a symmetric PDF such that $\textsf{M}_{K}^s \coloneq \int_{\R^d} \nnorm{x}_2^s \, K(x) \, dx < \infty$  and such that its Fourier transform $\wh{K}$ is continuous with support contained in $[-1,1]^d$, and for $\delta > 0$, let $K_{\delta}(\cdot) := \frac{1}{\delta^d} K(\cdot / \delta)$. By the triangle inequality, we have 
\begin{equation}\label{eq:wasserstein_triangle}
\dW_k^k(\nu_n^*, \wh{\nu}) \leq 2^{2(k-1)} \left\{  \dW_k^k(\nu_n^* , \nu_n^* \star K_{\delta}) + \dW_k^k(\wh{\nu}, \wh{\nu} \star K_{\delta} ) + \dW_k^k(\nu_n^* \star K_{\delta}, \wh{\nu} \star K_{\delta}) \right\}.  
\end{equation}
The first two terms inside the curly brackets are of order $O(\delta^k)$. To see this, consider couplings 
defined by the pairs of random variables $(X,X+\eps)$ and $(\wh{X}, \wh{X} + \eps)$ with $X \sim \nu_n^*$, $\wh{X} \sim \wh{\nu}$, and $\eps$ (independent of ${X}$ and $\wh{X}$)  distributed according to the PDF $K_{\delta}$, and note that
$\E[\nnorm{X - (X + \eps)}^k_2] = \E[\nnorm{\wh{X} - (\wh{X} + \eps)}^k_2] = O(\delta^k)$. 
\vskip1ex
\noindent In the sequel, the third term 
$\dW_k^k(\nu_n^* \star K_{\delta}, \wh{\nu} \star K_{\delta})$ will be controlled. By Lemma \ref{lem:villani} in Appendix \ref{app:misc}, we have
\begin{equation}\label{eq:use_of_villani}
\dW_k^k(\nu_n^* \star K_{\delta}, \wh{\nu} \star K_{\delta}) \leq 2^{k-1} \int_{\R^d} \nnorm{x}_2^k \;\,d|\nu_n^* \star K_{\delta} - \wh{\nu} \star K_{\delta}|(x).
\end{equation}
Next, we aim to bound the right hand side of \eqref{eq:use_of_villani} by invoking Lemma \ref{lem:nguyen} in Appendix
\ref{app:misc}. For this purpose, we need to establish first that the $s$-th moment of $\nu_n^* \star K_{\delta}$ and 
$\wh{\nu} \star K_{\delta}$ are finite. For $\nu_n^* \star K_{\delta}$, this follows from 
\begin{align*}
\int_{\R^d} \nnorm{x}_2^s \; d(\nu_n^* \star K_{\delta})(x) \; &=  \int_{\R^d} \int_{\R^d}\nnorm{x}_2^s \; K_{\delta}(x - \theta) \, dx \, d\nu_n^*(\theta)  \\
&=\int_{\R^d} \int_{\R^d} \nnorm{x + \theta}_2^s \; K_{\delta}(x) \, dx \, d\nu_n^*(\theta) \\
&\leq 2^{s-1} \left(\delta^s \int_{\R^d} \nnorm{x}_2^s \, K(x) \; dx + \int_{\R^d} \nnorm{\theta}_2^s \; d\nu_n^*(\theta) \right) < \infty. 
\end{align*}
Above, we have used that the $s$-th moment of $K$ is finite by construction and that the support of $\nu_n^*$ is uniformly bounded.

Showing that the $s$-th moment of $\wh{\nu} \star K_{\delta}$ is finite is more intricate since the support of $\wh{\nu}$ cannot be assumed to be uniformly bounded a priori. A careful truncation argument that relies on tail bounds and the Hellinger rates of the NPMLE is presented in Appendix \ref{app:truncation}. Specifically, consider the two events $\mc{H}$ and $\mc{M}$ given by  
\begin{align}
&\mc{H} \coloneq \left \{ \dH(\textsf{f}_n, \wh{\textsf{f}}_n) \leq \, C(d, \sigma, B) \, \frac{(\log n)^{(d+1)/2}}{\sqrt{n}}  \right\},\label{eq:eventHcal}\\ &\mc{M} \coloneq \left \{ \int_{\R^d} \nnorm{x}_2^s \, d\wh{\nu}(x) \leq C'(d, \sigma, s, B) \frac{(\log n)^{(s + d + 1)/2}}{\sqrt{n}} \leq C''(d, \sigma, s, B) \right\} \notag.       
\end{align}
We bound the probability of the complementary event of $\mathcal{H} \cap \mathcal{M}$ as follows: 
\begin{equation*}
\p(\mc{H}^{\textsf{c}} \cup \mc{M}^{\textsf{c}}) \leq \p(\mathcal{H}^{\textsf{c}}) + \p(\mathcal{M}^{\textsf{c}}) \leq 2 \p(\mathcal{H}^{\textsf{c}}) + \p(\mathcal{M}^{\textsf{c}}|\mathcal{H}) \p(\mathcal{H}).  
\end{equation*}
By Lemma \ref{lem:truncation}, we have 
\begin{equation*}
\p(\mathcal{M}^{\textsf{c}}|\mathcal{H}) \leq (1/n)  \big/ \p(\mc{H}).  
\end{equation*}
Substituting this into the previous display yields that 
\begin{equation*}
\p(\mathcal{M}) \geq \p(\mathcal{M} \cap \mathcal{H}) \geq 1 - 2 \p(\mathcal{H}^{\textsf{c}}) - 1/n \geq 1 - 5/n, 
\end{equation*}
where the last inequality is obtained by using the definition of the event $\mathcal{H}$ and Lemma~\ref{lem:saha} 
with the choice $t = 1$.


\vskip1ex 
With these arguments in place, we apply Lemma \ref{lem:nguyen} to the right hand side of \eqref{eq:use_of_villani}, which yields 
\begin{align}
\dW_k^k(\nu_n^* \star K_{\delta}, \wh{\nu} \star K_{\delta}) &\leq C(s,k,d)   \, (\textsf{M}_{\nu_n^* \star K_{\delta}}^s + \textsf{M}_{\wh{\nu} \star K_{\delta}}^s)^{\frac{(s - t)d  + k}{s(d + 2s)}} \, \nnorm{\nu_n^* \star K_{\delta} - \wh{\nu} \star K_{\delta}}_{L_2}^{\frac{2(s - k)}{d + 2s}}  \notag \\
&\leq C'(s,k,d, K) \, \nnorm{\nu_n^* \star K_{\delta} - \wh{\nu} \star K_{\delta}}_{L_2}^{\frac{2(s - k)}{d + 2s}} \label{eq:wasserstein_difficult_intermediate}
\end{align}
where $C$ and $C'$ are positive quantities depending only on the quantities in parentheses, assuming for now that $\delta$ is uniformly bounded from above.  

Consider the Fourier transforms $\wh{K}_{\delta}$ and $\wh{\mphi_{\sigma}}$ of $K_{\delta}$ and $\mphi_{\sigma}$, respectively, and let $g_{\delta} := \wt{\wh{K}_{\delta} / \wh{\mphi_{\sigma}}}$ be the inverse Fourier transform
of $\wh{K}_{\delta} / \wh{\mphi_{\sigma}}$. Note that by construction, $\wh{K}_{\delta}$ has bounded support and hence
so has $g_{\delta}$ whose Fourier transform is therefore given by $\wh{g}_{\delta} = \wh{K}_{\delta} / \wh{\mphi_{\sigma}}$
according to the Fourier inversion theorem. Furthermore, by the convolution theorem we have $\wh{K}_{\delta} = \wh{g}_{\delta} \cdot \wh{\mphi_{\sigma}} = \wh{g_{\delta} \star \mphi_{\sigma}}$ and in turn $K_{\delta} = g_{\delta} \star \mphi_{\sigma}$ (cf.~Appendix \ref{app:fourier}). It follows that $K_{\delta} \star \nu_n^* = g_{\delta} \star \textsf{f}_n$ and 
 $K_{\delta} \star \wh{\nu} = g_{\delta} \star \wh{\textsf{f}}_n$. This yields the following with regard to the term in 
 \eqref{eq:wasserstein_difficult_intermediate}:
 \begin{align}
  \nnorm{\nu_n^* \star K_{\delta} - \wh{\nu} \star K_{\delta}}_{L_2} &=  \nnorm{g_{\delta} \star (\wh{\textsf{f}}_n - \textsf{f}_n)}_{L_2} \notag \\
                                                                 &\leq \nnorm{\wh{\textsf{f}}_n - \textsf{f}_n}_{L_1} \nnorm{g_{\delta}}_{L_2} \notag \\
                                                                 &\leq 2\dH(\wh{\textsf{f}}_n, \textsf{f}_n )\nnorm{g_{\delta}}_{L_2} \label{eq:fourier_tricks_result}
 \end{align}
by the distributivity of convolution, Young's inequality, and the fact that $\nnorm{\wh{\textsf{f}}_n - \textsf{f}_n}_{L_1} = 2 \dtv(\wh{\textsf{f}}_n , \textsf{f}_n) \leq 2 \dH(\wh{\textsf{f}}_n, \textsf{f}_n )$. It remains to upper bound $\nnorm{g_{\delta}}_{L_2}$. The Plancherel theorem (cf.~Appendix \ref{app:fourier}) yields 
\begin{equation*}
\nnorm{g_{\delta}}_{L_2}^2 = \frac{1}{(2 \pi)^d} \int_{\R^d} \frac{\wh{K}_{\delta}^2(\omega)}{\wh{\mphi_{\sigma}}^2(\omega)} \; d\omega =  \frac{1}{(2 \pi)^d} \int_{\R^d} \frac{\wh{K}^2(\omega \delta)}{\wh{\mphi_{\sigma}}^2(\omega)} \; d\omega \leq C(K)  \frac{1}{(2 \pi)^d} \int_{[-\delta^{-1}, \delta^{-1}]^d} \frac{1}{\wh{\mphi_{\sigma}}^2(\omega)} \; d\omega. 
\end{equation*}
For the second equality, we have used the definition of the Fourier transformation as integral transform and have
a made a change of variables. For the above inequality, we have used that $\wh{K}$ is supported on $[-1,1]^d$ with essential supremum bounded by a positive constant $C(K)$. It is well known that
\begin{equation*}
\wh{\mphi_{\sigma}}(\omega) = \frac{1}{(2\pi)^d} \sigma^{d/2} \exp\left(-\frac{\sigma^2}{2} \nnorm{\omega}_2^2 \right).   
\end{equation*}
Combining this with the previous display yields
\begin{align}
\nnorm{g_{\delta}}_{L_2}^2 &\leq C(K, d, \sigma) \int_{[-\delta^{-1}, \delta^{-1}]^d} \exp\left(\sigma^2 \nnorm{\omega}_2^2 \right) \; d\omega \notag \\ 
&\leq C(K, d, \sigma) \, (2/\delta)^d \, \exp\left(\sigma^2 d \delta^{-2} \right) \notag \\
&\leq C(K, d, \sigma) \,  \, \exp\left(2 \sigma^2 d \delta^{-2} \right). \label{eq:fourier_bound}
\end{align}
Combining \eqref{eq:wasserstein_triangle}, \eqref{eq:wasserstein_difficult_intermediate}, \eqref{eq:fourier_tricks_result} and \eqref{eq:fourier_bound} then yields 
\begin{align}
\dW_k^k(\nu_n^*, \wh{\nu}) &\leq C(s,k,d, K, \sigma) \left\{ \delta^k + \dH(\wh{\textsf{f}}_n, \textsf{f}_n)^{\frac{2(s - k)}{d + 2s}} \, \exp\left(\frac{2(s-k)}{d + 2s} \sigma^2 d \delta^{-2} \right) \right \} \notag \\
&\leq C'(s,k,d, K, \sigma) \left\{ \left(\frac{d \sigma^2}{\log \left( \frac{1}{\dH(\wh{\textsf{f}}_n, \textsf{f}_n)} \right)} \right)^{k/2}  + \dH(\wh{\textsf{f}}_n, \textsf{f}_n)^{\frac{(s - k)}{d + 2s}} \right \} \label{eq:deconvolution_final}
\end{align}
by choosing $\delta^{-2} = -\frac{1}{2 d \sigma^2} \log \dH(\wh{\textsf{f}}_n, \textsf{f}_n)$. Conditional on the event event $\mc{H}$ in \eqref{eq:eventHcal} 
and the stated condition on the sample size $n$, $\dH(\wh{\textsf{f}}_n, \textsf{f}_n) < 1$, and thus the above choice of $\delta$ is valid in the sense that $\delta > 0$.
Substituting the bound on $\dH(\wh{\textsf{f}}_n, \textsf{f}_n)$ under event $\mc{H}$ in \eqref{eq:eventHcal} into \eqref{eq:deconvolution_final}, absorbing terms depending only on $s, k, d, \sigma$ and $B$ into a constant, and absorbing the second summand inside the curly brackets in \eqref{eq:deconvolution_final} into the first summand at the expense of modified constants yields the assertion (the dependence on the function $K$ can be absorbed into the dependence on $d$).    
\end{bew}

\begin{rem}\label{rem:polynomial_decay} Conditional on an event of the form \eqref{eq:eventHcal}, 
a bound on $\text{\emph{$\dW_k^k$}}(\nu_n^*, \wh{\nu})$ can be obtained for other than Gaussian errors distributions. 
Faster rates are possible if the Fourier transformation $\wh{\mphi_{\sigma}}$ exhibits less rapid
decay. Under the polynomial decay condition
\begin{equation*}
\wh{\mphi_{\sigma}}(\omega) \geq C(\sigma, d) \nnorm{\omega}_2^{-\alpha}, 
\end{equation*}
for some $\alpha > 0$, we obtain in place of \eqref{eq:fourier_bound} that
\begin{equation*}
\nnorm{g_{\delta}}_{L_2}^2 \leq C(K, d, \sigma) \delta^{-(2\alpha + d)},
\end{equation*}
and thus
\begin{align*}
\emph{\text{$\dW_k^k$}}(\nu_n^*, \wh{\nu}) &\leq C(s,k,d, K, \sigma)  \left\{ \delta^k + \emph{\text{$\dH$}}(\emph{\text{$\wh{\textsf{f}}$}}_n, \emph{$\text{\textsf{f}}$}_n)^{\frac{2(s - k)}{d + 2s}} \, \delta^{-(2\alpha + d) (s - k) / (d + 2s)}  \right \} \\   
&\leq C(s,k,d, K, \sigma) \emph{\text{$\dH$}}(\emph{\text{$\wh{\textsf{f}}$}}_n, \emph{$\text{\textsf{f}}$}_n)^{\frac{2k(s - k)}{k(d + 2s) + (2\alpha + d)(s-k)}} = C(s,k,d, K, \sigma) \emph{\text{$\dH$}}(\emph{\text{$\wh{\textsf{f}}$}}_n, \emph{$\text{\textsf{f}}$}_n)^{c(k,s,d,\alpha)},
\end{align*}
where $0 < c(k,s,d,\alpha) = \frac{2k(s - k)}{k(d + 2s) + (2\alpha + d)(s-k)} < 1$. 
\end{rem} 

%
%


\section{Rates of convergence of the NPMLE for Gaussian location mixtures}\label{app:hellinger_npmle}
\noindent  Rates of convergence of the NPMLE for Gaussian location mixtures in Hellinger distance for general dimension $d \geq 1$ were established in the paper \cite{Saha2020}, generalizing earlier result in \cite{Zhang2009} concerning the case $d = 1$.  

\begin{lemma}[Theorem 2.1 and Corollary 2.2 in \cite{Saha2020}]\label{lem:saha} 
Let $\wh{\text{\emph{\textsf{f}}}}_n$ denote the NPMLE \eqref{eq:Kiefer_Wolfowitz} with $\varphi(z) := (2\pi)^{-d/2} \exp(-\nnorm{z}_2^2)$
given $\{ Y_i \}_{i = 1}^n \overset{\text{\emph{i.i.d.}}}{\sim} \varphi_{\sigma} \star \nu_n^*$ with 
$\nu_n^* := \frac{1}{n} \su \delta_{\theta_i^*}$ such that $\{ \theta_i^* \}_{i = 1}^n$ is contained in a Euclidean ball of radius $B$ centered at the origin. Then for all $t \geq 1$
\begin{equation*}
\p(\emph{\dH}(\text{\emph{\textsf{f}}}_n, \wh{\text{\emph{\textsf{f}}}}_n) > r_n t)  \leq 2 n^{-t^2}, \qquad r_n \asymp \frac{(\log n)^{(d+1)/2}}{\sqrt{n}},
\end{equation*}
where $\asymp$ involves hidden constants depending (only) on $d$, $B$, and $\sigma$. 
\end{lemma}

\section{Truncation argument}\label{app:truncation}

\begin{lemma}\label{lem:truncation}
Consider the setup of Lemma \ref{lem:saha}, and denote by $\wh{\nu}$ the mixing measure associated with the NPMLE. 
Consider the event $\mc{E} = \{ \emph{\dH}(\text{\emph{\textsf{f}}}_n, \wh{\text{\emph{\textsf{f}}}}_n) \leq \overline{h} \}$ for some $\overline{h} > 0$. Conditional on $\mc{E}$, for any $s \geq 1$, we have $\int_{\R^d} \nnorm{x}_2^s \;d\wh{\nu}(x) \leq C_1(d, \sigma, s) (\log n)^{s/2} \cdot \overline{h} + C_2(d,s,\sigma, B)$ with probability at least 
$1 - 1/\{ n \cdot \p(\mc{E}) \}$, where $C_1$ and $C_2$ are positive constants depending only on the quantities in the parentheses. 
\end{lemma} 

\begin{proof} 
We first note that in order to show that the $s$-th moment of $\wh{\nu}$ is finite it suffices to show that the  $s$-th moment of $\wh{\nu} \star \mphi_{\sigma}$ is finite. In fact, consider random variables
$\wh{X}$, $\eps$ such that $\wh{X} \sim \wh{\nu}$ and $\eps \sim \mphi_{\sigma}$ where $\wh{X}$ and $\eps$ are independent. We then have
\begin{equation*}
\E[\nnorm{\wh{X}}_2^s] = \E[\nnorm{\wh{X} - \eps + \eps}_2^s] \leq 2^{s-1} (\E[\nnorm{\wh{X} + \eps}_2^s] + \E[\nnorm{\eps}_2^s]).
\end{equation*}
In order to show that the $s$-th moment of $\wh{\nu} \star \mphi_{\sigma}$ is finite, we will use Lemma \ref{lem:saha} regarding the Hellinger rates of convergence between 
$\nu_n^* \star \mphi_{\sigma}$ and $\wh{\nu} \star \mphi_{\sigma}$ and the fact that the support of $\nu_n^*$ is contained in an Euclidean ball of radius $B$ by assumption. 

First note that according to established properties of the NPMLE (e.g., \cite{Lindsay1983, Koenker2014}), $\wh{\nu}$ is an atomic measure, i.e., it can be written as $\wh{\nu} = \sum_{j = 1}^p \wh{\alpha}_j \delta_{\wh{\theta}_j}$
for non-negative coefficients $\{ \wh{\alpha}_j \}_{j = 1}^p \subset \R_+$ summing to one and atoms 
$\{ \wh{\theta}_j  \}_{j = 1}^p \subset \R^d$. Let 
\begin{equation}\label{eq:lem_trunc_quant}
\wh{B} = \max_{1 \leq j \leq p} \nnorm{\wh{\theta}_j}_2, \quad    
\rho = \wh{B} + \sigma R, \quad \mc{R} = \mathbb{B}_2^d(\rho), \quad \mc{R}_0 = \bigcup_{j = 1}^p ( \mathbb{B}_2^d(\sigma R) + \wh{\theta}_j)
\end{equation}
for $R > 0$ to be chosen later. Observe that $\mc{R}_0 \subset \mc{R}$ and hence $\mc{R}^{\mathsf{c}} \subset \mc{R}_0^{\mathsf{c}}$. We have
\begin{align}
\int_{\R^d} \nnorm{x}_2^s \, d (\mphi_{\sigma} \star \wh{\nu})(x) &= \int_{\mc{R}} \nnorm{x}_2^s \, d(\mphi_{\sigma} \star \wh{\nu})(x)  +  \int_{\mc{R}^{\mathsf{c}}} \nnorm{x}_2^s \, d (\mphi_{\sigma} \star \wh{\nu})(x) \notag \\
&\leq \int_{\mc{R}} \nnorm{x}_2^s \, d (\mphi_{\sigma} \star \nu_n^*)(x) + \int_{\mc{R}} \nnorm{x}_2^s\, d |\mphi_{\sigma} \star \wh{\nu} - \mphi_{\sigma} \star \nu_n^*|(x) \notag \\
&\qquad +\int_{\mc{R}^{\mathsf{c}}} \nnorm{x}_2^s \, d(\mphi_{\sigma} \star \wh{\nu})(x) \notag \\
&\hspace{-2ex}\leq C_1(d, s, B, \sigma) + 2 \rho^s \underbrace{\dH(\mphi_{\sigma} \star \wh{\nu}, \mphi_{\sigma} \star \nu_n^*)}_{=\dH(\wh{\textsf{f}}_n, \textsf{f}_n) \leq \overline{h} \; \text{on} \; \mc{E}} + \int_{\mc{R}^{\mathsf{c}}} \nnorm{x}_2^s \, d(\mphi_{\sigma} \star \wh{\nu})(x) \label{eq:lem_trunc_decomp}
\end{align}  
for some constant $C_1 > 0$ depending only on the quantities given in parentheses. In order to obtain the bound on the middle term, we use that the integral is over $\mathbb{B}_2^d(\rho)$ and that the total variation distance can be bounded by twice the Hellinger distance. We now turn our attention to the third term in \eqref{eq:lem_trunc_decomp}. We have 
\begin{align}
\int_{\mc{R}^{\mathsf{c}}} \nnorm{x}_2^s (\mphi_{\sigma} \star \wh{\nu})(x) \; dx &\leq \int_{\mc{R}_0^{\mathsf{c}}} \nnorm{x}_2^s (\mphi_{\sigma} \star \wh{\nu})(x) \; dx \notag \\
                                                               &=\sum_{j = 1}^p \wh{\alpha}_j  \int_{\sigma^{-1 }(\mc{R}_0^{\mathsf{c}} - \wh{\theta}_j)} \nnorm{\sigma z + \wh{\theta}_j}_2^s \,  \mphi(z) \; dz  \notag \\
  &\leq \sum_{j= 1}^p \wh{\alpha}_j 2^{s-1} \left\{ \sigma^{s} \int_{\R^d} \nnorm{z}_2^s \, \mphi(z) \; dz + \nnorm{\wh{\theta}_j}_2^s \int_{\sigma^{-1 }(\mc{R}_0^{\mathsf{c}} - \wh{\theta}_j)}  \mphi(z) \, dz \right \} \notag \\
&\leq 2^{s-1} \left\{ \sigma^{s} \int_{\R^d} \nnorm{z}_2^s \, \mphi(z) \; dz + \max_{1 \leq j \leq p} \nnorm{\wh{\theta}_j}_2^s \int_{\R^d \setminus \mathbb{B}_2^d(R)}  \mphi(z) \, dz \right \}   \notag \\ 
&\leq C_2(d, s, \sigma) +  \wh{B}^s \p(\nnorm{Z}_2 \geq R), \quad Z \sim N(0, I_d) \notag \\
&\leq  C_2(d, s, \sigma) +  (\wh{B}/n)^s \label{eq:lem_trunc_remainder}
\end{align}
by choosing $R = \sqrt{2 s \log n}$, as follows from standard concentration of measure results. In the third inequality from the bottom, we have used that for any $j$
\begin{equation*}
\sigma^{-1} (\mc{R}_0^{\mathsf{c}} - \wh{\theta}_j) = \sigma^{-1} \left(\bigcap_{j = 1}^p \{\mathbb{B}_2^d(\sigma R) + \wh{\theta}_k \}^{\mathsf{c}} - \wh{\theta}_j \right) \subseteq \sigma^{-1}  \left[ \{\mathbb{B}_2^d(\sigma R) + \wh{\theta}_j \}^{\mathsf{c}} - \wh{\theta}_j \right] = \R^d \setminus \mathbb{B}_2^d(R).
\end{equation*}
In order to wrap up this proof, it remains to control $\wh{B}$ (with high probability). With the same concentration result as used before in combination with the union bound, one shows that 
\begin{align}
\p(\wh{B} \geq B + \sigma (\sqrt{d} + 2 \sqrt{\log n})) &\leq \p \left(\max_{1 \leq i \leq n} \nnorm{y_i}_2 \geq B + \sigma (\sqrt{d} + 2 \sqrt{\log n}) \right) \label{eq:first_inequality_Bhat}\\
&\leq \p \left(\max_{1 \leq i \leq n} \nnorm{\theta_i^*}_2  + \max_{1 \leq i \leq n} \nnorm{\epsilon_i}_2 \geq B + \sigma (\sqrt{d} + 2 \sqrt{\log n} \right) \notag \\
&= \p \left(\max_{1 \leq i \leq n} \nnorm{\epsilon_i}_2 \geq \sigma (\sqrt{d} + 2 \sqrt{\log n} \right) \leq 1/n. \notag
\end{align}
Let $\mc{A}$ denote the event inside $\p(\ldots)$ in the last line, and observe that
$\p(\mc{A}|\mc{E}) \leq \p(\mc{A}) / \p(\mc{E})$. Combining this with \eqref{eq:lem_trunc_decomp}, \eqref{eq:lem_trunc_remainder}, and the above choice of $R$ then yields the assertion.   

Note that in the first inequality \eqref{eq:first_inequality_Bhat}, we have used that 
$\wh{B} = \max_{1 \leq j \leq p} \nnorm{\wh{\theta}_j}_2 \leq \max_{1 \leq i \leq n} \nnorm{y_i}_2 \invcoloneq Q$ since $\varphi(z) = \varphi(\nnorm{z}_2)$ is decreasing in $\nnorm{z}_2$. Accordingly, we have 
\begin{equation*}
\sum_{i = 1}^n -\log\left(\sum_{j = 1}^p \wh{\alpha}_j  \varphi_{\sigma} \left(y_i - P_{\mathbb{B}_2^d(Q)}(\wh{\theta}_j) \right) \right) \leq \sum_{i = 1}^n -\log\left(\sum_{j = 1}^p \wh{\alpha}_j  \varphi_{\sigma} \left(y_i - \wh{\theta}_j \right) \right),
\end{equation*}
where $P$ denotes the Euclidean projection, which is a non-expansive operator for convex sets. The latter property implies that for
$1 \leq i \leq n$ and $1 \leq j \leq p$, it holds that
\begin{equation*}
\nnorm{y_i - P_{\mathbb{B}_2^d(Q)}(\wh{\theta}_j)}_2 = \nnorm{P_{\mathbb{B}_2^d(Q)}(y_i) - P_{\mathbb{B}_2^d(Q)}(\wh{\theta}_j)}_2 \leq \nnorm{y_i - \wh{\theta}_j}_2. 
\end{equation*}
\end{proof}

\begin{rem}\label{rem:truncation} Close inspection of the proof reveals that the above ``truncation" argument does not rely on specific properties of the Gaussian PDF $\varphi$
other than the following: 
\begin{itemize}
\item[(i)] $\varphi(z) = \varphi(\nnorm{z}_2)$  is decreasing in $\nnorm{z}_2$,
\item[(ii)] $\p_{Z \sim \varphi}(\nnorm{Z}_2 \geq C(d) \cdot r^{\beta}) \leq C' \exp(-c r)$ for positive constants
$c, C(d), C', \beta > 0$,
\end{itemize}
in which case, for any $s \geq 1$, it holds that $\int_{\R^d} \nnorm{x}_2^s \;d\wh{\nu}(x) \leq C_1(d, \sigma, s) \log(n)^{s \cdot \beta} \cdot \text{\emph{$\dH$}}(\varphi_{\sigma} \star \nu_n^*, \varphi_{\sigma} \star \wh{\nu}) + C_2(d, s, \sigma, B)$ with probability at least $1 - 1/n$. The above two properties are satisfied, e.g., by the density of the (multivariate) Laplace distribution and
other suitable elliptical distributions. It is not hard to verify that for the Laplace distribution (ii) holds with exponent
$\beta = 3/2$. 

\end{rem}


\section{Miscellaneous technical lemmas}\label{app:misc}
The following result for controlling the $p$-Wasserstein distance in terms of the total variation distance
can be found in \cite{Villani2009}. 
\begin{lemma}[Theorem 6.15 in \cite{Villani2009}]\label{lem:villani} Let $\mu$ and $\nu$ be two probability measures on $\R^d$. Then for any $1 \leq p < \infty$, we have
\begin{equation*}
\emph{\textsf{W}}_p^p(\mu, \nu) \leq 2^{p/q} \int_{\R^d} \nnorm{x}_2^p \;d|\mu - \nu|(x), \qquad \frac{1}{p} + \frac{1}{q} = 1.    
\end{equation*}
\end{lemma}
The next result, which is taken from \cite{Nguyen2013}, in turn bounds the right hand side of Lemma \ref{lem:villani} if $\mu$ and $\nu$ have densities. 
\begin{lemma}[Lemma 6 in \cite{Nguyen2013}]\label{lem:nguyen}
Let $f$ and $g$ be probability density functions on $\R^d$, and suppose that $\textsf{\emph{M}}_f^s := \int \nnorm{x}_2^s \, f(x) \, dx$ and 
$\textsf{\emph{M}}_g^s := \int \nnorm{x}_2^s \, g(x) dx $ are finite. We then have for any $0 < t < s$,
\begin{equation*}
\int_{\R^d} \nnorm{x}_2^s \; |f(x) - g(x)| \, dx \leq 4 V_d^{\frac{s - t}{d + 2s}} \, (\textsf{\emph{M}}_f^s + \textsf{\emph{M}}_g^s)^{\frac{(s - t)d  + t}{s(d + 2s)}} \, \nnorm{f - g}_{L_2}^{\frac{2(s - t)}{d + 2s}},
\end{equation*}
where $V_d := \pi^{d/2} / \Gamma(d/2 + 1)$ denotes the volume of the unit Euclidean ball in $\R^d$.  
\end{lemma}

The next result, which is a special case of Theorem 2 in \cite{Fournier2015}, yields a concentration inequality between
the squared 2-Wasserstein distance of a measure and its empirical counterpart constructed from $n$ i.i.d.~samples. 

\begin{lemma}[\cite{Fournier2015}]\label{lem:Wasserstein_concentration} Let $\{ X_i \}_{i = 1}^n \overset{\text{i.i.d.}}{\sim} \nu$, where $\nu$ is a measure in $\R^d$ with compact support. Let $\nu_n = \frac{1}{n} \su \delta_{X_i}$. We then have, for all $t > 0$,
\begin{equation*}
\p(\emph{\text{$\dW$}}_2^2(\nu_n, \nu) \geq t) \leq C \begin{cases}
\exp(-c n t^2) \quad &\text{if} \; d \leq 3, \\
\exp(-c n (t / \log(2 + 1/t))^2) \quad &\text{if} \; d = 4,\\
\exp(-c n t^{d/2}) \quad &\text{if} \; d > 4. 
\end{cases}
\end{equation*}
\end{lemma}

\noindent The following result is a key ingredient in the proof of Proposition \ref{prop:denoising_d1}. 
\begin{lemma} \label{lem:barycentric} Let $P = \su \alpha_i \delta_{x_i}$ and $Q = \sum_{j = 1}^m \beta_j \delta_{x_j'}$ be two atomic probability measures on $\{ x_i \}_{i = 1}^n \subset \R^d$ and $\{ x_j' \}_{j = 1}^m \subset \R^d$, and suppose that $\Gamma = (\Gamma_{ij})_{1 \leq i \leq n, \; 1 \leq j \leq m}$ specifies an optimal coupling between $P$ and $Q$ with respect to any cost function $c$ of the form $c(x,x') = h(\nnorm{x - x'})$, for some norm $\nnorm{\cdot}$ and $h: \R \rightarrow \R$ convex. 
Let $\wt{x}_i := \sum_{j = 1}^m \frac{\Gamma_{ij}}{\alpha_i} x_j'$, $1 \leq i \leq n$. It then holds that
\begin{equation*}
\emph{\dW}_c(P, Q) := \sum_{i=1}^n \sum_{j = 1}^m \Gamma_{ij} c(x_i, x_j') \geq \sum_{i = 1}^n \alpha_i c(x_i, \wt{x}_i). 
\end{equation*}
\end{lemma}
\begin{bew} Define $\lambda_{ij} = \frac{\Gamma_{ij}}{\alpha_i}$, and note that by construction $\sum_{j = 1}^m \lambda_{ij} = 1$, for each $i$. Furthermore, observe that $c$ is convex in either of its arguments. We hence have by Jensen's inequality that 
\begin{equation*}
\sum_{i=1}^n \sum_{j = 1}^m \Gamma_{ij} c(x_i, x_j') = \sum_{i=1}^n \alpha_i \sum_{j = 1}^m  \lambda_{ij} c(x_i, x_j') \geq
\sum_{i=1}^n \alpha_i    c\Big(x_i, \sum_{j = 1}^m \lambda_{ij} x_j' \Big) = \sum_{i=1}^n \alpha_i  c(x_i, \wt{x}_i). 
\end{equation*}
\end{bew}

\section{Notions and Results from Optimal Transport}\label{app:optimaltransport}
To make this paper self-contained, we here present notions and results from the theory of optimal transport
as far as needed for the purpose of the paper. This material or slight modifications thereof are accessible from 
popular monographs and lecture notes on the subject, e.g., \cite{COT2019, Villani2009, Villani2003, Santambrogio2015, McCann2011}. 

\begin{defn}[Push-forward]\label{def:pushforward} Let $\mu$ and $\nu$ be two Borel probability measures on measurable spaces 
$(\mc{X}, \mc{B}_{\mc{X}})$ and $(\mc{Y}, \mc{B}_{\mc{Y}})$ respectively, and let $T$ be a measurable map from 
$\mc{X}$ to $\mc{Y}$. The map $T$ is said to {\bfseries push forward} $\mu$ to $\nu$, in symbols
$T \# \mu = \nu$ if $T \# \mu(B) \equiv \mu(T^{-1}(B)) = \nu(B)$ for all $B \in \mc{B}_{\mc{Y}}$.
\end{defn}

\begin{defn}[Optimal transport problem; Monge's problem]\label{def:optimaltransport}  Let $\mu$ and $\nu$ be as in the previous definition, and let $c: \mc{X} \times \mc{Y} \rightarrow [0, \infty)$ be a measurable function (``cost function"). The optimal transport problem (Monge's problem) with $\mu$, $\nu$, and $c$ is given by  
\begin{equation*}
\inf_T \int_{\mc{X}} c(x, T(x)) \; d\mu(x) \qquad \text{subject to} \quad T\#\mu = \nu. 
\end{equation*}
Any minimizer of the above problem is called an optimal transport map. 
\end{defn}
\noindent The following optimization problem is in general a relaxation of the above problem; under certain conditions, both 
problems are equivalent. 
\begin{defn}[Kantorovich problem]\label{def:kantorovich} Let $\mu$ and $\nu$ be as in Definition \ref{def:pushforward}, and let $c$ be a cost function as in Definition~\ref{def:optimaltransport}. Let further $\Pi(\mu, \nu)$ denote
the set of all couplings between $\mu$ and $\nu$, i.e., probability measures on $\mc{X} \times \mc{Y}$ whose marginals 
equal to $\mu$ and $\nu$. The Kantorovich problem is given by the optimization problem 
\begin{equation*}
\inf_{\gamma \in \Pi(\mu, \nu)} \int_{\mc{X}} \int_{\mc{Y}}
 c(x,y) \; d\gamma(x,y). 
 \end{equation*}
 Any minimizer of the above problem is called an optimal transport plan. 
\end{defn}
\noindent For measures $\mu$ and $\nu$ on $\R^d$ with finite $k$-th moments ($k \geq 1$), i.e., $\int \nnorm{x}_2^k \, d\mu(x) < \infty$ and 
$\int \nnorm{x}_2^k \, d\nu(x) < \infty$, the $k$-Wasserstein distance between $\mu$ and $\nu$ is defined via the above Kantorovich
problem with cost function $c(x,y) = \nnorm{x - y}_2^k$, i.e., 
\begin{equation}
\dW_k(\mu, \nu) := \left( \inf_{\gamma \in \Pi(\mu, \nu)} \int \int \nnorm{x - y}_2^k \; d\gamma(x,y) \right)^{1/k}.      
\end{equation}
\vskip1.5ex
\noindent A celebrated result due to Brenier characterizes optimal transport maps in the sense of Definition 
\ref{def:optimaltransport} for $\mc{X} = \mc{Y} = \R^d$ and quadratic cost, i.e., $c(x,y) = \nnorm{x - y}_2^2$ and $\mu$ absolutely continuous with respect to the Lebesgue measure. In the sequel, we let $g^{\star}(x) := \sup_{y \in \R^d} \{ \nscp{y}{x} - g(y) \}$ denote the Legendre-Fenchel conjugate of a convex function $g: \R^d \rightarrow \R  \cup \{ +\infty \}$. 
\begin{theo}[Brenier]\label{theo:Brenier} Suppose that $\mu$ and $\nu$ are Borel probability measures on $\R^d$ with
finite second moments, and suppose further that $\mu$ is absolutely continuous with respect to the Lebesgue measure. Then the optimal transport problem has a ($\mu$-a.e.) unique minimizer $T = \nabla \psi$ for a convex function $\psi: \R^d \rightarrow \R \cup \{ +\infty \}$. Furthermore, the optimal transport problem and its Kantorovich relaxation are equivalent in the sense
that the optimal coupling in Definition \ref{def:kantorovich} is of the form $(\text{\emph{id}} \times T) \#\mu$. Moreover, if 
in addition $\nu$ is absolutely continuous, then $\nabla \psi^{\star}$ is the ($\nu$-a.e.) minimizer of the Monge problem transporting
$\nu$ to $\mu$, and it holds that $\nabla \psi^{\star} \circ \nabla \psi(x) = x$ ($\mu$-a.e.), and $\nabla \psi \circ \nabla \psi^{\star}(y) = y$ ($\nu$-a.e.).
\end{theo}
\section{Fourier transform on $\R^d$}\label{app:fourier}
For a function $g \in L^1(\R^d)$, we define its Fourier and inverse Fourier transform by  
\begin{equation*}
\wh{g}(\omega) :=  \frac{1}{(2\pi)^{d}} \int_{\R^{d}} \exp(-i \nscp{\omega}{x}) \, g(x) \,dx, \quad \qquad \wt{g}(x) := \int_{\R^d} \exp(i \nscp{\omega}{x}) \, g(\omega) \; d\omega, 
\end{equation*}
$\omega, x \in \R^d$, respectively, where $\nscp{\cdot}{\cdot}$ here refers to the inner product on $L^2(\R^d)$. According to the Fourier inversion theorem, we have $\wt{\wh{g}} = g = \wh{\wt{g}}$ if $g$ or $\wh{g}$ has bounded support. Other important properties that are used herein are as follows:
\begin{align*}
&\text{Plancherel theorem:}\quad \nscp{\wh{g}}{\wh{h}} = \frac{1}{(2\pi)^d} \nscp{g}{h}, \qquad \text{Convolution theorem:}\quad \wh{(f \star g)} = \wh{f} \cdot \wh{g},
\end{align*}
where the symbol $\star$ denotes convolution.

\end{document}